\documentclass [intlimits, 12pt]{amsart}

\allowdisplaybreaks

\usepackage[top=100pt,bottom=60pt,left=80pt,right=80pt]{geometry}

\usepackage[mathscr]{eucal}
\usepackage{hyperref} 
\usepackage{amsmath}
\usepackage{setspace}
\usepackage{amsthm}
\usepackage{amssymb}
\usepackage{amsfonts} 
\usepackage{mathrsfs}
\usepackage{amscd}
\usepackage{enumerate} 
\usepackage{color}
\usepackage{xcolor}
\usepackage{layout}
\usepackage{stmaryrd} 
\usepackage{float} 
\usepackage{longtable}
\usepackage{mathtools} 
\usepackage{graphicx}
\usepackage{txfonts} 
\usepackage{setspace} 
\usepackage{subcaption} 
\usepackage{verbatim}
\usepackage{multicol}
\usepackage[parfill]{parskip}
\usepackage{setspace} 
\allowdisplaybreaks
\numberwithin{equation}{section}
\usepackage[normalem]{ulem}

\usepackage{tikz}
\usetikzlibrary{calc,cd,patterns,decorations.pathreplacing}

\definecolor{linkblue}{rgb}{0,0,.6}
\definecolor{citered}{rgb}{.7,0,0}
\hypersetup{colorlinks =true, linkcolor=linkblue, citecolor = citered, urlcolor=linkblue}

\usepackage{enumitem}
\usepackage{nicefrac}

\usepackage{subcaption}
\newtheorem{theorem}{Theorem}[section]
\newtheorem{proposition}[theorem]{Proposition}
\newtheorem{corollary}[theorem]{Corollary}
\newtheorem{lemma}[theorem]{Lemma}

\theoremstyle{definition}
\newtheorem{definition}[theorem]{Definition}
\newtheorem{remark}[theorem]{Remark}
\newtheorem{example}[theorem]{Example}

\theoremstyle{plain}

\newcommand{\purge}[1]{} 

\newcommand{\vungoc}{V\~u Ng\d{o}c}
\newcommand{\ra}{{\rho_2}}
\newcommand{\rc}{{\rho_1}}
\newcommand{\tp}{{\tau_1^{\text{pref}}}}
\newcommand{\ssp}{{\sigma_1^{\text{pref}}}}

\newcommand{\Sp}{{S^{\text{pref}}}}

\def\epsilon{\varepsilon}
\def\phi{\varphi}

\newcommand{\al}{\alpha}
\newcommand{\be}{\beta}
\newcommand{\ga}{\gamma}

\newcommand{\ep}{\epsilon}
\newcommand{\ze}{\zeta}

\newcommand{\ka}{\kappa}
\newcommand{\lam}{\lambda}

\newcommand{\om}{\omega}

\newcommand{\Ga}{\Gamma}
\newcommand{\De}{\Delta}

\newcommand{\Lam}{\Lambda}

\def\C{{\mathbb C}}

\def\R{{\mathbb R}}
\def\mbS{{\mathbb S}} 
\def\T{{\mathbb T}}

\def\Z{{\mathbb Z}}

\newcommand{\mcF}{\mathcal F}
\newcommand{\mcG}{\mathcal G}

\newcommand{\mcI}{\mathcal I}

\newcommand{\mcN}{\mathcal N}
\newcommand{\mcO}{\mathcal O}

\newcommand{\mcT}{\mathcal T}
\newcommand{\mcU}{\mathcal U}
\newcommand{\mcV}{\mathcal V}
\newcommand{\mcW}{\mathcal W}
\newcommand{\mcX}{\mathcal X}

\newcommand{\mfI}{\mathfrak I}

\newcommand{\dee}{\mathrm{d}}

\newcommand{\beq}{\begin{equation}}
\newcommand{\eeq}{\end{equation}}
\newcommand{\beqs}{\begin{equation*}}
\newcommand{\eeqs}{\end{equation*}}

\newcommand{\nff}{{n_{\text{FF}}}}

\def\slashii#1{\setbox0=\hbox{$#1$}             
\dimen0=\wd0                                 
\setbox1=\hbox{\sl/} \dimen1=\wd1            
\ifdim\dimen0>\dimen1                        
\rlap{\hbox to \dimen0{\hfil\sl/\hfil}}   
#1                                        
\else                                        
\rlap{\hbox to \dimen1{\hfil$#1$\hfil}}   
\hbox{\sl/}                               
\fi}                                         %
\def\slashiii#1{\setbox0=\hbox{$#1$}#1\hskip-\wd0\hbox to\wd0{\hss\sl/\/\hss}}
\begin{document}

\title[The twisting index in semitoric systems]{The twisting index in semitoric systems}

\author{Jaume Alonso \quad $\&$ \quad Sonja Hohloch \quad $\&$ \quad Joseph Palmer}

\date{\today}

\begin{abstract}

Semitoric integrable systems were symplectically classified by Pelayo and \vungoc\ in 2009-2011 in terms of five invariants.
Four of these invariants were already well-understood prior to the classification, but the fifth invariant, the so-called \emph{twisting index invariant}, came as a surprise.
Intuitively, the twisting index encodes how the structure in a neighborhood of a focus-focus fiber compares to the large-scale structure of the semitoric system and it was originally defined by comparing certain momentum maps.

In the first half of the present paper, we produce several new formulations of the twisting index which give rise to dynamical, geometric, and topological interpretations. More specifically, we describe it in terms of differences of action variables, Taylor series, and homology cycles. In the second half of the paper, we compute the twisting index invariant of a specific family of systems with two focus-focus singular points (the so-called \emph{generalized coupled angular momenta}), which is the first time that the twisting index has been computed for a system with more than one focus-focus point. Moreover, we also compute the terms of the Taylor series invariant up to second order. Since the other invariants of this family were already computed, this becomes the third family of semitoric systems for which all invariants are known, after the {\it coupled spin oscillators} and the {\it coupled angular momenta}.

\end{abstract}

\maketitle

\section{Introduction}

A Hamiltonian dynamical system is known as \emph{integrable} if it admits the maximal
number of independent Poisson commuting quantities conserved by the dynamics.
Integrable systems form an important class of dynamical systems, for instance in classical
mechanics.
Furthermore, they arise often in nature, and though they enjoy certain symmetries, they
can also exhibit very complicated behavior.
Many examples of integrable systems admit a circular symmetry in the form of an $\mbS^1$-action, and such systems will be the focus
of the present paper.

The classification of toric integrable systems due to Delzant~\cite{Delzant1988} 
was extended in dimension four in 2009-2011 by Pelayo \& \vungoc~\cite{PVNinventiones,PVNacta} 
to a more general class of integrable systems called \emph{semitoric}.
While a toric integrable system in dimension four admits an underlying $\mathbb{T}^2$-action,
a semitoric integrable system in dimension four admits an underlying $\mbS^1\times\R$-action.
This allows for semitoric systems to have a certain class of singular points known as \emph{focus-focus singularities}
which cannot appear in toric integrable systems.
The original classification of Pelayo \& \vungoc~\cite{PVNinventiones,PVNacta} applied only to systems satisfying the generic condition that the momentum map of the $\mbS^1$-action contained at most one focus-focus point in each fiber (such systems are called \emph{simple semitoric systems}), but this was extended by Palmer \& Pelayo \& Tang~\cite{PPT} to a classification of all semitoric systems, simple or not. 
The classification of semitoric systems represents a substantial generalization
of the toric classification.
Semitoric systems have been an active area of research in recent years, see for instance the survey papers by Alonso $\&$ Hohloch~\cite{AH}, Pelayo~\cite{Pe}, Henriksen $\&$ Hohloch $\&$ Martynchuk~\cite{HHM}, and Gullentops $\&$ Hohloch~\cite{GH}.

The classification of simple semitoric systems is in terms of five invariants.
Four of these invariants were already conceptually well-understood: the definitions of the \emph{number of focus-focus points invariant} and the \emph{height
invariant} are straightforward, and the \emph{Taylor series} and \emph{semitoric polygon
invariants} had already been defined and geometrically interpreted in this context by \vungoc~\cite{VuNgoc03,VuNgoc07} several years earlier.
Surprisingly, these four invariants were not enough to classify simple semitoric systems.
In order to complete the classification, Pelayo and \vungoc\ had to introduce a fifth
invariant, the so-called \emph{twisting index invariant}.

The twisting index invariant is defined by comparing the fibration induced by the integrable system locally near a focus-focus singular
point to the global structure determined by a choice of the semitoric polygon invariant.
This original definition was stated in terms of comparing momentum maps.
The present paper will give several equivalent descriptions in terms of differences of action variables, Taylor series, or homology cycles, thus providing dynamical-geometrical-topological interpretations of the twisting index invariant.
Having a better understanding of the twisting index invariant of a semitoric system is a necessary step towards extending the symplectic classification to more general situations, such as almost-toric systems, hypersemitoric systems, or higher dimensional systems which admit underlying complexity-one torus actions.

Pelayo \& \vungoc\ formulated the invariants and classified simple semitoric systems, but actually computing these
invariants for explicit examples, and even for relatively elementary systems, turns out to be very involved, for instance see Alonso $\&$ Dullin $\&$ Hohloch~\cite{ADH, ADH2}, Alonso $\&$ Hohloch \cite{AH2}, Dullin~\cite{Du}, and Le Floch $\&$ Pelayo~\cite{LFPecoupledangular}.

In this paper, we do not only provide geometric formulations of the twisting index invariant, but we
compute for the first time the twisting index for an explicit system with two focus-focus points.
The twisting index has never been computed for a system with more than one focus-focus point
before, and we will see that new phenomena appear in this case compared with the case of a single
focus-focus point, which was studied in Alonso \& Dullin \& Hohloch \cite{ADH,ADH2}.
These calculations are of high computational complexity, and involve combining theoretical knowledge
of the twisting index and Taylor series invariants with computation techniques related
to solving elliptic integrals.

In summary, the main contributions of this paper are:
\begin{enumerate}
	\item For simple semitoric systems, we briefly outline the construction of the semitoric invariants
	due to Pelayo and V\~{u} Ng\d{o}c. Motivated by the development of the field during the past decade, we updated a few conventions and details compared to their original papers (see Section~\ref{ss:invariants}).
	\item We introduce dynamical, geometrical, and topological interpretations of the twisting index, which is a necessary prerequisite for any attempts to extend it to any broader class of integrable systems (see Theorem~\ref{thm:geometry}).
	\item We compute for the first time the twisting-index invariant (and on the way we compute several terms of the Taylor series invariant) for a system with two focus-focus points (see Theorems~\ref{thm:twist-intro} and~\ref{thm:Taylor-intro}).
	Note that while we only give terms of the Taylor series invariant up to second order, the computational method that we use can be employed to obtain terms up to arbitrary order, given sufficient computational power and time. 
	\item In order to compute these invariants, we investigate in general how the Taylor series invariant is affected by transformations that change the signs of the components of the momentum map (see Proposition~\ref{prop:symmetry-general}).
\end{enumerate}

The system that we consider in the current paper is actually a one-parameter family of systems, which is semitoric for all but a finite number of values of the parameter. 
In fact, since the symplectic manifold and underlying $\mbS^1$-action are independent of the parameter, it is an example of a \emph{semitoric family}, as studied by Le Floch \& Palmer~\cite{LFP-fam1,LFP-fam2}.

The underlying $\mbS^1$-action is an important feature of semitoric systems. 
Recall that a group action is called \emph{effective} if the identity is the only element that fixes the entire space.
A four dimensional symplectic manifold equipped with an effective Hamiltonian $\mbS^1$-action is called an $\mbS^1$-space, and such spaces were classified by Karshon~\cite{karshon} in terms of a labeled graph.
The relationship between the classification of $\mbS^1$-spaces and the classification of semitoric systems was studied by Hohloch \& Sabatini \& Sepe~\cite{HSS}, and in particular the twisting index is one of the invariants of semitoric systems which is not encoded in the invariants of the underlying $\mbS^1$-space; it is a purely semitoric invariant. The problem of lifting Hamiltonian $\mbS^1$-actions to integrable systems has also been considered by Hohloch \& Palmer~\cite{Hohloch2022}.
The problem of recovering the semitoric invariants from the joint spectrum of the corresponding quantum integrable system has also been studied, such as by Le Floch \& Pelayo \& V\~{u} Ng\d{o}c \cite{LFPeVN2016}, and in particular a technique for recovering the twisting index was first discovered by Le Floch \& V\~{u} Ng\d{o}c~\cite{LFVN21}.

\subsection{Main results}

The twisting index can be encoded as one integer for
each focus-focus point assigned to each representative of the semitoric
polygon invariant.
As mentioned above, it encodes the way the semilocal normal form around the focus-focus fiber lies with respect to the global integral affine structure.
The original definition was in terms of a local preferred momentum map, and in the present paper we will show that there are several equivalent formulations.

Let $m$ be a focus-focus point.
For the following theorem we use the notation:
\begin{itemize}
 \item $\kappa^\Delta$ denotes the twisting index of $m$ relative to the semitoric polygon $(\De,b,\varepsilon)$ (for the definition of $(\De,b,\varepsilon)$, see Section~\ref{sss:polygon});
 \item $\mu_\Delta$ denotes the generalized momentum map associated to the polygon $(\De,b,\varepsilon)$ (for the definition, see Section~\ref{sss:polygon});
 \item $\nu$ is the preferred momentum map near $m$, as in Pelayo \& V\~{u} Ng\d{o}c~\cite{PVNinventiones} (see Section~\ref{sss:twisting});
 \item $I_\Delta$ is the action associated to $(\Delta,b,\varepsilon)$ (see Section~\ref{sec:geometricinterpret});
 \item $\xi$ is the preferred action near $m$ (see Section~\ref{sss:twisting});
 \item $(S^\Delta)^\infty$ and $(S^\text{pref})^\infty$ are the Taylor series associated to $\Delta$ and the preferred Taylor series respectively, both are Taylor series in the variables $l,j$ (see Section~\ref{sec:geometricinterpret});
 \item $\gamma_L^z,\gamma_\Delta^z,\gamma_\text{pref}^z\in H_1(\Lambda_z)$ are cycles on a fiber $\Lambda_z$ near $m$ (see Section~\ref{sec:geometricinterpret});
 \item $T$ is the matrix given in Equation~\eqref{eqn:T}.
\end{itemize}
Denote by $\mathrm{Re}(z)$ the real part of a complex number $z$.
Combining known results with results established in this paper, we obtain the following equivalent ways to characterize the twisting index:

\begin{theorem}\label{thm:geometry}
  The twisting index can equivalently be described in the following ways:
 \begin{enumerate}
  \item \label{itemthm:original} The original definition, which is a difference of momentum maps: $T^{\kappa^\Delta} \mu_\Delta = \nu$; 
  \item \label{itemthm:actions} A difference of actions: $\kappa^\Delta \mathrm{Re}(z)=I^\Delta(z)-\xi(z)$;
  \item \label{itemthm:series} A difference of Taylor series: $\kappa^\Delta 2\pi l = (S^\Delta)^\infty - (S^\text{pref})^\infty$;
  \item \label{itemthm:cycles} A difference of cycles: $\kappa^\Delta \gamma_L^z = \gamma_\Delta^z - \gamma_\text{pref}^z$ as elements of $H_1(\Lambda_z)$
 \end{enumerate}
\end{theorem}

In each case, we obtain $\kappa^\Delta$ by comparing an object determined by the choice of polygon with a local preferred object described in a neighborhood of the fiber containing the focus-focus point $m$.
Item~\eqref{itemthm:original} is the original definition of Pelayo \& V\~{u} Ng\d{o}c~\cite{PVNinventiones}, and the other items are all proved in Section~\ref{sec:geometricinterpret}. Item~\eqref{itemthm:actions}, which follows quickly from the original definition, has already appeared in several places in the literature, and is stated as Lemma~\ref{lem:twist-action}.
Item \eqref{itemthm:series} is discussed in Proposition~\ref{prop:difference-Taylor}, and item~\eqref{itemthm:cycles} is given in Proposition~\ref{prop:cycles}.
The way in which the twisting index is linked to the Taylor series in item~\eqref{itemthm:series} above is analogous to the treatment in Palmer $\&$ Pelayo $\&$ Tang~\cite{PPT}, Alonso~\cite{Jaume-thesis}, and in particular Le Floch $\&$ \vungoc~\cite[Proposition 2.14]{LFVN21}.

Now we introduce the family of systems for which we will compute the twisting index.
This family is a one-parameter subfamily of the family of systems (sometimes called {\it generalized coupled angular momenta}) studied in Hohloch \& Palmer \cite{HoPa2018} and Alonso \& Hohloch \cite{AH2}. We chose this subfamily because it has a relatively simple dependence on the parameter while still having two focus-focus points during a certain parameter interval.
Let $\mbS^2$ be equipped with cartesian coordinates $(x,y,z)$ induced from the inclusion $\mbS^2\subset\R^3$ as the unit sphere
and symplectic form $\om_{\mbS^2} = x \dee y\wedge \dee z + y \dee z\wedge \dee x +z \dee x\wedge \dee y$.
Let $(M,\om)$ be the symplectic manifold given by $M:=\mbS^2 \times \mbS^2$ with coordinates $(x_1,y_1,z_1,x_2,y_2,z_2)$
and symplectic form $\om := - (\om_{\mbS^2} \oplus 2  \om_{\mbS^2})$.
Then we define, for $s\in[0,1]$, the parameter-dependent family of systems $(M,\om,F_s)$, with $F_s:=(L,H_s):M \to \R^2$ given by
\begin{equation}\label{eqn_ssys}
 \left\{
  \begin{aligned}
   L(x_1, y_1, z_1 , x_2, y_2, z_2)& := z_1 + 2 z_2 , \\ 
   H_s(x_1, y_1, z_1 , x_2, y_2, z_2)  & :=   (1 - s) z_1 + s z_2 +2 (1 - s) s (x_1 x_2 + y_1 y_2).
  \end{aligned}
 \right.
\end{equation} 
The Hamiltonian flow of the function $L$ corresponds to a simultaneous rotation on both spheres about their respective vertical axes. The Hamiltonian flow of $H_s$ corresponds to an interpolation between rotations on each sphere and the scalar product of the horizontal projection of both spheres. 
Note that we use this symplectic form, in which the second factor is scaled by two, so that the focus-focus values occur at different values of $J$ (i.e.~so that this integrable
system satisfies the condition of \emph{simplicity}).

\begin{remark}
Note that there are two main conventions of notations in the literature around semitoric systems: In this paper, we use the notation $F = (L,H)$ since $L$ is an angular momentum and, later on, we use $J$ to denote the imaginary action (see Section~\ref{sec:prep-for-TS}). 
The other convention uses $J$ instead as the first component of the momentum map,  to obtain $F=(J,H)$.
\end{remark}

Let $\mathcal{N}=(0,0,1)$ and $\mathcal{S}=(0,0,-1)$ denote the poles of the 2-sphere $\mbS^2$.
Let $s_-,s_+\in [0,1]$ be as in Proposition~\ref{prop:nff}.
The result of that proposition implies that 
if $s\notin \,\,]s_-,s_+[\,$ then the system has no focus-focus points,
and if $s\in \,\,]s_-,s_+[\,$, then the system has exactly two focus-focus points, given by
$m_1:=\mathcal{N}\times\mathcal{S}$ and $m_2:=\mathcal{S}\times\mathcal{N}$.

The twisting index consists of an integer (``the integer label'') assigned to each focus-focus point on each semitoric polygon representative.
Given such labels on a single representative, the group action given in Equation~\eqref{eq:twisact} can be used to determine the labels on all representatives.
That is, to specify the twisting index of our system we just have to state the
integer labels on a single choice of semitoric polygon representative, which is precisely what we do in the following theorem.

\begin{theorem}\label{thm:twist-intro}
Let $(M,\om, F_s = (L,H_s))$ be as in Equation~\eqref{eqn_ssys} with $s\in \,\,]s_-,s_+[\,$. Then the twisting index invariant of $(M,\om, F_s)$ is independent of the choice of $s$. It is the unique one which assigns the tuple of indices $(\kappa_1^s,\kappa_2^s)=(-1,-2)$ to the
representative of the semitoric polygon invariant with upwards cuts which is the convex hull of $(-3,2)$, $(-1,2)$, $(1,0)$, and $(3,-4)$, which is displayed as the upper left subfigure in Figure \ref{fig:polygons-with-kappa-2}.
\end{theorem}


\begin{figure}[ht]
\centering
    \begin{subfigure}[b]{0.35\textwidth}
    	\centering
        \includegraphics[width=0.8\textwidth]{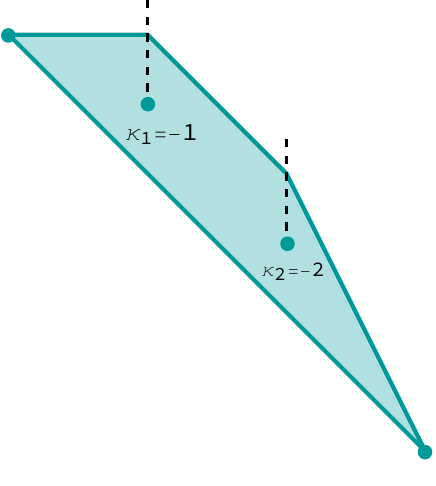}
        \caption*{\small{
        $(\kappa_1^s,\kappa_2^s)=(-1,-2)$
        }}
    \end{subfigure}
    \begin{subfigure}[b]{0.35\textwidth}
    	\centering
        \includegraphics[width=0.8\textwidth]{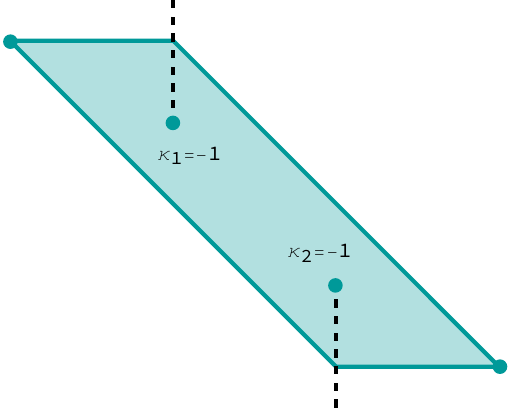}
        \caption*{\small{
        $(\kappa_1^s,\kappa_2^s)=(-1,-1)$
        }}
    \end{subfigure}\\[10pt]
    \begin{subfigure}[b]{0.35\textwidth}
    	\centering
        \includegraphics[width=0.8\textwidth]{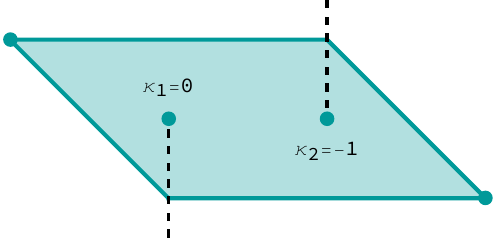}
        \caption*{\small{
        $(\kappa_1^s,\kappa_2^s)=(0,-1)$
        }}
    \end{subfigure}
    \begin{subfigure}[b]{0.35\textwidth}
    	\centering
        \includegraphics[width=0.8\textwidth]{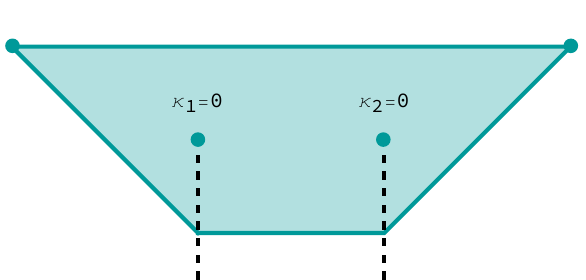}
        \caption*{\small{
        $(\kappa_1^s,\kappa_2^s)=(0,0)$
        }}
    \end{subfigure}
	\caption{
	Four representatives of the polygon invariant of the system with twisting indices labeled. To prove Theorem~\ref{thm:twist-intro}, we computed that the twisting index labels for the lower right polygon are $(\kappa_1^s,\kappa_2^s)=(0,0)$, and the others are determined by Equation \eqref{eq:twisact}.
	}
    \label{fig:polygons-with-kappa-2}
\end{figure}

Theorem~\ref{thm:twist-intro} is proved in Section~\ref{sec:twist-proof}.
Note that the twisting index is independent of the parameter $s$ for this particular system (which is very symmetric), but this may not be true for other one-parameter families of systems.
More precisely, we are not aware of any obstruction to the twisting index changing in such a family.

This paper is the first time, to our knowledge, that anyone has attempted to explicitly
compute the twisting index of a system with more than one focus-focus singular point.
In the process of performing these calculations, we noticed a small oversight in the original formula about how the twisting index changes between different polygons:
there is an extra term which was missing from the original formulation in Pelayo \& V\~{u} Ng\d{o}c~\cite{PVNinventiones}.
This term is automatically zero if the system has strictly less than two focus-focus singular points, which has been the case for all systems whose twisting index has been treated so far in the literature. We include a corrected formula in Equation~\eqref{eq:twisact} and discuss this in Remark~\ref{rmk:twist-changes}.

In order to compute the twisting index in Theorem~\ref{thm:twist-intro}, we also needed to obtain the lower order terms of the Taylor series invariant at each focus-focus point of the system.

\begin{theorem}\label{thm:Taylor-intro}
Let $(M,\om, F_s = (L,H_s))$ be as in Equation~\eqref{eqn_ssys} with $s\in \,\,]s_-,s_+[\,$.
Let $S_{i,s}^\infty(l,j)$ denote the Taylor series invariant at the focus-focus point $m_i$, for $i\in\{1,2\}$.
Then
\begin{equation*}
\begin{aligned}
S_{1,s}^\infty(l,j) &= l \arctan \left( \dfrac{6-9s}{\ra} \right)  + j \log \left( \dfrac{\ra^3}{\sqrt{2}\rc (1-s)^2 {s}^2}  \right) 
\\&+ \dfrac{l^2}{8 \ra^3 \rc^2} \left( 3(2-3s)(16 - 96 {s} - 216 {s}^2 + 1944 {s}^3  - 3211 {s}^4 + 424 {s}^5 \right. \\ & \qquad \left. + 
 3252 {s}^6 - 2816 {s}^7 + 704 {s}^8) \right)  \\&+ \dfrac{lj}{16 \ra^3 \rc^2} \left( \ra (16 - 96 {s} + 360 {s}^2 - 936 {s}^3 + 2693 {s}^4 - 6200 {s}^5 + 
 8004 {s}^6 \right. \\ & \qquad \left.- 5120 {s}^7 + 1280 {s}^8)\right) 
\\&+ \dfrac{j^2}{8 \ra^3 \rc^2} \left( (-96 + 720 {s}   - 7248 {s}^2 + 36312 {s}^3 - 99558 {s}^4  + 174957 {s}^5 \right. \\&\qquad  \left. - 
 211536 {s}^6 + 171924 {s}^7 - 83328 {s}^8 + 17856 {s}^9) \right)  + \mcO(3),
\end{aligned}
\end{equation*}
where $\mcO(3)$ denotes terms of order greater than or equal to three in the variables $l$ and $j$, and $\rho_1,\rho_2$ are the functions of $s$ given in Equation~\eqref{eqn:rho}.
Furthermore, due to the symmetries of the system, 
\begin{equation}\label{eqn:symmetry-intro}
S_{2,s}^\infty(l,j) = -S_{1,s}^\infty(-l,-j)+\pi l.
\end{equation} 
\end{theorem}

Theorem~\ref{thm:Taylor-intro} is proved in Section~\ref{sec:proof-of-Taylor}.
One important ingredient of the proof is the following proposition which explains how the Taylor series invariant changes under symplectormorphisms that reverse the signs of the components of the momentum map.

\begin{proposition}\label{prop:symmetry-general}
Let $(M,\om,F=(L,H))$ and $(M',\om',F'=(L',H'))$ be semitoric integrable systems, and let $m\in M$ be a focus-focus point which we further assume to be the only singular point in $F^{-1}(F(m))$.
Moreover, let $\Phi\colon M\to M'$ be a symplectomorphism such that $\Phi^*(L',H') = (\varepsilon_1 L, \varepsilon_2 H)$ for some $\varepsilon_1,\varepsilon_2\in \{-1,+1\}$. 
Then $\Phi(m)$ is a focus-focus point, and
\[
 S_m^\infty(l,j) = \varepsilon_2 (S_{\Phi(m)}')^\infty(\varepsilon_1 l, \varepsilon_2 j) + \left(\frac{1-\varepsilon_1}{2}\right) \pi l \quad (\textrm{mod }2\pi l),
\]
where $S_m^\infty(l,j)$ denotes the Taylor series invariant at $m$ and $(S_{\Phi(m)}')^\infty(l,j)$ denotes the Taylor series invariant at $\Phi(m)$.
\end{proposition}

Proposition~\ref{prop:symmetry-general} is proved in Section~\ref{sec:symmetry}, relying on some results by Sepe \& V\~{u} Ng\d{o}c~\cite{SepeVN-notes}.

\subsection{Impact and future applications}

The semitoric classification represents important progress in the task
of obtaining global symplectic classifications of integrable systems.
In dimension four, the classes of almost toric systems (see Symington~\cite{Sy2003}) and hypersemitoric systems (see Hohloch $\&$ Palmer~\cite{Hohloch2022}) are natural candidates for the next broader class of systems to be symplectically classified,
and similar classifications could also be performed in higher dimensions.
In any of these cases, an invariant similar to the twisting index invariant would 
have to appear in the classification, either directly or indirectly encoded in other invariants.

Thus, to further expand the classification of integrable systems from semitoric to some broader
class, one must first understand the true nature of the twisting index, on both a conceptual and computational level,
to be able to adapt it properly. These foundations are laid in the present paper.

\subsubsection*{Structure of the article:} 
In Section~\ref{sec:preliminaries}, we explain the necessary background knowledge. In Section~\ref{ss:invariants}, we describe, in detail, the five invariants which appear in the semitoric classification.
In Section~\ref{sec:geometricinterpret}, we provide dynamical, geometric, and topological interpretations of the twisting index by producing several equivalent definitions of the invariant.
In Section~\ref{sec:specificexample}, we compute the twisting
index and some terms of the Taylor series for a specific example, the system given in Equation~\eqref{eqn_ssys} which has two focus-focus singular points.

\subsubsection*{Acknowledgements:} 
We would like to thank Holger Dullin, San \vungoc, Konstantinos Efstathiou, Yohann Le Floch, and Marine Fontaine for helpful discussions and Wim Vanroose for sharing his computational resources. The authors have been partially funded by the FWO-EoS project G0H4518N. The last author was partially supported by UA-BOF project with Antigoon-ID 31722 and an FWO senior postdoctoral fellowship 12ZW320N.

\section{Preliminaries}
\label{sec:preliminaries}

In this section we introduce the background, concepts, notation, and results necessary for this paper. In particular, we summarise some of the properties of semitoric systems and describe their classification by Pelayo \& \vungoc\ \cite{PVNinventiones, PVNacta} and Palmer \& Pelayo \& Tang~\cite{PPT}. For general overviews of semitoric systems we recommend the following surveys: Pelayo \& \vungoc\ \cite{PV2}, Pelayo \cite{Pe}, Sepe \& \vungoc\ \cite{SepeVN-notes}, Alonso \& Hohloch \cite{AH}, Henriksen \& Hohloch \& Martynchuk~\cite{HHM}, and Gullentops $\&$ Hohloch \cite{GH}.

\subsection{Integrable systems}
\label{ss:integrable}

Let $(M,\om)$ be a symplectic manifold. Since the symplectic form $\om$ is non-degenerate, we can associate to each $f \in \mathcal{C}^\infty(M, \R)$ a vector field $\mathcal{X}_f$ using the relation $\om(\mathcal{X}_f,\cdot) = -\dee f(\cdot)$. In this situation, we say that $f$ is a \emph{Hamiltonian function} and $\mathcal{X}_f$ its \emph{Hamiltonian vector field}. The flow of $\mcX_f$ is called the \emph{Hamiltonian flow generated by $f$}. The \emph{Poisson bracket} of $f,g \in \mathcal{C}^\infty(M, \R)$ is defined by $\{f,g\} := \om(\mathcal{X}_f,\mathcal{X}_g)$.

\begin{definition}
A \emph{$2n$-dimensional (completely) integrable system} is a triple $(M,\om,F)$, where $(M,\om)$ is a $2n$-dimensional symplectic manifold and $F:=(f_1, \ldots, f_n):M \to \mathbb{R}^n$, called the \emph{momentum map}, satisfies the following conditions:
\begin{enumerate}
	\item $\{f_i,f_j\}=0$ for all $1 \leq i,j \leq n$.
	\item The Hamiltonian vector fields $\mcX_{f_1}, \ldots, \mcX_{f_n}$ are linearly independent almost everywhere in $M$.
\end{enumerate}
\end{definition}

\begin{definition}
\label{def:isom}
Two integrable systems $(M,\om,F)$ and $(M',\om',F')$ are said to be \emph{isomorphic} if there exists a pair of maps $(\varphi, \varrho)$, where $\varphi:(M,\om) \to (M',\om')$ is a symplectomorphism and $\varrho:F(M) \to F'(M')$ is a diffeomorphism, such that $\varrho \circ F = F' \circ \varphi$. 	
\end{definition}

The rank of $dF$ at any point $x\in M$ is equal to the dimension of the span of $\mcX_{f_1}, \ldots, \mcX_{f_n}$ at $x$. A point $x\in M$ is called a \emph{regular} point of $F$ if $dF(x)$ has maximal rank, and otherwise it is called \emph{singular} and the rank of $dF(x)$ is called the \emph{rank of $x$}.
A point $c\in F(M)$ is called a \emph{regular value of $F$} if $F^{-1}(c)$ only contains regular points, in which case the fiber $F^{-1}(c)$
is also called \emph{regular}.
The symplectic form vanishes on the fibers of $F$, and in particular, the regular fibers of $F$ are Lagrangian submanifolds of $M$. 
Thus, the map $F$ induces what is called a \emph{singular Lagrangian fibration} on $M$.

\subsection{Regular points}
An example of an integrable system is the cotangent bundle $T^*\T^n$ of the $n$-torus $\T^n$. If $(q_1, \ldots, q_n, p_1, \ldots, p_n)$ are local coordinates in $T^*\T^n\cong \mathbb{T}^n\times \R^n$, then the symplectic form is locally given by $\om_{T^*\T^n} = \sum_{j=1}^n dp_j \wedge dq_j$ and the momentum map by $F_{T^*\T^n} = (p_1, \ldots, p_n)$. This example serves as a model for neighborhoods of regular fibers, as shown in the following theorem.

\begin{theorem}[Liouville Arnold Mineur theorem~\cite{Ar}]\label{thm:action-angle}
Let $(M,\om,F)$ be an integrable system and let $c\in F(M)$ be a regular value. If $\Lambda_c := F^{-1}(c)$ is a regular, compact and connected fiber, then there exist neighborhoods $U \subset F(M)$ of $c$ and $V \subset \R^n$ of the origin, such that taking
\[\mcU := \bigsqcup_{r \in U}F^{-1}(r) \text{ and } \mcV:= \T^n \times V \subset T^*\T^n\]
we have that
 $(\mcU,\om|_\mcU,F|_\mcU)$ and $(\mcV,\om_{T^*\T^n}|_{\mcV}, F_{T^*\T^n}|_\mcV)$ are isomorphic integrable systems as in Definition \ref{def:isom}. 
\end{theorem}

In particular this means that $\Lam_c \simeq \T^n$ and that $F|_\mcU$ is a trivial torus bundle. The local coordinates $p_j$ on $T^*\T^n$ introduced above are called \emph{action coordinates} and the $q_j$ are called \emph{angle coordinates}. Furthermore, following~\cite{Ar}, if $\{\ga_1(c), \ldots, \ga_n(c)\}$ form a basis of the homology group $H_1(\Lam_c)$, varying smoothly with $c \in U$, then the action coordinates are given by 
\begin{equation}\label{eqn:action-coords}
 p_j(c) = \frac{1}{2\pi} \oint_{\ga_j(c)} \varpi,
\end{equation}
where $\varpi$ is any one-form such that $d \varpi|_{\mcU} = \omega$ on $\mcU$.

\begin{remark}\label{rmk:integrate-cycles}
 As above, consider $\mathbb{T}^n\times \R^n$ with coordinates $(q_1,\ldots, q_n,p_1,\ldots, p_n)$ and symplectic form $\om_{T^*\T^n}$. Note that $\om = \dee \alpha$ where $\alpha = \sum p_i\dee q_i$. Note that the Hamiltonian vector field of $p_i$ is simply $\mathcal{X}_{p_i}=\frac{\partial}{\partial q_i}$. Fix some value $x\in\R^n$ and consider $\Lambda = \pi^{-1}(x)$ where $\pi$ is the projection onto $\R^n$. 
 Let $\gamma$ denote an orbit of the flow of $\mathcal{X}_{p_i}$ contained in $\Lambda$. Then
 \[
 \frac{1}{2\pi}\int_{\gamma}\alpha = \frac{1}{2\pi}\int_0^{2\pi} p_i \dee t = p_i.
 \]
 Let $\tilde{p}$ be an action variable of an integrable system, let $x$ be a regular value of the integrable system, and suppose that $\gamma_x$ is a closed orbit of the flow generated by $\tilde{p}$ contained in the fiber over $x$. Then combining the above argument with Theorem~\ref{thm:action-angle}, we conclude that 
 $\frac{1}{2\pi} \int_{\gamma_x} \alpha = \tilde{p}(x)$
 for any $\alpha$ such that $\dee \alpha = \omega$.
\end{remark}

\subsection{Integral affine structures}
\label{sec:int-affine}

The group $\mathrm{GL}(n,\Z) \ltimes \R^n$ is called the group of integral affine maps of $\R^n$, and this group acts on $\R^n$ where the first component acts by composition and the second acts by translation.
An \emph{integral affine structure} on an $n$-manifold $X$ is an atlas of charts on $X$
such that the transition functions between these charts are integral affine maps on $\R^n$.
If $X$ and $Y$ are manifolds equipped with integral affine structures, we call a map $g\colon X \to Y$ an integral affine map if it sends the integral affine structure of $X$ to the integral affine structure of $Y$.

Now suppose that $(M,\om,F)$ is an integrable system with connected fibers, and let $B=F(M)$ denote the momentum map image.
Let $B_r \subseteq B$ denote the set of regular values of $F$.
Given any $c\in B_r$, applying Theorem~\ref{thm:action-angle} induces action coordinates $p_1,\ldots, p_n$
in a neighborhood of $c$.
Since they arise from a choice of basis of $H_1(\Lambda_c)$ and a choice of primitive of $\om$ via Equation~\eqref{eqn:action-coords}, 
the action coordinates are unique up to the action of $\mathrm{GL}(n,\Z) \ltimes \R^n$.
Thus, the action coordinates induce an integral affine structure on $B_r$.
Note that this construction also works if the fibers of $F$ are disconnected, in which case the base $B$ of the Lagrangian fibration 
cannot be identified with the image $F(M)$ so the description is somewhat more complicated.

\subsection{Singular points}
Theorem~\ref{thm:action-angle} gives a normal form for a neighborhood of a regular fiber.
Now, we move on to gaining a local understanding of singular points.
In order to classify singular points of an integrable system, we first introduce a notion of non-degeneracy.

\begin{definition}
\label{def_nondeg}
Let $(M,\om,F=(f_1,\ldots,f_n))$ be an integrable system. A rank $0$ singular point $p\in M$ of $F$ is {\em non-degenerate} if the Hessians
$d^2f_1(p),\ldots, d^2f_n(p)$ span a Cartan subalgebra
of the Lie algebra of quadratic forms on the tangent space $T_pM$.
\end{definition}

The notion of non-degenerate singular points can also be extended to points of all ranks, see Bolsinov \& Fomenko~\cite[Section 1.8]{bolsinov-fomenko}.
We will explain an equivalent formulation of this condition for four dimensional integrable systems in Section~\ref{sec:dim4}.

The following local theorem gives a local normal form for non-degenerate singular points. It is based on works by Rüssmann~\cite{Russmann64}, Vey~\cite{Vey78}, Colin de Verdière and Vey~\cite{ColindeVerdiereVey79}, Eliasson~\cite{Eliasson84, Eliasson90}, Dufour and Molino~\cite{DufourMolino88}, Miranda~\cite{Miranda2003,Miranda2014}, Miranda and Zung~\cite{MirandaZung04}, Miranda and V\~{u} Ng\d{o}c~\cite{MirandaVuNgoc05}, V\~{u} Ng\d{o}c and Wacheux~\cite{VuNgocWacheux13} and Chaperon~\cite{Chaperon13}.

\begin{theorem}[Local normal form for non-degenerate singular points]
\label{EliassonMZ}
 Let $(M,\om,F)$ be a $2n$-dimensional integrable system
 and let $m\in M$ be a non-degenerate singular point. Consider also $(\mathbb{R}^{2n},\omega_{\mathbb{R}^{2n}})$ with coordinates $(q,p):=(q_1, \ldots, q_n, p_1, \ldots, p_n)$, where $\omega_{\mathbb{R}^{2n}} = \sum_i \dee p_i\wedge \dee q_i$. 
 Then there exist neighborhoods $U$ of $m$ and $V$ of the origin in $\mathbb{R}^{2n}$, a symplectomorphism $\varphi:U \to V$ and a function
 $G:=(g_1, \ldots, g_n): \mathbb{R}^{2n}\to\R^n$ where each $g_i$ is given by one of the following:
 \setlength{\columnsep}{-2.5cm}
 \begin{multicols}{2}
 \begin{enumerate}[nosep]
  \item 
 non-singular: $g_i = p_i$,
  \item 
  elliptic: $g_i = \frac{1}{2}(q_i^2+p_i^2)$,
  \item 
  hyperbolic: $g_i = q_ip_i$,
 \item 
 focus-focus: $\begin{cases}  g_i  = q_ip_{i+1}-q_{i+1}p_i,\\ g_{i+1}  = q_ip_i+q_{i+1}p_{i+1},\end{cases}$ 
 \end{enumerate}
 \end{multicols}
 such that $\{f_i, g_j \circ \varphi \}=0$ for all $i,j$. Moreover, if there are no hyperbolic components, then there is a local diffeomorphism $\varrho:(\R^n,F(m)) \to (\R^n,0)$ such that $\varrho \circ F =  G \circ \varphi$ on $U$, so $(U,\omega|_U,F|_U)$ and $(V,\omega_{\mathbb{R}^{2n}}|_V,G|_V)$ are isomorphic as in Definition \ref{def:isom}.
\end{theorem}

\subsection{Singular points in dimension four}\label{sec:dim4}
Now we specialize to the situation that $n=2$, and so $\mathrm{dim}(M)=4$.
We use the notation $f_1=L$ and $f_2=H$, so $F=(L,H)$. In this case, the following proposition becomes very useful for the verification of non-degeneracy.

\begin{lemma}[Bolsinov $\&$ Fomenko~\cite{bolsinov-fomenko}]
\label{nondegCriterium}
 Let $(M,\om,F = (L,H))$ be a four dimensional integrable system.
 Let $p\in M$ be a rank $0$ singular point of $F$ and $\om_p$ the matrix of the symplectic form with respect to a
 basis of $T_p M$ and $\om_p^{-1}$ its inverse. Denote by $d^2L(p)$ and $d^2H(p)$ the matrices of the Hessians
 of $L$ and $H$ with respect to the same basis.
 Then $p$ is non-degenerate if and only if $d^2L(p)$ and
 $d^2H(p)$ are linearly independent and there exists a linear
 combination of $\om_p^{-1}d^2 L(p)$ and $\om_p^{-1}d^2 H(p)$ with four distinct eigenvalues. 
\end{lemma}

Let $p$ be a singular point of rank 1 in a four dimensional integrable system $(M,\om,F=(L,H))$.
Then there exists some $\al,\be\in\mathbb{R}$ with $\al d H(p) + \be dL(p) = 0$.
The flows of $\mcX^L$ and $\mcX^H$ induce an $\mathbb{R}^2$\--action
which has a one-dimensional orbit through $p$.
Let $P_p\subset T_pM$ be the tangent line of this orbit at $p$ and denote by $P^\om_p$ the symplectic orthogonal complement of $P_p$.
Notice that $P_p\subset P^\om_p$ and since $L$ and $H$ Poisson commute they are invariant under the $\R^2$\--action. Thus $\al d^2H(p) + \be d^2L(p)$
descends to the quotient $P_p/P^\om_p$.

\begin{definition}[Bolsinov $\&$ Fomenko~\cite{bolsinov-fomenko}]
 Let $(M,\om,F = (L,H))$ be a four dimensional integrable system.
 A rank 1 singular point $p$ is \emph{non-degenerate} if $\al d^2H(p) + \be d^2L(p)$
 is invertible on $P_p/P^\om_p$.
\end{definition}

\section{Semitoric systems}
\label{ss:invariants}

Here we introduce the class of systems that we will be focusing on in this paper.
We will also give a precise and concise description of the invariants of a simple semitoric system.
In what follows we identify $\mbS^1$ with $\R/2\pi\Z$, and so an effective action of $\mbS^1$ is equivalent to a periodic $\R$-action with minimal period $2\pi$. Recall that a map is called \emph{proper} if the preimage of every compact set is compact.

\begin{definition}
A \emph{semitoric system} is a four dimensional integrable system $(M,\om,F=(J,H))$, such that:
\begin{enumerate}
	\item The Hamiltonian flow generated by $L$ is an effective $\mbS^1$-action.
	\item $L$ is proper. 
	\item The singular points of $F$ are non-degenerate and have no hyperbolic components.
\end{enumerate}
If, moreover, the following condition is satisfied, we say that the semitoric system is \emph{simple}:
\begin{enumerate}
\setcounter{enumi}{3}
	\item In each level set of $L$, there is at most one singularity of focus-focus type.
\end{enumerate}
\end{definition}

Note that, if $M$ is compact, then $L$ is automatically proper. Given semitoric systems $(M,\om,(L,H))$ and $(M',\om',(L',H'))$ an \emph{isomorphism of semitoric systems} is a pair $(\varphi,\varrho)$, where $\varphi$ is a symplectomorphism $\varphi \colon (M, \om)\to (M', \om')$ and $\varrho = (\varrho_1, \varrho_2):F(M) \to F'(M')$ is a diffeomorphism
satisfying $\varrho \circ F = F' \circ \varphi$ and $\varrho$ is of the form $\varrho(l,h) = (l, \varrho_2(l,h))$ with $\partial_h \varrho_2 >0$.

\begin{remark}
\label{re:isomtypes}
Isomorphisms of semitoric systems are a particular case of the isomorphisms of integrable systems from Definition \ref{def:isom}, where $(\varphi,\varrho)$ is chosen to preserve the first coordinate and the orientation of the second one. 	
\end{remark}

A consequence of Theorem \ref{EliassonMZ} is that the only types of singularities which can occur in semitoric systems are:
\begin{itemize}
 \item \emph{elliptic-elliptic points}: rank 0 points with two elliptic components,  
 \item \emph{focus-focus points}: rank 0 points with one focus-focus pair of components,
 \item \emph{elliptic-regular points}: rank 1 points with one elliptic component and one non-singular component.
\end{itemize}

Furthermore, due to \vungoc~\cite{VuNgoc07}, the fibers of the momentum map
$F\colon M\to \R^2$ in a semitoric system are connected, and thus we may identify the image $F(M)$ with the base of the singular Lagrangian fibration that $F$ induces on $M$.

Pelayo $\&$ V\~{u} Ng\d{o}c \cite{PVNinventiones, PVNacta} classified simple semitoric systems by means of five invariants.
Their original classification only applied to simple semitoric systems, but
 this classification has since been extended by Palmer \& Pelayo \& Tang~\cite{PPT} to include all semitoric systems, simple or not.
 Note that the twisting index invariant appears rather differently in the non-simple case, so for the present paper we will focus on the original classification which only applies to simple semitoric systems.

Now we will briefly describe each of the five invariants. 
More details can be found in the original papers by Pelayo \& \vungoc~\cite{PVNinventiones, PVNacta}.
In the following, let $(M,\om,F=(L,H))$ be a simple semitoric system.
Note that the invariant that we are focused on in this paper is the twisting index invariant, and we will now see that it is mainly concerned with the relationship between the Taylor series invariant and the polygon invariant.

\subsection{The number of focus-focus points invariant}
\label{sss:nff}  
The first symplectic invariant of a semitoric system is the \emph{number of focus-focus singularities} $\nff \in \mathbb{N}_0$, which is finite by V\~{u} Ng\d{o}c~\cite{VuNgoc07}. Since our system is simple, we can order the focus-focus singularities $m_1, \ldots, m_\nff$ according to their $L$-values, so that $L(m_1) < \ldots < L(m_\nff)$.

\subsection{The polygon invariant}
\label{sss:polygon}
Let $B=F(M) \subset \R^2$ be the image of the momentum map and $B_r \subseteq B$ be the set of regular values of $F$. 
As discussed in Section~\ref{sec:int-affine}, the action coordinates of Theorem~\ref{thm:action-angle} induce an integral affine structure on $B_r$, which
in general does not agree with the one induced by its inclusion in $\R^2$.
Moreover, if the semitoric system has at least one focus-focus point, then there cannot exist an integral affine map $g\colon B_r\to\R^2$ due to the presence of monodromy in the integral affine structure of $B_r$.
In this section we will explain a procedure, introduced by \vungoc\ \cite{VuNgoc07}, to obtain a map from $B_r$ to $\R^2$ which
is an integral affine map away from certain sets in $B_r$. 
Such a map can be thought of as ``straightening out'' the integral affine structure of $B_r$, and the image of such a map is a rational convex polygon. The equivalence class of such polygons (up to the freedom in the choice of map) is the semitoric polygon invariant. We explain the details now.

Recall that $\{m_1,\ldots,m_{\nff}\}$ denote the focus-focus points of the system. For $r \in \{1, \ldots, \nff\}$, let $c_r:=F(m_r)=(\lam_r,\eta_r)\in\R^2$ denote the focus-focus singular value and let $b_{\lam_r}: = \{(\lam_r,y)\;|\; y \in \R\}$ be the vertical line through $c_r$. We associate a sign $\varepsilon_r \in \Z_2: = \{-1,+1\}$ with each singular value and let $b^{\varepsilon_r}_{\lam_r}$ denote the half-line starting in $c_r$ which goes upwards if $\varepsilon_r=+1$ and downwards if $\varepsilon_r=-1$. For each choice of signs $\varepsilon=(\varepsilon_r)_{r=1}^\nff \in (\Z_2)^\nff$, we remove the union of half-lines $b^\varepsilon := \bigcup_r b_{\lam_r}^{\varepsilon_r}$ from $B$, obtaining the simply connected set $B \backslash b^\varepsilon \subset \R^2$.

\vungoc\ \cite{VuNgoc07} showed that for each vector $\varepsilon$ there exists a map $f_\varepsilon:B \to \R^2$ such that:
\begin{enumerate}
 \item $\De:= f_\varepsilon(B)$ is a rational convex polygon, 
 \item $f_\varepsilon$ is a homeomorphism,
 \item $f_\varepsilon$ preserves the first coordinate, i.e.\ $f_\varepsilon(l,h)=\bigl(f_\varepsilon^{(1)},f_\varepsilon^{(2)}\bigr)(l,h)=\bigl(l,f_\varepsilon^{(2)}(l,h)\bigr)$,
 \item  $(f_\varepsilon)|_{B\backslash b^\varepsilon}$ is a diffeomorphism into its image which sends the integral affine structure of $B_r\backslash b^\varepsilon$ to the integral affine structure of $\R^2$.
 \end{enumerate}
   The map $f_\varepsilon$ is commonly referred to as \emph{cartographic homeomorphism}.
Cartographic homeomorphisms are not unique:
let \begin{equation}\label{eqn:T}
T := \begin{pmatrix}
1 & 0 \\ 1 & 1
\end{pmatrix} \in \text{GL}(2,\Z),
\end{equation} 
let $\mcG:=\{T^n \;|\; n \in \Z\}$,
and let $\mcV := \mcG \ltimes \R$.
Then, given a choice of $\varepsilon$, the map $f_\varepsilon$ obtained via the construction above is unique up to left composition by an element of $\mcV$. 
Note that the components of the cartographic homeomorphism are the actions of the action-angle coordinates.

By definition, the map $f_\varepsilon$ depends on the choice of $\varepsilon$, the effect of which we describe now. Denote by $t_{b_{\lam_r}} : \R^2 \to \R^2$ the map given by
\[
 t_{b_{\lam_r}}(x,y) = \begin{cases} (x,y), &\text{ if }x\leq \lam_r\\ (x,y + x-\lam_r), &\text{ if }x>\lam_r. \end{cases}
\]
That is, intuitively $t_{b_{\lam_r}}$
leaves the half-plane to the left of $b_{\lam_r}$ invariant and applies $T$, relative to a choice of origin on $b_{\lam_r}$, to the half-plane on the right of $b_{\lam_r}$. 
Given a vector $\ka=(\ka_1,\ldots, \ka_\nff)\in \Z^\nff$, set $t_\ka := t_{\lam_1}^{\ka_1} \circ \cdots \circ t_{\lam_\nff}^{\ka_\nff}$, and note that $t_\ka$ is a piecewise integral affine map. 
A different choice of signs changes the cartographic homeomorphism by composition with such a map, as we will describe below.

We now define the polygon invariant. The triple $\big(\De, (b_{\lam_r})_{r=1}^\nff,(\varepsilon_r)_{r=1}^\nff\big)$ is called \emph{weighted polygon}. The freedom in the definition of $f_\varepsilon$ can be expressed as an action of the group $(\Z_2)^\nff \times\ \mcG$ on the space of weighted polygons: letting $(\varepsilon',T^n) \in (\Z_2)^\nff \times \mcG$ and taking $u := \tfrac{1}{2}(\varepsilon - \varepsilon\varepsilon')$, the action is given by
$$ ( \varepsilon', T^n) \cdot \big(\De, (b_{\lam_r})_{r=1}^\nff,(\varepsilon_r)_{r=1}^\nff\big) = \bigl( t_u(T^n(\De)), (b_{\lam_r})_{r=1}^\nff,(\varepsilon_r'\varepsilon_r)_{r=1}^\nff\bigr).$$
Since the non-uniqueness of the cartographic homeomorphism is encoded in the group action, this assignment described in the following definition is well-defined.

\begin{definition}\label{def:polygon}
Given a choice of cartographic homeomorphism, we define the 
\emph{polygon invariant} of $(M,\om,F)$ to be the $(\Z_2)^\nff \times \mcG$-orbit of weighted polygons obtained
from its image, briefly denoted by $(\De, b, \varepsilon)$.
\end{definition}

Now, let $f_\varepsilon$ be any choice of cartographic homeomorphism with cut directions $\varepsilon$,
let $M^\varepsilon := M\setminus F^{-1}(b^\varepsilon)$, and let $\mu:=f_\varepsilon\circ F\colon M\to \R^2$.
Then $\mu$ restricted to $M^\varepsilon$ is an integrable system of which all points are regular, elliptic-regular, or elliptic-elliptic, and by definition the induced integral affine structure
on the regular points of $\mu(M^\varepsilon)$ is equal to the integral affine structure on $\R^2$.
Thus, the Hamiltonian flows of the components of $\mu$ form a Hamiltonian $2$-torus action on $M^\varepsilon$.
For this reason, we call $\mu$ a \emph{generalized toric momentum map}.
Note that the semitoric polygon is the image of $\mu$, that is, $\mu(M) = f_\varepsilon(B) = \Delta$.

\subsection{The height invariant}
\label{sss:height}
Let $f_\varepsilon$ be a cartographic homeomorphism and let $\mu = (\mu_1,\mu_2):=f_\varepsilon\circ F$ denote the associated generalized toric momentum map.
 For $r \in \{1, \ldots, \nff\}$, we define $h_r$ as the height of the corresponding focus-focus value measured from the lower boundary of the polygon $\De=f_\varepsilon(M)$:
 $$ h_r := \mu_2(m_r) - \min_{p \in \De \cap b_{\lam_r}} \operatorname{proj}_2(p),$$
where $\operatorname{proj}_2:\R^2 \to \R$ is the natural projection on the second coordinate. Geometrically, the value $h_r$ measures the symplectic volume of the submanifold
\[
 \{p \in M \;|\; L(p)=L(m_r) \text{ and }H(p)<H(m_r)\}
\]
and therefore is independent of the choice of $f_\varepsilon$. The \emph{height invariant} of the simple semitoric system $(M,\om,F)$ is the $\nff$-tuple $(h_1, \ldots, h_\nff)$ of heights of all focus-focus values.

\subsection{The Taylor series invariant}
\label{sss:taylor} 
In this section, we outline a construction due to \vungoc~\cite{VuNgoc03} of a semilocal invariant of a focus-focus fiber.
Let $m \in M$ be a focus-focus singularity of the semitoric system $(M,\om,F)$. By Theorem \ref{EliassonMZ} there exist neighborhoods $U\subset M$ of $m$ and $V \subset \R^4$ of the origin and a isomorphism of integrable systems $(\varphi, \varrho)\colon(U,\om|_U,F|_U)\to (V,\om_{\R^4}|_V,G|_V)$, where 
\begin{equation} \label{eq:defG}
	G(q_1,q_2,p_1,p_2)=(q_1p_2-q_2p_1,q_1p_1+q_2p_2).
\end{equation}
As discussed in Sepe \& V\~{u} Ng\d{o}c~\cite{SepeVN-notes}, for a semi-local neighborhood of a focus-focus point in a general integrable system, there is still a degree of freedom in the choice of
such a $(\phi,\varrho)$.
In this section, following~\cite{VuNgoc03}, we will describe how a choice of $(\phi,\varrho)$  around a focus-focus point produces an invariant of the fiber containing that point, called the Taylor series invariant.
Note that different choices of $(\phi,\varrho)$ produce different Taylor series invariants (related by a simple formula, we discuss this in Section~\ref{sec:symmetry}), but for a semitoric system there are preferred choices.
If $(M,\om,(L,H))$ is semitoric, then $\varrho = (\varrho_1,\varrho_2)$ is of the form $\varrho(l,h) = (\pm l,\varrho_2(l,h))$ with $\frac{\partial \varrho_2}{\partial h}\neq 0$, and we make the preferred choice of $\phi$ and $\varrho$ such that $\varrho(l,h)=(l,\varrho_2(l,h))$ where $\frac{\partial \varrho_2}{\partial h}>0$. With this choice the invariant we construct in this section is well-defined. Understanding this choice is important in Section~\ref{sec:symmetry}, when we discuss how changing the signs of the components of the momentum map of a semitoric integrable system impacts the invariants.

Now we will consider a neighborhood of the focus-focus fiber. Let $W := F^{-1}(F(U))$ and let $\Phi=(\Phi_1,\Phi_2): W \to G(V)$ be defined by $\Phi = \varrho \circ F$. Let $z:=(l,j) := l+ ij$  denote the coordinate on $G(V) \subseteq \R^2 \simeq \C$ induced by the identification with $\C$, and let $\Lam_z:= \Phi^{-1}(z)$. 
For any nonzero $z\in\Phi(W)$, note that $\Phi^{-1}(z)$ is a regular fiber, and thus by Theorem~\ref{thm:action-angle} and the fact that $L$ is proper, this fiber is diffeomorphic to a $2$-torus.
Note that $\Phi_1=L$, so $\mcX_{\Phi_1} = \mcX_L$ and thus the flow of $\mcX_{\Phi_1}$ generates an effective $\mbS^1$-action on $\Lam_z$. In contrast, the flow of $\mcX_{\Phi_2}$ will be in general quasi-periodic.

Note that the monodromy induced by a focus-focus point discussed in Section~\ref{sss:polygon} will appear again here in the following way:
in principle, to directly emulate the situation of Equation~\eqref{eqn:action-coords}, a smoothly varying basis of the fundamental group of $\Lambda_z$ for $z\in\Phi(W)$ is needed, but this unfortunately does not exist due to the monodromy. Instead, for $\varepsilon\in \Z_2$ let $b_\varepsilon\subset \R^2$ be a ray starting at the origin and going up if $\varepsilon=+1$ and down if $\varepsilon=-1$. Then let
$\tilde{W} = F^{-1}(F(W)\setminus b_\varepsilon)$. Now we choose a basis
$\{\gamma_1^z,\gamma_2^z\}$ of the fundamental group of the torus $\Lam_z$ which varies smoothly with $z\in\Phi(\tilde{W})$, 
and where $\ga_1^z = \gamma_L^z \subset \Lam_z$ is the cycle corresponding to the flow of $\mcX_L$ with its same orientation. 
From Theorem \ref{EliassonMZ} the actions of the system are  given by 
\begin{equation}
	L(z) = \dfrac{1}{2\pi} \oint_{\ga_L^z} \varpi = l,\qquad I(z):= \dfrac{1}{2\pi} \oint_{\gamma_2^z} \varpi,
\label{eq:semigact}
\end{equation} where $\varpi$ is a primitive of the symplectic form $\om|_W$. We moreover impose that the orientation of $\gamma_2^z$ is such that $\tfrac{\partial I}{\partial j}>0$, where $z=l+ij$, which is equivalent to taking the preferred choice of $\varrho$ for semitoric systems. The function $I(z)$ can be extended continuously to all of $\Phi(W)$ but not smoothly. To address this, let now log denote a determination of the complex logarithm with its branch along the ray $b_\varepsilon$ (i.e.~the ray corresponds to the choice of domain of definition for log)
which we used
above for determining a basis of cycles.
V\~{u} Ng\d{o}c~\cite{VuNgoc03} showed that 
\begin{equation}\label{eqn:def-of-S}
S(z) := 2\pi I(z) - 2\pi I(0) + \text{Im}(z \log z -z),
\end{equation} can be extended to a smooth function on all of $\Phi(W)$.
The function $S$ is often referred to as the \emph{desingularised} or \emph{regularised} action, cf.\ Pelayo \& \vungoc\ \cite{PV2}. 
The Taylor series invariant $S^\infty$ associated to the focus-focus singularity $m$ is the Taylor series of the function $S$.
Note that $S(z)$ is normalized to satisfy $S(0)=0$ in Equation~\eqref{eqn:def-of-S}.

Though $S$ is not unique (see Remark~\ref{rmk:nonunique}), its Taylor series $S^\infty$ is uniquely defined up to the addition of integer multiples of $2\pi l$.
That is, letting $\R_0[[l,j]]$ denote the set of power series in the variables $l$ and $j$ with zero constant term, $S^\infty$ is uniquely determined as an element of $\R_0[[l,j]]/(2\pi l\, \Z)$. For the purposes of the calculations in this paper, we will take the representatives which satisfies $-\tfrac{\pi}{2}  \leq \partial_l S(0) < \tfrac{3\pi}{2}$. By performing this construction for all focus-focus singularities $m_1, \ldots, m_\nff$ of the system, we obtain the Taylor series \[S^\infty_1, \ldots, S^\infty_\nff\in \R_0[[l,j]]/(2\pi l\, \Z)\] corresponding to each singularity.

\begin{definition}
The Taylor series invariant associated to the simple semitoric system $(M,\om,F)$ is the $\nff$-tuple of Taylor series $(S^\infty_1, \ldots, S^\infty_\nff)$.
\label{def:taylorinv}
\end{definition}

\begin{remark}\label{rmk:nonunique}
The function $S(z)$ is given in Equation~\eqref{eqn:def-of-S}, but the function $S(z)$ is still 
not uniquely defined because it depends on certain choices encoded in the function $I(z)$.
First of all, $I(z)$ depends on a choice of complementary cycle $\gamma^z_2$ chosen so that $\{\gamma_L^z, \gamma_2^z\}$ generates the fundamental group of the torus $\Lambda_z$. Such a choice is not unique and changing the choice of $\gamma_2^z$ will change $S(z)$ by an integer multiple of $2\pi l$.
 This dependence of the function $S(z)$ on a choice of basis of the fundamental group of $\Lambda_z$ is related to a geometric interpretation of the twisting index, see Section~\ref{sec:geometricinterpret}.
 Furthermore, $S(z)$ also depends on the choice of a chart for the local normal form from Theorem~\ref{EliassonMZ}, and different choices of such charts change $I(z)$, and in turn $S(z)$, by adding on a function for which all derivatives vanish at $z=0$. Such functions are called \emph{flat functions}. This is why $S(z)$ is not well-defined, and instead we have to take its Taylor series $S^\infty$ which is invariant under the addition of flat functions.
\end{remark}

Following \vungoc~\cite{VuNgoc03},
the Taylor series $S^\infty$ associated with a focus-focus singularity is the only semiglobal invariant of a focus-focus fiber. That is, if two systems have a focus-focus singularity with the same Taylor series $S^\infty$, then there exists an isomorphism of integrable systems between semilocal neighborhoods of the respective singular fibers. Moreover,  every power series in two real variables $l,j$ with no constant term and with the linear term in $l$ lying in $[-\tfrac{\pi}{2},\tfrac{3\pi}{2}[$, appears as the Taylor series invariant of a focus-focus point.
Putting these two facts together, we obtain that the map from semilocal neighborhoods of focus-focus points up to isomorphisms to $\R_0[[l,j]]/(2\pi l \,\Z)$ is a bijection.

\begin{figure}[ht]
\includegraphics[trim=0 50 0 30,width=250pt]{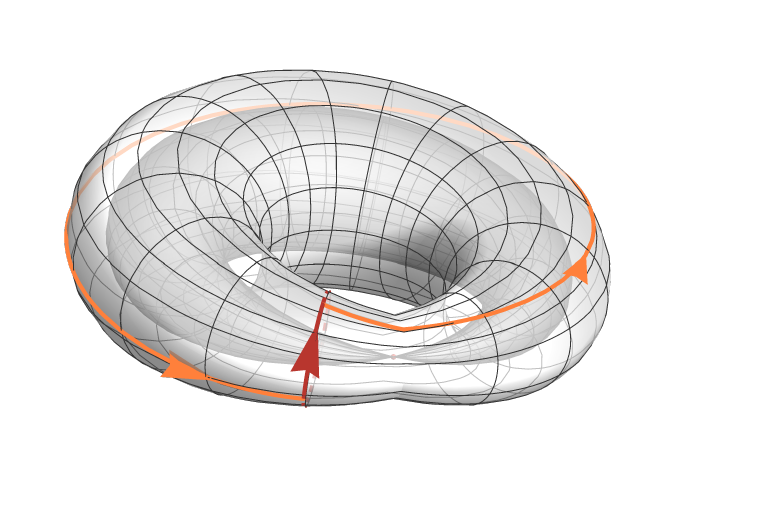}
\caption{The red curve is along the Hamiltonian flow of the $\mbS^1$-action generated by $L$ and the orange curve is along the Hamiltonian flow of $\Phi_2$. The value of $\tau_1(z)$ is the time spent along the flow of the red curve and $\tau_2(z)$ represents the time spent along the flow of the orange curve.}
\label{fig:torustorus}
\end{figure}

An alternative interpretation of the Taylor series invariant which does not directly make use of the function $I(z)$, also due to \vungoc\ \cite{VuNgoc03}, is as follows. Let $z\in\Phi(W)$ and let $p \in \Lam_z= \Phi^{-1}(z)$ be any point. Consider the closed curve constructed by following the Hamiltonian flow of $\mcX_{\Phi_2}$ until the $L$-orbit of $p$ is reached and then following the Hamiltonian flow of $\mcX_L$ to go back to $p$, see Figure \ref{fig:torustorus}. Let $\tau_1(z) \in \R/2\pi\Z$ be the time spent along the flow of $\mcX_L$ and $\tau_2(z) >0$ be the time spent along the flow of $\mcX_{\Phi_2}$. The functions $\tau_1(z)$ and $\tau_2(z)$ are independent of the choice of $p$.
Letting $\log$ be any choice of determination of the complex logarithm and taking a choice of lift of $\tau_1(z)$ to $\R$ discontinuous along the same branch cut, the functions
\begin{equation}
\begin{cases}
\sigma_1(z):= \tau_1(z) - \text{Im}(\log z), \\
\sigma_2(z):= \tau_2(z) + \text{Re}(\log z) 
\end{cases}
\label{eq:sigmas}
\end{equation} extend to smooth single-valued functions around the origin. 
Similar to Equation~\eqref{eqn:def-of-S}, subtracting logarithms is necessary to compensate for the discontinuity in the lift of $\tau_1$, which is unavoidable because of the monodromy, and for the rate at which $\tau_2(z)$ diverges to infinity as $z$ approaches the origin.
Moreover, the 1-form $\sigma := \sigma_1 \dee l + \sigma_2 \dee j$ is closed and therefore exact. Given such a $\sigma$, we take the function $S$ to be the unique smooth function that satisfies $\dee S=\sigma$ and $S(0)=0$. As before, there are choices encoded in $\sigma$, such as the choice of lift of $\tau_1(z)$, but the Taylor series of $S$ is well defined as an element of $\R[[l,j]]/(2\pi l\,\Z)$.


\subsection{The twisting index invariant}
\label{sss:twisting}
In this section, we summarize the construction of the twisting index invariant due to Pelayo \& \vungoc~\cite{PV2}.
Roughly speaking, the twisting index invariant is a label of $\nff$ integers assigned to each of the polygons $\De$ of the polygon invariant.
These integers will be obtained by comparing the semiglobal action coordinates from Equation~\eqref{eq:semigact} with the generalized toric momentum map $\mu$ of Section \ref{sss:polygon} around each focus-focus point.
We explain the details now.

Near each focus-focus singularity $m_r$, where $r\in \{1, \ldots, \nff\}$, we now describe how to construct a so-called \emph{privileged momentum map}.
Fix a choice of signs $\varepsilon = (\varepsilon_1,\ldots,\varepsilon_{n_{\text{FF}}}) \in (\Z_2)^{n_{\text{FF}}}$ and let $\tilde{W}:=F^{-1}\big(F(W)\backslash b^\varepsilon\big)$ be the set $W$ without the preimage of the cuts.
Recall the map $\Phi= (L,\Phi_2): W \to \R^2$ introduced in Section \ref{sss:taylor}. Let $\tp\colon \Phi(W)\to\R$ be the lift of $\tau_1$ to $\R$ that is continuous on $\Phi(\tilde{W})$ and also satisfies the condition that if we take $\ssp$ to be defined by Equation~\eqref{eq:sigmas} using $\tp$ then \begin{equation}\label{eqn:pref-tau1}
\ssp(0) =   \lim_{z\to 0} \Big(\tp(z) - \text{Im}(\log(z))\Big)  \in [-\tfrac{\pi}{2} ,\tfrac{3\pi}{2}[.
\end{equation}
Note that changing the choice of lift will change $\sigma_1(0)$ by integer multiples of $2\pi$, so such a $\tp$ always exists.

We define now the vector field
\begin{equation}
2 \pi \mcX^r := (\tp \circ \Phi) \mcX_L + (\tau_2 \circ \Phi) \mcX_{\Phi_2},
\label{eqHamvf}
\end{equation}which
 is smooth on $\tilde{W}$ and, by Pelayo \& V\~{u} Ng\d{o}c~\cite{PVNinventiones,PVNacta}, it turns out that there is a unique continuous function 
 \begin{equation}\label{eqn:Xi}
  \Xi_r:W \to \R
 \end{equation} which satisfies:
\begin{itemize}
 \item the function is smooth on $\tilde{W}$,
 \item the Hamiltonian vector field  of $\Xi_r$ on $\tilde{W}$ is $\mcX^r$,
 \item $\Xi_r(p)$ tends to $0$ as $p$ approaches $ m$.
\end{itemize}
The \emph{privileged momentum map at $m_r$} is defined by $\nu_r :=(L,\Xi_r)\colon W\to\R^2$ and it is smooth on $\tilde{W}$.

Let $f_\varepsilon$ be a choice of cartographic homeomorphism. Recall $\mu = f_\varepsilon\circ F$ and $\De = \mu (M)$.
Both $\mu$ and $\nu_r$ are continuous on $W$
and smooth on $\tilde{W}$, they both have $L$ as their first component, and on $\tilde{W}$ they both generate an effective Hamiltonian $2$-torus action.
Thus, there exists $\kappa_r^\Delta\in\Z$ such that $\mu$ and $\nu$ are related via 
\begin{equation}\label{eqn:twist-def}
\mu = T^{\kappa_r^\Delta} \circ \nu_r,
\end{equation}
where $T$ is as in Equation~\eqref{eqn:T}. The integer $\kappa_r^\Delta \in \Z$ is called the \emph{twisting index} of $\De$ at the focus-focus value $c_r$. Note that the integer $\kappa_r^\Delta$ depends on the choice of the cartographic homeomorphism $f_\varepsilon$, so the assignment of the integers $\kappa_r^\Delta$ differs for each representative of the semitoric polygon. Given any such set of integer labels, the labels
on the other representatives of the semitoric polygon are obtained via the following group action:
\begin{equation}
(\varepsilon', T^n) \cdot \Big(\De, (b_{\lam_r})_{r=1}^\nff,(\varepsilon_r)_{r=1}^\nff,(\kappa_r^\Delta)_{r=1}^\nff\Big)\, = \left(t_u(T^n(\De)), (b_{\lam_r})_{r=1}^\nff,(\varepsilon_r'\varepsilon_r)_{r=1}^\nff,\left(\kappa_r^\Delta+n+\sum_{i=1}^{r}u_i\right)_{r=1}^\nff\right),
\label{eq:twisact}
\end{equation} for $(\varepsilon',T^n) \in (\Z_2)^\nff \times \mcG$ where again $u_r = \frac{1}{2}(\varepsilon_r - \varepsilon_r \varepsilon_r')$.

\begin{remark}\label{rmk:twist-changes}
There are two differences between Equation~\eqref{eq:twisact} and the original equation for this group action contained in Pelayo \& \vungoc~\cite{PVNinventiones}, both contained in the extra term $\sum_{i=1}^{r}u_i$ which does not appear in Pelayo \& \vungoc~\cite{PVNinventiones}.
\begin{itemize} 
 \item The first $r-1$ terms of the sum, $u_1+\ldots+u_{r-1}$, had to be added to this group action for the following reason: 
 	changing the cut direction at any focus-focus point with smaller $x$-component than $m_r$ will have the effect of applying $T$ to the generalized toric momentum map near $m_r$, but will not change the privileged momentum map. 
 	Therefore, such a change in cut direction changes the twisting index; see Example~\ref{ex:twist-change} for an explicit example of this.
 	Since the twisting index of a system with at least two focus-focus points had never been computed until the present paper, this oversight has no impact on the existing literature. 
 	This was noticed independently by both Yohann Le Floch and the authors of the present paper. 
 \item The last term $u_r$ of the summation comes from a change in convention compared to Pelayo \& \vungoc~\cite{PVNinventiones}, which is independent of the previous item. In the original formulation, the definition of the preferred momentum map near each focus-focus point was slightly different depending on whether the corresponding cut was upwards or downwards. Specifically, they take the argument of $\tau_1$ to be in different intervals depending on whether the corresponding cut was up or down, taking it in $[0,2\pi\,[$ if the cut was up and $[\pi,3\pi\,[$ if the cut was down.
 In this paper, as suggested to us by Yohann Le Floch, we have opted to take a more unified approach, and the construction of $\nu$ presented in the current section does not depend on the cut direction, which comes at the cost of the appearance of the term $u_r$. 
 \end{itemize}
 \end{remark}

\begin{definition}
The \emph{twisting index invariant} of the simple semitoric system $(M,\om,F)$ is the $(\Z_2)^\nff \times \mcG$-orbit of the quadruple $\left(\De, (b_{\lam_r})_{r=1}^\nff,(\varepsilon_r)_{r=1}^\nff,(\kappa_r^\Delta)_{r=1}^\nff\right)$ constructed as above with group action as given in Equation~\eqref{eq:twisact}.
\end{definition}

Thus, the twisting index can be thought of as an assignment of an integer to each focus-focus point
for each choice of semitoric polygon.
For $r\in\{1,\ldots, n_{\text{FF}}\}$, given a representative of the semitoric polygon $\De$, we call the associated $\kappa_r^\Delta$ the \emph{twisting index of $m_r$
relative to $\De$}.

\begin{remark}\label{rmk:Xi-and-xi} 
Notice that since the Hamiltonian vector field of $\Xi_r$ is $\mcX^r$, the function $\Xi_r$ can be written as $\Xi_r = \xi_r \circ \Phi$, where $\xi_r$ is determined by 
\begin{equation}
\label{eqn:xi_deriv}
	\frac{\partial \xi_r}{\partial l}(z) = \dfrac{1}{2\pi} \tp(z),\qquad \frac{\partial \xi_r}{\partial j}(z) = \dfrac{1}{2\pi} \tau_2(z)
\end{equation}
 and $\displaystyle \lim_{z \to 0} \xi_r (z) =0$.
\end{remark}

\begin{example}\label{ex:twist-change}
 Consider the two polygon representatives shown in Figure~\ref{fig:twist-change}, which differ by the action of $(\ep'=(-1,1), T^0)$, which changes the cut direction at the first focus-focus point $m_1$. Let $\mu$ be the generalized toric momentum map for the polygon on the left, and $\mu'$ be the generalized toric momentum map for the polygon on the right.
 Since $\mu' = t_{b_{\lam_1}}\circ \mu$, near the second focus-focus point $m_2$, the maps $\mu$ and $\mu'$ differ by translation and an application of $T$. Since the privileged momentum map $\nu_2$ does not depend on the choice of polygon, we conclude that $\mu = T^{\kappa_2^\Delta}\circ\nu_2$ implies that $\mu' = T^{\kappa_2^\Delta+1}\circ\nu_2$. Thus, we see that changing the cut direction at $m_1$ increases the twisting index at $m_2$ by 1.
\end{example}

\begin{figure}
\centering
    \begin{subfigure}[b]{0.35\textwidth}
    	\centering
        \includegraphics[width=0.7\textwidth]{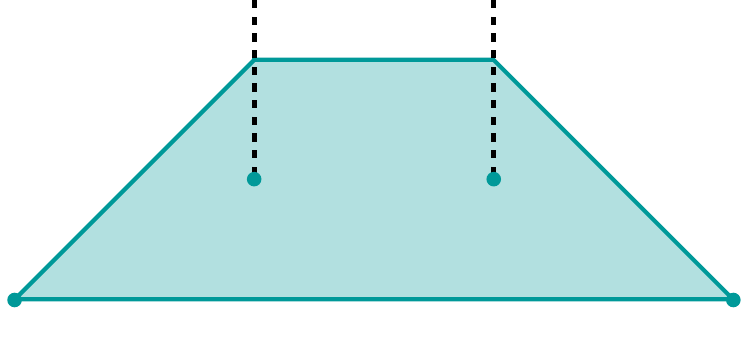}
        \caption*{}
    \end{subfigure}
    \begin{subfigure}[b]{0.35\textwidth}
    	\centering
        \includegraphics[width=0.7\textwidth]{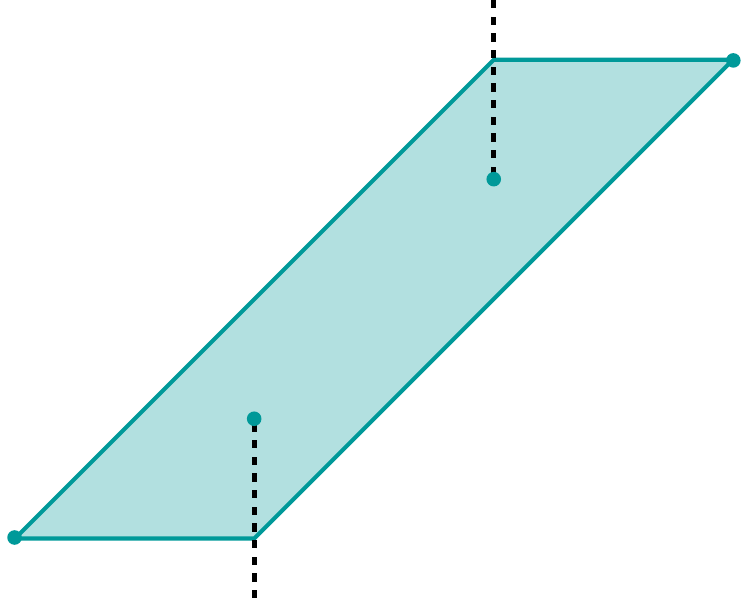}
        \caption*{}
    \end{subfigure}
    \caption{Two representatives of the same semitoric polygon related by changing the cut direction at the leftmost focus-focus point.}
    \label{fig:twist-change}
\end{figure}

\subsection{The classification}
Pelayo \& V\~{u} Ng\d{o}c~\cite{PVNinventiones} described how to obtain the five invariants from a given simple semitoric system.
Then, in~\cite{PVNacta}, they also described exactly which data can appear as invariants of a semitoric system, and call the abstract list of such data a \emph{list of semitoric ingredients}.
Furthermore, they proved that these five invariants completely classify semitoric integrable systems:
\begin{theorem}[Pelayo $\&$ V\~{u} Ng\d{o}c~\cite{PVNinventiones, PVNacta}]\label{thm:PVN-classification}
\mbox{}
 \begin{enumerate}
  \item 
  Two simple semitoric systems are isomorphic if and only if they have the same five semitoric invariants.
  \item 
  Given a list of semitoric ingredients there exists a simple semitoric system which has those as its five invariants.
 \end{enumerate}
\end{theorem}

Theorem~\ref{thm:PVN-classification} thus implies that there is a natural bijection between isomorphism classes of simple semitoric systems and lists of semitoric ingredients.

\section{Geometric interpretations of twisting index}
\label{sec:geometricinterpret}

The primary goal of this section is to give an equivalent formulation of the twisting index in terms of comparing homology cycles near a focus-focus point to those which arise from a choice of semitoric polygon. 
The main results of this section are collected in Theorem~\ref{thm:geometry}.

\subsection{The local and semilocal maps}

First let us briefly review the relevant maps, as described in Section~\ref{ss:invariants}. Each focus-focus point $m_r$ of the system has a neighborhood $U_r$ which is isomorphic as an integrable system to a neighborhood $V_r$ of the origin in the system on $\R^4$ given by
\[
 G(q_1,q_2,p_1,p_2)=(q_1p_2-q_2p_1,q_1p_1+q_2p_2)
\]
via an isomorphism $(\phi_r,\varrho_r)$.
Let $W_r = F^{-1}(F(U_r))$. We can then define a map $\Phi_r\colon W_r\to\R^2$ via $\Phi_r = \varrho_r \circ F$, which is of the form $\Phi = (L,\Phi_2)$ and essentially amounts to extending the momentum maps of the local normal form for $m_r$ to a neighborhood of the fiber $F^{-1}(F(m_r))$.

Given a $z\in\Phi_r(W_r)$, we consider the fiber $\Lambda_z = \Phi_r^{-1}(z)$. An action variable $I_r$ is obtained by integrating a primitive $\varpi$ of $\omega$ along a cycle in $\Lambda_z$. There are choices of such cycles, and therefore various options for $I_r$. Any such $I_r$ induces a momentum map $\mu_r^{I_r}$. Due to the monodromy around the focus-focus value $F(m_r)$, such a $I_r$ and $\mu_r^{I_r}$ cannot be smoothly defined in all of $U_r$ or $W_r$, respectively.

There is a preferred choice of action and momentum map. These are the action $\xi_r$ and corresponding momentum map $\Xi_r$ discussed in Section~\ref{sss:twisting} and in particular Remark~\ref{rmk:Xi-and-xi}.

The local maps are shown in Figure~\ref{fig:diagram-local}, and the semilocal maps are shown in Figure~\ref{fig:diagram-semilocal}.

Later we will see that for the specific system we study in Section~\ref{sec:specificexample} there is a relationship between the map $\varrho_r$ and a quantity called imaginary action, see Remark~\ref{rmk:varrho_is_J}.

\begin{figure}
\centering
\begin{subfigure}[b]{.45\linewidth}
\centering
\[ \begin{tikzcd}
      M\arrow[r, phantom, sloped, "\supset"]&[-2em] U_r \arrow{r}{\phi_r} \arrow[d,"F"] \arrow[dr,"\Phi_r"]  &  V_r   \arrow[d,"G"] &[-2em] \R^4 \arrow[l,phantom,sloped,"\subset"]\\
      \R^2 \arrow[r, phantom, sloped, "\supset"] &F(U_r) \arrow{r}{\varrho_r}   &  \R^2 \\[-1em]
      & (l,h)\arrow[u, phantom, sloped, "\in"] & (l,j)\arrow[u, phantom, sloped, "\in"] &
\end{tikzcd}
\] 
\caption{The local maps defined in a neighborhood of $m_r$.}
\label{fig:diagram-local}
\end{subfigure} 
\hspace{20pt}
\begin{subfigure}[b]{.45\linewidth}
\centering
\[ \begin{tikzcd}
      M\arrow[r, phantom, sloped, "\supset"] &[-2em] W_r \arrow[bend left = 20, looseness = 1]{rrrd}{\mu_r^{I_r} := I\circ \Phi_r} \arrow[d,"F"] \arrow[dr,"\Phi_r"]  &  &[-2em] & \\
      \R^2 \arrow[r, phantom, sloped, "\supset"] &F(W_r) \arrow{r}{\varrho_r}   &  \R^2 \arrow[r, phantom, sloped, "\simeq"] & \C \arrow[r,"I_r"] & \R\\[-1em]
      & (l,h)\arrow[u, phantom, sloped, "\in"] & (l,j)\arrow[u, phantom, sloped, "\in"] & z \arrow[u, phantom, sloped, "\in"] &
\end{tikzcd}
\] 
\caption{The semilocal maps defined in a neighborhood of $F^{-1}(F(m_r))$.}
\label{fig:diagram-semilocal}
\end{subfigure}
 \caption{A review of the relevant maps for constructing the invariants. We typically use coordinates $(l,h)$ for the image of $F$ in $\R^2$, and coordinates $(l,j)$ for the image of $\Phi_r$ in $\R^2$, or $z = l + ij$ when viewing $\R^2$ as $\C$.}
\end{figure}

\subsection{Geometric interpretations}

Throughout this section we will often make use of the following fact:
any free continuous $\mbS^1$-action on a connected manifold $N$ determines a well-defined cycle in
$H_1(N)$ by taking the orbit of any point $p\in N$.
We denote the orbit through $p$ by $\mbS^1\cdot p$.
The resulting cycle is well-defined in $H_1(N)$ because if $p,q\in N$ then
any path $\gamma\colon [0,1]\to N$ with $\gamma(0)=p$ and $\gamma(1)=q$ determines a
homotopy between the orbit through $p$ and the orbit through $q$ by $\mbS^1\cdot (\gamma(s))$ for $0\leq s \leq 1$.

Let $(M,\om,F)$ be a simple semitoric system and let 
 $(\De,b,\varepsilon)$ be a representative of its semitoric polygon. 
Let $M^\varepsilon := M\setminus F^{-1}(b^\varepsilon)$ denote the manifold $M$ without 
the preimage of the cuts.
Recall, as in Section~\ref{sss:polygon}, that there is a generalized toric momentum map $\mu_\Delta=(L,\mu_2^\Delta)\colon M\to \R^2$ such that
\begin{itemize}
 \item $\mu_\Delta$ is continuous on $M$ and smooth on $M^\varepsilon$,
 \item $\mu_\Delta(M)=\De$,
 \item the Hamiltonian flows of $L$ and $\mu_2^\Delta$ generate an effective $\mathbb{T}^2$-action on $M^\varepsilon$.
\end{itemize}
Roughly speaking, the twisting index measures the difference between
$\mu_\Delta$ and the dynamics near a given focus-focus point, which can be expressed in several different, but equivalent, ways.

Throughout this section, let $m:=m_r$ be a focus-focus point, and we will drop the subscript $r$ from all notation.
Following the construction in Section~\ref{sss:taylor}, we obtain a neighborhood $W$ of the focus-focus fiber and
a map $\Phi\colon W \to \C$.
As before, let $\Lambda_z :=\Phi^{-1}(z)$, which is a torus for $z\neq 0$ and a pinched torus for $z=0$.
For $z\neq 0$, the Hamiltonian flow of $L$ generates a free $\mbS^1$-action on $\Lambda_z$, and therefore determines a cycle in $H_1(\Lambda_z)$ which we denote by $\gamma^z_L$.
Now consider $\tilde{W} = W\setminus F^{-1}(b^\varepsilon)$ and for $z\in\Phi(\tilde{W})$ denote
by $\gamma^z_\Delta$ the cycle determined by the flow of $\mu_2^{\Delta}$ on $\Lambda_z$. 
Notice that $\{\gamma_L^z,\gamma_\Delta^z\}$ form a basis of $\Lambda_z$ for any $z\in\Phi(\tilde{W})$.
Analogously to Equation~\eqref{eq:semigact},
we integrate the primitive $\varpi$ of $\omega$ over this cycle to determine an action $I_\Delta\colon \Phi(\tilde{W})\to\R$ via
\begin{equation}\label{eqn:I-Delta}
 I_\Delta(z) := \frac{1}{2\pi}\int_{\gamma_\Delta^z}\varpi.
\end{equation}
Note that $I_\Delta(z)$ is not defined when $z=0$, but the limit as $z$ goes to zero exists, so we denote $I_\Delta(0) := \lim_{z\to 0}I_\Delta(z)$.
Due to the discussion in Remark~\ref{rmk:integrate-cycles}, we have that $I_\Delta \circ \Phi=\mu_2^\Delta$.

Now we let
\begin{equation}\label{eqn:S-Delta}
S^\Delta(z) := 2\pi I_\Delta(z)-2\pi I_\Delta(0) + \text{Im}(z\log z - z) \in \R_0[[l,h]]
\end{equation}
and obtain its Taylor series $(S^\Delta)^\infty$.
Note that, in Section~\ref{sss:taylor}, the Taylor series $S^\infty$ defined via Equation~\eqref{eqn:def-of-S} is only
well-defined up to the addition of integer multiples of $2\pi l$.
On the other hand, in the present section the polygon $\De$ and associated generalized toric momentum map $\mu^\Delta$ give a \emph{preferred choice}
of cycle $\gamma_\Delta^z$ to use when computing $I_\Delta(z)$ in Equation~\eqref{eqn:I-Delta}.
Thus, the Taylor series $(S^\Delta)^\infty$ is completely
determined in the process described in this section.

We now want to compare this to another preferred Taylor series near the focus-focus point $m$, which is related to the preferred momentum map $\nu = (L,\Xi)$ from Section~\ref{sss:twisting}.
Let $\Xi\colon W\to\R$ be as in Equation~\eqref{eqn:Xi}, and as in Remark~\ref{rmk:Xi-and-xi}, there exists a map $\xi\colon \Phi(W)\to \R$ such that $\xi \circ\Phi = \Xi$.

\begin{lemma}\label{lem:Spref}
 Let $(S^{\text{pref}})^\infty\in\R_0[[l,j]]$ denote the representative of $S^\infty$ in which the coefficient of $l$ lies in $[-\frac{\pi}{2},\frac{3\pi}{2}[$. 
  Then $(S^{\text{pref}})^\infty$ is equal to the Taylor series of \[\hat{S}(z) := 2\pi \xi(z) - 2\pi \xi(0) + \textup{Im}(z \log z -z),\]
  where $\xi \circ\Phi = \Xi$ as in Remark~\ref{rmk:Xi-and-xi}.
\end{lemma}

\begin{proof}
 Let $(\hat{S})^\infty$ denote the Taylor series of $\hat{S}$ at the origin.  Then $(S^{\text{pref}})^\infty, (\hat{S})^{\infty}\in\R[[l,j]]$ are representatives of the Taylor series $S^\infty\in\R[[l,j]]/(2\pi l\Z)$, so
 \[
  (S^{\text{pref}})^\infty - (\hat{S})^\infty \in 2\pi l \Z.
 \]
 Thus, to show that $(S^{\text{pref}})^\infty = (\hat{S})^\infty$, and therefore complete the proof, it is sufficient to show that the coefficient of $l$ in the series $(\hat{S})^\infty$ lies in the interval $[\frac{-\pi}{2},\frac{3\pi}{2}[$.
 Recall from \eqref{eqn:xi_deriv} that $\frac{\partial \xi}{\partial l}= \frac{1}{2\pi} \tp$. By direct calculation from $z =l+ij$,
it can be shown that $\frac{\partial }{\partial l}(\text{Im}(z\log(z)-z))=\text{Im}(\log(z))$. 
 Now we compute
 \begin{align*}
 \frac{\partial }{\partial l}\hat{S}(z) &= \frac{\partial}{\partial l}\left(2\pi\xi(z) - 2\pi \xi(0) + \text{Im}(z\log(z)-z)\right)\\
 &= \tp(z) - \frac{\partial}{\partial l}\big( \text{Im}(z\log(z)-z)\big)\\
 &= \tp(z) - \text{Im}(\log(z)).
\end{align*}
Since the $l$ coefficient of the Taylor series $(\hat{S})^\infty$ is given by $\frac{\partial\hat(S)}{\partial l}|_{z=0}$, by \eqref{eqn:pref-tau1}, the proof is complete.
\end{proof}

\begin{lemma}\label{lem:twist-action}
The functions $I_\Delta(z)$ and $\xi(z)$, wherever both are defined, are related by 
$$I_\Delta(z) - \xi(z) = \kappa^\Delta l,$$
where $z = l + ij$.
\end{lemma}
\begin{proof}
	The first component of both $\mu_\Delta$ and $\nu$ is $L$, and therefore $T^{\kappa^\Delta}\nu = \mu_\Delta$ is equivalent to $\kappa^\Delta L = \mu_2^\Delta-\Xi$. This implies the result.
\end{proof}

\begin{proposition}\label{prop:difference-Taylor}
Let $r\in\{1,\ldots, n_{\text{FF}}\}$. Write briefly $ m_r=m$ and drop below the subscript $r$ from all notation.
Fix a semitoric polygon $(\Delta,b, \varepsilon)$ of $(M,\om,F)$ and let:
\begin{itemize}
 \item $(S^\Delta)^\infty$ denote the Taylor series obtained from Equation~\eqref{eqn:S-Delta}, 
 \item $(S^{\text{pref}})^\infty$ denote the representative of $S^\infty$ in which the coefficient of $l$ lies in $[\-\frac{\pi}{2},\frac{3\pi}{2}[$,
 \item $\kappa^\Delta$ denote the twisting index of $m$ relative to the semitoric polygon representative
$\Delta$.
\end{itemize} 
Then $\kappa^\Delta = \frac{1}{2\pi l}\left((S^{\Delta})^\infty-(S^{\text{pref}})^\infty \right)$.
\end{proposition}

\begin{proof}
Let $\mu_\De = (L, \mu_2^\Delta)$ be the generalized toric momentum map satisfying $\mu(M) = \De$, and let $I_\Delta\colon \Phi(\tilde{W})\to \R$ be as in Equation~\eqref{eqn:I-Delta}, so that $I_\Delta \circ \Phi = \mu_2^\Delta$.
Recall that, as in Section~\ref{sss:twisting}, near $m$ there is a preferred local momentum
map $\nu = (L,\Xi)$, which is smooth on $\tilde{W}$, such that
 $
  T^{\kappa^\Delta}\nu = \mu_\Delta.
 $
 Furthermore, as in Remark~\ref{rmk:Xi-and-xi}, there is a map $\xi\colon \Phi(W)\to\R$ such that $\xi\circ\Phi = \Xi$.

Using Lemma~\ref{lem:Spref}, Equation~\eqref{eqn:S-Delta}, and Lemma~\ref{lem:twist-action}, we conclude that
\[
 (S^\Delta)^\infty - (S^{pref})^\infty = 2\pi (I_\Delta - \xi)^\infty = 2\pi \kappa^\Delta l,
\]
where $(I_\Delta - \xi)^\infty$ denotes the Taylor series of $I_\Delta(z) - \xi(z)$ expanded at the origin.
\end{proof}

The result of Proposition~\ref{prop:difference-Taylor} is similar to how the twisting index was treated in~\cite{LFVN21,PPT,Jaume-thesis}, by packaging it along with the Taylor series.

Let $\gamma_{\text{pref}}^z\in H_1(\Lambda_z)$ be the cycle defined by following $\mathcal{X}_{\Phi_2}$ for time $\tau_2(z)$ and following $\mathcal{X}_L$ for time $\tp(z)$. 
In other words, $\gamma_\text{pref}^z$ is the piecewise smooth loop shown in Figure~\ref{fig:torustorus}.

Now we will show that $\gamma_\text{pref}^z$ is homotopic to $\gamma_{\Xi}^z$, which is the cycle determined by the flow of $\Xi$. For $s\in [0,1]$, define a vector field $\mathcal{X}(s)$ on $\Lambda_z$ by:
\[
 2\pi \mathcal{X}(s) = s(\tp\circ\Phi)\mathcal{X}_L + (\tau_2\circ\Phi)\mathcal{X}_{\Phi_2}.
\]
Let $\gamma^z(s)$ be the cycle determined by flowing along $\mathcal{X}(s)$ for time $2\pi$, and then flowing along $\mathcal{X}_L$ for time $(1-s)\tp(z)$. Then $\gamma^z(s)$ is a loop for all $s\in [0,1]$, $\gamma^z(0)=\gamma_\text{pref}^z$, and $\gamma(1) = \gamma^z_{\Xi}$.

Thus, applying this equivalence and Remark~\ref{rmk:integrate-cycles}, we have that
\begin{equation}\label{eqn:gamma-pref-Xi}
\frac{1}{2\pi} \int_{\gamma_\text{pref}^z}\varpi = \frac{1}{2\pi} \int_{\gamma_\Xi^z}\varpi = \xi(z).
\end{equation}

\begin{remark} \label{re:xi_as_action}
Because of \eqref{eqn:gamma-pref-Xi}, the function $\xi(z)$ can be understood as a \emph{preferred local action} around the focus-focus value.	
\end{remark}

\begin{proposition}\label{prop:cycles}
 $\gamma_\Delta^z - \gamma_\text{pref}^z = \kappa^\Delta \gamma_L$.
\end{proposition}

\begin{proof}
 Since $\{\gamma_L^z, \gamma_\Delta^z\}$ and $\{\gamma_L^z, \gamma_\text{pref}^z\}$ are each a basis of $H_1(\Lambda_z)\cong \Z^2$, there exists $K \in\Z$ such that $\gamma_\Delta^z - \gamma_\text{pref}^z = K \gamma_L$.
 By Equation~\eqref{eqn:gamma-pref-Xi},  $\gamma_\Delta^z - \gamma_\tau^z = \ell \gamma_L^z$ implies that 
\[
 I_\Delta(z) - \xi(z) = \frac{1}{2\pi} \int_{\gamma_\Delta^z} \varpi - \frac{1}{2\pi} \int_{\gamma_\text{pref}^z} \varpi = \frac{1}{2\pi} \int_{\gamma_\Delta^z-\gamma_\text{pref}^z} \varpi = \frac{1}{2\pi} \int_{K \gamma_L^z} \varpi = K l
\]
again applying Remark~\ref{rmk:integrate-cycles} in the last equality, where $z = l+ij$.
Hence, $I_\Delta(z) - \xi(z) = K l$, so $K = \kappa^\Delta$ by Lemma~\ref{lem:twist-action}.
\end{proof}

\begin{figure}[h]
 \centering
 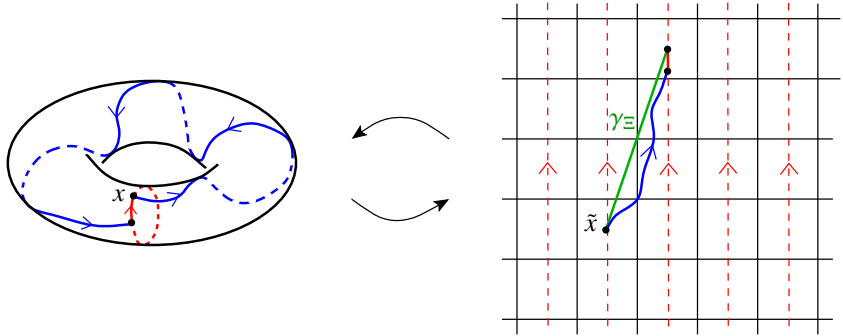
 \caption{A lift to the universal cover (displayed on the right) of the concatenated flows of $\mcX_L$ and $\mcX_{\Phi_2}$ (displayed on the left). The point $\tilde{x}$ is the corresponding lift of $x$. In the set-up here, the lifted flow of $\mcX_L$ is moving vertically. For comparison, on the universal cover, we also show $\gamma_\Xi$ (green), which is a straight line.}
 \label{lifting}
\end{figure}

\section{A two-parameter family with two focus-focus points}
\label{sec:specificexample}

In this section, we compute the symplectic invariants of the  family of simple semitoric systems given in Equation \eqref{eqn_ssys}, which can have up to two focus-focus singularities. When this is the case, the Taylor series invariant (Section \ref{sss:taylor}), the height invariant (Section \ref{sss:height}) and the twisting index invariant (Section \ref{sss:twisting}) each have two components.

\begin{remark}
The one-parameter family of systems given in Equation~\eqref{eqn_ssys} is a special case of 
the four-parameter family of systems studied by Alonso \& Hohloch \cite{AH2}, which in turn is a special
case of the broad six-parameter family of systems (sometimes called {\it generalized coupled angular momenta}) studied by Hohloch \& Palmer \cite{HoPa2018}.
In particular, 
the system in Equation~\eqref{eqn_ssys} is a particular case of the family of systems studied 
in Hohloch \& Palmer \cite{HoPa2018} with parameters
$$ \begin{aligned}
R_1=1,&\hspace{1.5cm}&&t_1 = (1-s), &\hspace{1.5cm}& t_3 = 2s(1 - s), \\
R_2=2,&&&t_2 = s,  && t_4 =0,
\end{aligned} $$
and of the family of systems studied 
in Alonso \& Hohloch \cite{AH2} with parameters 
$$ R_1 = 1,\qquad R_2=2, \qquad s_1 = 0, \qquad s_2=s.$$
\label{re:otherpapers}
\end{remark}

Since the system in Equation~\eqref{eqn_ssys} is a special case of the system from
Hohloch \& Palmer \cite{HoPa2018}, then as a consequence of 
Hohloch \& Palmer \cite[Theorem 1.1]{HoPa2018}, for any $s\in [0,1]$ the 
%
system in equation \eqref{eqn_ssys} is a completely integrable system with four singularities of rank 0, located at the product of the poles of the spheres. We will denote these points by $\mathcal{N} \times \mathcal{N}$, $\mathcal{N} \times \mathcal{S}$, $\mathcal{S} \times \mathcal{N}$ and $\mathcal{S} \times \mathcal{S}$, where $\mathcal{N},\mathcal{S}$ denote the North and South poles respectively.

As a consequence of Alonso $\&$ Hohloch \cite[\textit{Prop. 8, Prop. 9, Thm. 11}]{AH2}, we have:
\begin{proposition}
\label{splusmin}
The system \eqref{eqn_ssys} is semitoric for all values of $s \in [0,1] \backslash \{s_-, s_+\}$, where
\begin{equation*}
s_+ = \dfrac{1}{16}\left( 8 - 3 \sqrt{2} + \sqrt{82 + 16 \sqrt{2}} \right), \qquad s_- = \dfrac{1}{16}\left( 8 + 3 \sqrt{2} - \sqrt{82 - 16 \sqrt{2}} \right).
\end{equation*} The points $\mathcal{N} \times \mathcal{S}$ and $\mathcal{S} \times \mathcal{N}$ are focus-focus if $s_- < s < s_+$ and elliptic-elliptic if $s<s_-$ or $s>s_+$. The points $\mathcal{N} \times \mathcal{N}$ and $\mathcal{S} \times \mathcal{S}$ are always elliptic-elliptic. This situation is displayed in Figure~\ref{fig:momimg}.
\label{prop:nff}
\end{proposition}

In Figure \ref{fig:momimg}, we can observe how the image of the momentum map $F$ evolves as we change the parameter $s$: two of the singular values move from the boundary to the interior and back to the boundary. This corresponds to the transition of the singular points from elliptic-elliptic to focus-focus and focus-focus to elliptic-elliptic respectively, which are Hamiltonian-Hopf bifurcations.

\begin{figure}
\includegraphics[width=430pt]{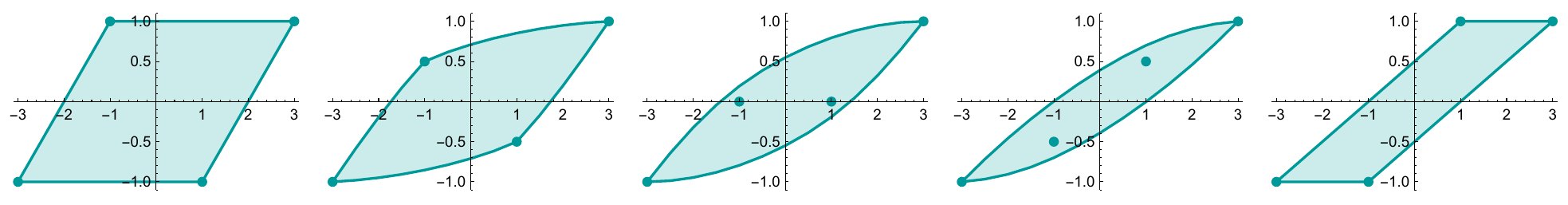}
\caption{
The image $B_s :=F_s(M)$ of the momentum map as the parameter moves from $s=0$ on the left to $s=1$ on the right. 
Notice that two of the singular points start as elliptic-elliptic at $s=0$, become
focus-focus at $s_- \approx 0.284$, and transition back into elliptic-elliptic  at $s_+ \approx 0.874$.}
\label{fig:momimg}
\end{figure}

\subsection{Polygon and height invariants}

The polygon invariant and the height invariant of the system in Equation \eqref{eqn_ssys} have been already calculated by Alonso \& Hohloch \cite{AH2}. In this subsection, we recall those results specialized to the parameter values we are interested in for this paper.

\begin{theorem}[{Alonso \& Hohloch \cite[Theorem 2]{AH2}}] 
\label{thm:poly-invt}The number of focus-focus points invariant and the polygon invariant of the system from Equation~\eqref{eqn_ssys} are as follows (see also Figure~\ref{fig:pols}):
\begin{itemize}
	\item For $s < s_-$, we have $\nff=0$ and the polygon invariant is the equivalence class of the polygon which is the convex hull of $(-3,-1)$, $(-1,1)$, $(1,-1)$, and $(3,1)$.
	\item For $s_- < s < s_+$, we have $\nff=2$ and the polygon invariant is the equivalence class of $\big(\De, (b_{\lam_r})_{r=1}^2,(\varepsilon_r)_{r=1}^2\big)$ where $\lam_1 = -1$, $\lam_2=1$, $\varepsilon_1 = 1$, $\varepsilon_2=1$, and $\De$ is the convex hull of $(-3,-1)$, $(-1,1)$, $(1,1)$, and $(3,-1)$.
	\item For $s > s_+$, we have $\nff=0$ and the polygon invariant is the equivalence class of the polygon which is the convex hull of $(-3,-1)$, $(-1,-1)$, $(1,1)$, and $(3,1)$.
\end{itemize}
\label{thm:poly}
\end{theorem}

\begin{figure}[ht]
\centering
    \begin{subfigure}[b]{0.4\textwidth}
        \includegraphics[width=.8\textwidth]{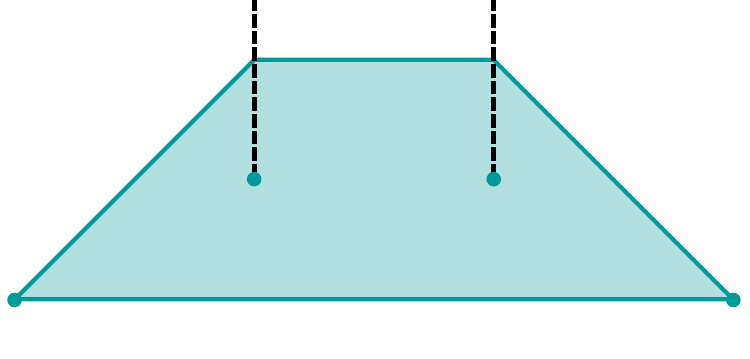}
        \caption{\small \textit{$(\varepsilon_1,\varepsilon_2)=(+1,+1)$}}
        \label{fig:pols00}
    \end{subfigure}
    \hspace{0.1\textwidth}
    \begin{subfigure}[b]{0.4\textwidth}
        \includegraphics[width=.8\textwidth]{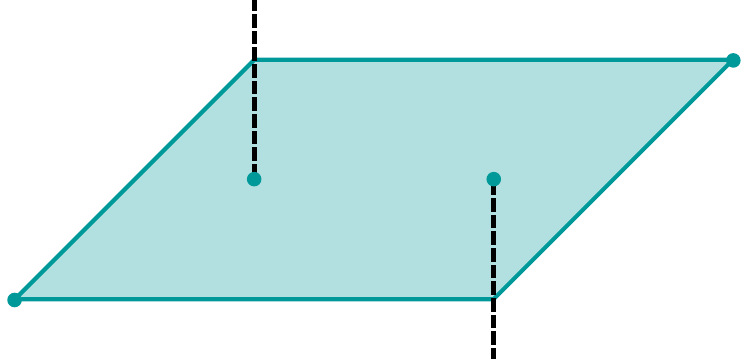}
        \caption{\small \textit{$(\varepsilon_1,\varepsilon_2)=(+1,-1)$}}
        \label{fig:pols10}
    \end{subfigure}    
    \newline
    \begin{subfigure}[b]{0.4\textwidth}
        \includegraphics[width=.8\textwidth]{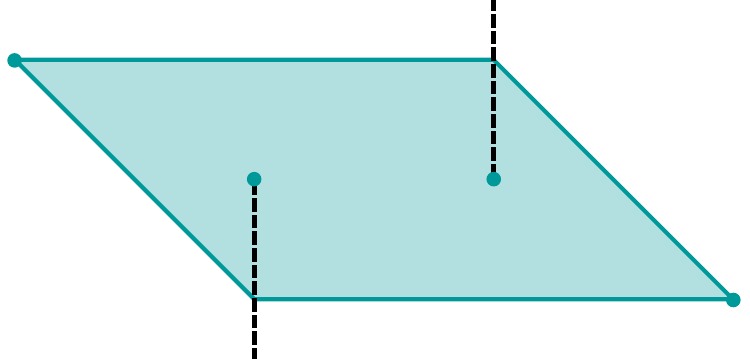}
        \caption{\small \textit{$(\varepsilon_1,\varepsilon_2)=(-1,+1)$}}
        \label{fig:pols01}
    \end{subfigure}
    \hspace{0.1\textwidth}    
    \begin{subfigure}[b]{0.4\textwidth}
        \includegraphics[width=.8\textwidth]{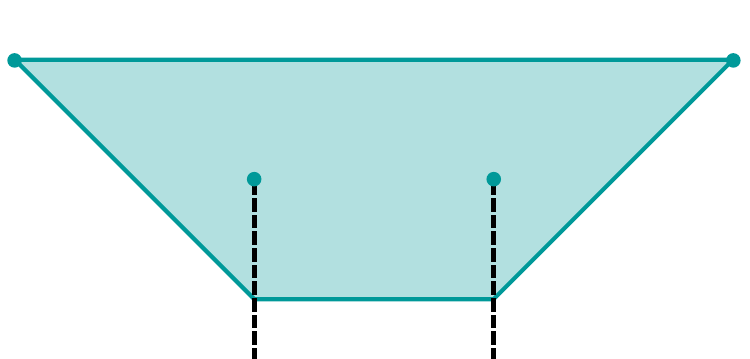}
        \caption{\small \textit{$(\varepsilon_1,\varepsilon_2)=(-1,-1)$}}
        \label{fig:pols11}
    \end{subfigure}
\caption{\small{Representatives of the polygon invariant discussed in Theorem~\ref{thm:poly-invt} for the case $s_-<s<s_+$ corresponding to different choices of signs $\varepsilon=(\varepsilon_1,\varepsilon_2)$. The cuts are indicated by dashed lines.
For $s<s_-$ the polygon invariant is represented by the polygon in Figure~\ref{fig:pols10} but without the cuts, and for 
$s>s_+$ the polygon invariant is represented by the polygon in Figure~\ref{fig:pols01} but without the cuts.
Comparing with Figure~\ref{fig:momimg}, notice that at $s=0$ the image of the system is the same polygon as in
Figure~\ref{fig:pols10} (again without the cuts), and at $s=1$ the image of the system is a polygon which is equivalent, via a global application of $T^{-1}$, to the polygon shown in Figure~\ref{fig:pols01} (again without the cuts).}}
\label{fig:pols}
\end{figure}

Let $u\colon \R \to \{0,1\}$ denote the usual Heavyside step function,
\[
 u(t) := \begin{cases} 0, & \text{if }t \leq 0,\\ 1, & \text{if }t > 0.\end{cases}
\]

Also, for $s_- < s < s_+$ we define
\begin{equation}
\label{eqn:rho}
\begin{aligned}
 \rho_1 &:= \sqrt{4 - 12 s + 13 s^2 - 8 s^3 + 4 s^4} , \\ 
 \rho_2 &:= \sqrt{-4 + 12 s + 23 s^2 - 64 s^3 + 32 s^4}.
\end{aligned}
\end{equation}
Note that
$s_-$ and $s_+$ are precisely the only two roots for $s\in [0,1]$ of the polynomial
\begin{equation}\label{eqn:sminus-plus}
-4 + 12 s + 23 s^2 - 64 s^3 + 32 s^4.
\end{equation}
Thus, the argument of $\rho_2$ is strictly positive when $s_- < s < s_+$.
Furthermore, the argument of $\rho_1$ is strictly positive for all $s\in\R$.

\begin{theorem}[{Alonso \& Hohloch \cite[Theorem 22]{AH2}}]
The height invariant $(h_1(s),h_2(s))$ associated to the system \eqref{eqn_ssys} for $s_- < s < s_+$ is given by 
\begin{equation*}
\begin{aligned}
h_1(s) &= -\dfrac{1}{2\pi} \mcF(s) + 2u (2-3s),  \\
h_2(s) &= 2 - h_1(s),
\end{aligned}
\end{equation*} where
\begin{equation*}
\begin{aligned}
 \mcF(s) := &-4 \arctan \left( \dfrac{-4 - 16 s^3 + 8 s^4 + 
 4 s (3 + \rho_1) - 
 s^2 (1 + 4 \rho_1)}{(-2 + 3 s) \rho_2}\right) 
 \\&-2 \arctan \left( \dfrac{-4 + 32 s^3 - 16 s^4 + 
 4 s (3 + 2 \rho_1) - 
 s^2 (25 + 8 \rho_1)}{(-2 + 3 s) \rho_2}  \right) 
 \\& + \dfrac{2-3s}{2(s-1)s} \log \left( \dfrac{- \rho_1}{-6s + 6s^2 + \rho_2} \right).
\end{aligned}
\end{equation*}
\end{theorem}

In Figure \ref{fig:hei} we can see the representation of the height invariant as a function of the parameter $s$.

\begin{figure}[ht]
\includegraphics[width=200pt]{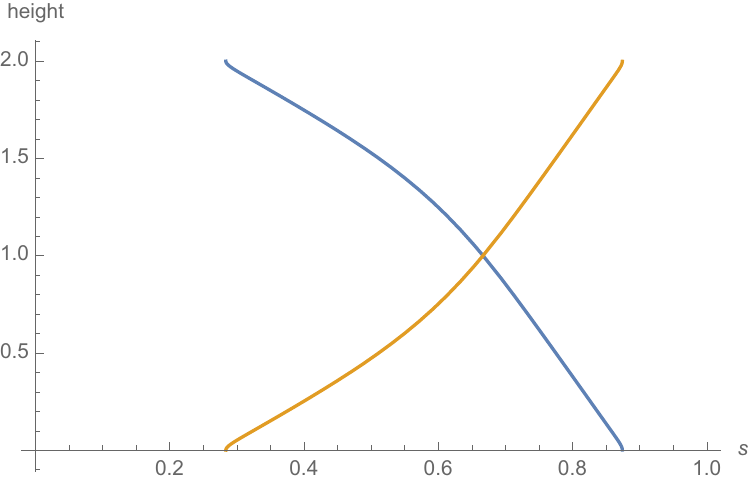}
\caption{
The height invariant as a function of the parameter $s$ exists only for $s_- < s < s_+$. The height $h_1(s)$ (blue) corresponds to the focus-focus singularity $\mathcal{N} \times \mathcal{S}$ and the height  $h_2(s)$ (yellow) to $\mathcal{S} \times \mathcal{N}$. }
\label{fig:hei}
\end{figure}

\subsection{Symmetry between the Taylor series}
\label{sec:symmetry}

Before computing the Taylor series for the focus-focus singularity $\mathcal{N}\times\mathcal{S}$, in this section we will
show that symmetries of the system can be used to determine a relationship between the
Taylor series invariants of the two focus-focus points of the system.

To do this we will make use of Proposition~\ref{prop:symmetry-general}, which explains how changing the signs of the components of the momentum map affects the Taylor series invariant. Sepe \& V\~{u} Ng\d{o}c~\cite[Theorem 4.56]{SepeVN-notes} show how different choices of local normal form charts (as in Theorem \ref{EliassonMZ}) impact the resulting Taylor series. In the following proof, we show how changing the signs of the components of the momentum map impact the preferred choice of  chart in a semitoric system, and then apply the result of Sepe \& V\~{u} Ng\d{o}c to obtain Proposition~\ref{prop:symmetry-general}.

\begin{proof}[Proof of Proposition~\ref{prop:symmetry-general}]
Since $\Phi$ is a fiber preserving symplectomorphism, the fact that $\Phi(m)$ is focus-focus is immediate.

 Since $m$ is focus-focus, by Theorem~\ref{EliassonMZ} there exists a map $\phi \colon U\to \R^4$ from a neighborhood $U$ of $m$ which is a symplectomorphism onto its image, and a local diffeomorphism $\varrho\colon \R^2\to\R^2$ around $0$ such that $\varrho(F(m)) = 0$ and $\varrho \circ F =  G \circ \phi$, where 
 \[
  G(x_1,y_1,x_2,y_2):= (G_1,G_2)(x_1,y_1,x_2,y_2) := (x_1y_2-x_2y_1, x_1 y_1 + x_2 y_2)
 \]
 and $\R^4$ is equipped with the symplectic form $\om_{\R^4} = \dee x_1 \wedge \dee y_1 + \dee x_2 \wedge \dee y_2$.
  As discussed in Section~\ref{sss:taylor}, we may, and do, choose these maps such that $\varrho_1(l,h) = l$ and that $\varrho_2(l,h)$ satisfies $\frac{\partial \varrho_2}{\partial h}>0$.
  
 Following Sepe \& V\~{u} Ng\d{o}c~\cite{SepeVN-notes} and adapting to our notation, let
\[
 A_{-1,1}(x_1,y_1,x_2,y_2) := (x_2,y_2,x_1,y_1), \qquad 
 A_{1,-1}(x_1,y_1,x_2,y_2) := (y_1, -x_1, y_, -x_2),
\]
$A_{-1,-1} := A_{-1,1}\circ A_{1,-1}$, and let $A_{1,1}$ denote the identity on $\R^4$.
Notice that $A_{\varepsilon_1,\varepsilon_2}\colon \R^4 \to \R^4$ is a symplectomorphism and
 that $G\circ A_{\varepsilon_1,\varepsilon_2} = (\varepsilon_1 G_1,\varepsilon_2 G_2)$ for any $\varepsilon_1,\varepsilon_2\in\{-1,+1\}$.

Now fix a choice of $\varepsilon_1,\varepsilon_2\in\{-1,1\}$ and let $\tilde{F}: =(\varepsilon_1 L,\varepsilon_2 H)$. Note that $(M,\om,\tilde{F})$ is a semitoric system, $\Phi$ satisfies $\Phi^*F' = \tilde{F}$, and $m$ is a focus-focus point of $(M,\om,\tilde{F})$. Let $\tilde{S}_m^\infty$ denote the Taylor series invariant of $m$ in $(M,\om,\tilde{F})$, and note that $\tilde{S}_m^\infty=(S_{\Phi(m)}')^\infty$.
To complete the proof, we will now show that
 \begin{equation}\label{eqn:tildeS-formula}
 S_m^\infty(l,j) = \varepsilon_2 \tilde{S}^\infty_{m}(\varepsilon_1 l, \varepsilon_2 j) + \left(\frac{1-\varepsilon_1}{2}\right) \pi l \quad (\textrm{mod }2\pi l).
 \end{equation}
 We identify the Klein group $K_4$ with $\{-1,1\}^2$.
 Note that since $(L,H)$ and $(\varepsilon_1 L, \varepsilon_2 H)$ induce the same fibration, due to Sepe \& V\~{u} Ng\d{o}c~\cite[Theorem 4.56]{SepeVN-notes}, the Taylor series in each case will be related by the $K_4$-action described in Sepe \& V\~{u} Ng\d{o}c~\cite[Lemma 4.52]{SepeVN-notes}. The remainder of the proof is showing that the preferred choices of each Taylor series, relative to the semitoric systems, are the ones which satisfy Equation~\eqref{eqn:tildeS-formula}.

 Define $\tilde{\varrho}= (\varepsilon_1\varrho_1,\varepsilon_2\varrho_2)$. 
Then $\tilde{\phi} := A_{\varepsilon_1,\varepsilon_2}\circ \phi\colon U \to \R^4$ is a symplectomorphism onto the image and $\tilde{\varrho}$ is a local diffeomorphism around zero such that $\tilde{\varrho}(\tilde{F}(m)) = 0$ and $\tilde{\varrho} \circ \tilde{F} =  G \circ \tilde{\phi}$.
Furthermore, writing $\tilde{\varrho} = (\tilde{\varrho}_1, \tilde{\varrho}_2)$, we see that $\tilde{\varrho}_1(\varepsilon_1 l,\varepsilon_2 h) = l$ and $\frac{\partial} {\partial h}\left(\tilde{\varrho}_2(\varepsilon_1 l, \varepsilon_2 h)\right) >0$.

Recall that associated to each local normal form chart $(\phi,\varrho)$ around a focus-focus point, there is a well defined choice of Taylor series invariant, and recall that (as described in the beginning of Section~\ref{sss:taylor}) there is a preferred choice of such a pair $(\phi,\varrho)$ around any focus-focus point in a semitoric system (up to flat functions). 
Given the preferred isomorphism $(\phi,\varrho)$ of $(M,\om,F)$, we produced a new preferred isomorphism $(\tilde{\phi},\tilde{\varrho})$ of $(M,\om,\tilde{F})$.

By Sepe \& V\~{u} Ng\d{o}c~\cite[Lemma 4.55]{SepeVN-notes}, the map which assigns the Taylor series invariant to a local normal form chart around a focus-focus point is equivariant with respect to actions of $K_4 \cong \{-1,1\}^2$.
The action of $(\varepsilon_1,\varepsilon_2)\in K_4$ on the charts is given by $(\phi,(\varrho_1,\varrho_2))\mapsto (A_{\varepsilon_1,\varepsilon_2}\circ\phi,(\varepsilon_1\varrho_1,\varepsilon_2\varrho_2))$.
Since $(\phi,\varrho)$ and $(\tilde{\phi},\tilde{\varrho})$ are related by this action, we conclude that the Taylor series are related by the formula given in Sepe \& V\~{u} Ng\d{o}c~\cite[Lemma 4.52]{SepeVN-notes}, which is Equation~\eqref{eqn:tildeS-formula}, as desired.
\end{proof}

Now we will apply this result to our system.
Recall that $m_1=\mathcal{N}\times\mathcal{S}$ and $m_2=\mathcal{S}\times\mathcal{N}$ are both focus-focus singular points for $s\in \,\,]s_-,s_+[\,$.

\begin{lemma}\label{lem:symmetry}
Let $s\in \,\,]s_-,s_+[\,$, and for such $s$ let $S^\infty_{i,s}$ denote the Taylor series invariant at the focus-focus point $m_i$ for $i\in\{1,2\}$. Then
\[
 S_{2,s}^\infty(l,j) = -S_{1,s}^\infty(-l,-j)+\pi l.
\]
\end{lemma}

\begin{proof}
 Consider the map $\Phi\colon \mbS^2\times \mbS^2 \to \mbS^2\times\mbS^2$ given by
 \[
  \Phi(x_1,y_1,z_1,x_2,y_2,z_2) = (-x_1,y_1,-z_1,x_2,-y_2,-z_2).
 \]
 Note that $\Phi^*\omega=\omega$, $\Phi(\mathcal{N}\times\mathcal{S}) = \mathcal{S}\times\mathcal{N}$, and $\Phi^*(L,H) = (-L,-H)$. Then we apply Proposition~\ref{prop:symmetry-general}, taking $\varepsilon_1=\varepsilon_2=-1$, which proves the claim.
\end{proof}

In order to calculate the twisting index invariant of the system, and therefore prove Theorem~\ref{thm:twist-intro}, we need to know
the lower order terms of the Taylor series invariant for each focus-focus point, which is the content of Theorem~\ref{thm:Taylor-intro}.
The above lemma is an important part of Theorem~\ref{thm:Taylor-intro}, since it implies that
to obtain both Taylor series it is sufficient to only explicitly compute $S^\infty_{1,s}(l,j)$.
This is what we will do now.

\subsection{The action integral}
\label{ss:action}

In order to obtain the remaining symplectic invariants of this system, we next need to compute the action integral. To do so, we perform singular symplectic reduction by the Hamiltonian $\mbS^1$-action generated by $L$, see Sjamaar \& Lerman~\cite{SL} for details of this concept.

We start by rewriting the system \eqref{eqn_ssys} using the usual cylindrical coordinates $(\theta_1,z_1,\theta_2,z_2)$ where $\theta_i$ measures the angle in the $x_iy_i$-plane in the counter clockwise direction starting from the positive $x_i$-axis for $i\in\{1,2\}$. We then have:
\begin{equation}
	\left\{
	\begin{aligned}
	L(\theta_1,z_1,\theta_2,z_2) &=  z_1 + 2 z_2,\\
	H_s(\theta_1,z_1,\theta_2,z_2) &= (1-s) z_1 + s z_2 + 2(1-s)s \sqrt{(1 - {z_1}^2) (1 - {z_2}^2)}
   \cos(\theta_1 - \theta_2).
	\end{aligned}
	\right.
	\label{eqncyl}
\end{equation} The symplectic form in cylindrical coordinates is given by $\omega =  \dee z_1 \wedge \dee \theta_1 + 2 \dee z_2 \wedge \dee \theta_2$, since the standard symplectic form on the sphere is $\omega_{\mbS^2} = \dee \theta \wedge \dee z$ and $\omega = -(\omega_{\mbS^2} \oplus 2 \omega_{\mbS^2})$. We now perform the affine coordinate change
\begin{equation}
\begin{array}{lcl}
q_1 := -\theta_1, && p_1 := z_1 + 2 z_2, \\
q_2 := \theta_1-\theta_2, && p_2 := 2(1+z_2),
\end{array}
\label{eq:changevars}
\end{equation} and obtain $L(q_1,p_1,q_2,p_2)=p_1$ and 
\begin{equation}
	\begin{aligned}
H_s(q_1,p_1,q_2,p_2) =& p_1 - p_2+2 - s(p_1 - p_2+2) + \frac{s}{2} (p_2-2) \\&+ (1-s)s \sqrt{p_2(p_2-p_1-1)(p_2-4)(p_2-p_1-3)} \cos(q_2).
       	\end{aligned}
       	\label{eq:LH}
\end{equation} 
In these coordinates, the symplectic form becomes $\omega = \dee q_1 \wedge \dee p_1 + \dee q_2 \wedge \dee p_2$. Moreover, it is obvious that $L=p_1$ is a constant of motion because $H_s$ does not depend on $q_1$.
This notation is thus particularly suitable to express the reduction by the $\mbS^1$-action induced by $L$. 

Instead of $p_1$, it is more convenient
to use the variable $l:=p_1+1\in[-2,4]$, so that $\mathcal{N} \times \mathcal{S}\in \{l=0\}$. 
We obtain the reduced space 
\[M^{\text{red},l} := \{L = l\}/\mbS^1.\] 
The reduced space at the levels $l=0$ and $l=2$ is a stratified symplectic space with the shape of a sphere with a conic singular point, at levels $l=-2$ and $l=4$ it is a point, and otherwise it is a smooth sphere.
Let $H^{\text{red},l}_s$ be the function on $M^{\text{red},l}$ induced by descending $H_s$ to the quotient.
We denote the coordinates on the reduced space by $(q_2,p_2)$, which are induced from the coordinates $(q_2,p_2)$ on $M$, defined above. The bounds for these coordinates on the reduced space depend on the level $l$:
\begin{equation}\label{eqn:coor-bounds}
-\pi \leq q_2 \leq \pi, \quad \max\{0,l\}\leq p_2 \leq \min\{l+2,4\}.
\end{equation}

In these coordinates, 
\begin{equation}\label{eqn:H-red_NS}
H_s^{\text{red},l}(q_2,p_2)=A^l_s(p_2) + \cos(q_2) \sqrt{B_s^l(p_2)},
\end{equation}
 where
\begin{equation}
	\left\{
	\begin{aligned}
	A_s^l(p_2) &= l+1-p_2-2s-l s + \tfrac{3}{2} s p_2, \\
	B_s^l(p_2) &= s^2 (1-s)^2 p_2 (p_2-l) (p_2-4) (p_2-2-l).
	\end{aligned}
	\right.
\end{equation}

For each $s\in [0,1]$, let $h$ be a constant and consider the level set
\begin{equation}\label{eqn:beta}
\beta^s_{l,h} := \left\{ (q_2,p_2)\in M^{\text{red},l} \mid H_s^{\text{red},l}(q_2,p_2)=h+(1-2s)\right\}.
\end{equation}
The level sets $\beta^s_{l,h}$ are closed curves in the reduced space, and the curve going through the focus-focus point $ \mathcal{N} \times \mathcal{S}$ corresponds to the value $h=0$, see Figure~\ref{4X2orbitsNS}.
The next step is to compute the (second) action integral from Equation~\eqref{eq:semigact},
\begin{equation}
\mcI(l,h) := \dfrac{1}{2\pi} \oint_{\beta^s_{l,h}} q_2 \dee p_2,
\label{eq:defAct}
\end{equation} where we use $\varpi = q_2 \dee p_2$ as the primitive of the symplectic form in $M^{\text{red},l}$ and the curve $\beta_{l,h}^s$ as our choice of cycle $\gamma_2^z$, cf.\ \eqref{eq:semigact}. The integral \eqref{eq:defAct} measures the symplectic volume of one of the two regions bounded by the curve $\beta_{l,h}^s$.
The condition $\partial_h \mcI(l,h)>0$ implies that, from the two regions that the curve bounds,
we have to choose the one with points satisfying $H_s^{\text{red},l}(q_2,p_2) -(1-2s) < h$.

Note that here our choice of coordinates has given us a choice of cycle, so the $\mathcal{I}$ defined in Equation~\eqref{eq:defAct} satisfies 
\[
 \mathcal{I} = I \circ \varrho
\] 
where $\varrho$ is as in the diagram in Figure~\ref{fig:diagram-semilocal} and $I$ is one of the possible choices of action as discussed in Section~\ref{sss:taylor}, and in particular Equation~\eqref{eq:semigact}. In fact, it will turn out to be the preferred action $\xi$ discussed in Remark~\ref{rmk:Xi-and-xi}, but we cannot see this a priori, as we discuss in Remark~\ref{rmk:right_cycle}.

Figure~\ref{fig:areas} shows an overview of the areas corresponding to the action integral. As $h$ increases (from left to right), so does the area in colour, which represents the area being integrated over. To obtain a circle-valued coordinate $q_2$,
we identify $(-\pi, p_2)$ with $(\pi,p_2)$ for each $p_2$, and thus obtain a cylinder with coordinates $(q_2,p_2)$.
We distinguish between three types of orbits on this cylinder:
\begin{itemize}
	\item Type I: The curve $\beta^s_{l,h}$ crosses the line $q_2=0$ and is homotopic to a point in the cylinder.
	\item Type II: The curve  $\beta^s_{l,h}$ crosses both $q_2=0$ and $q_2=\pm \pi$ and is not homotopic to a point in the cylinder.
	\item Type III: The curve $\beta^s_{l,h}$ crosses the line $q_2=\pm \pi$ and is homotopic to a point in the cylinder.
\end{itemize}

\begin{figure}[ht]
 \centering
  	\begin{subfigure}[b]{3cm}
       \includegraphics[width=3cm]{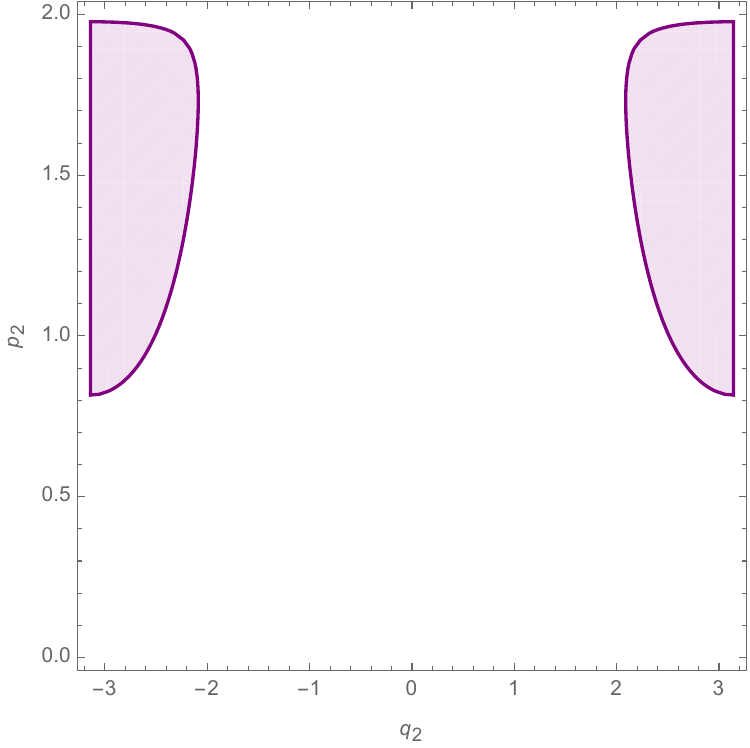}
        \caption*{\small \textit{$s<\tfrac{2}{3}$, $h<h_{l^+}$}}
    \end{subfigure}
        \begin{subfigure}[b]{4cm}
        \centering
       \includegraphics[width=3cm]{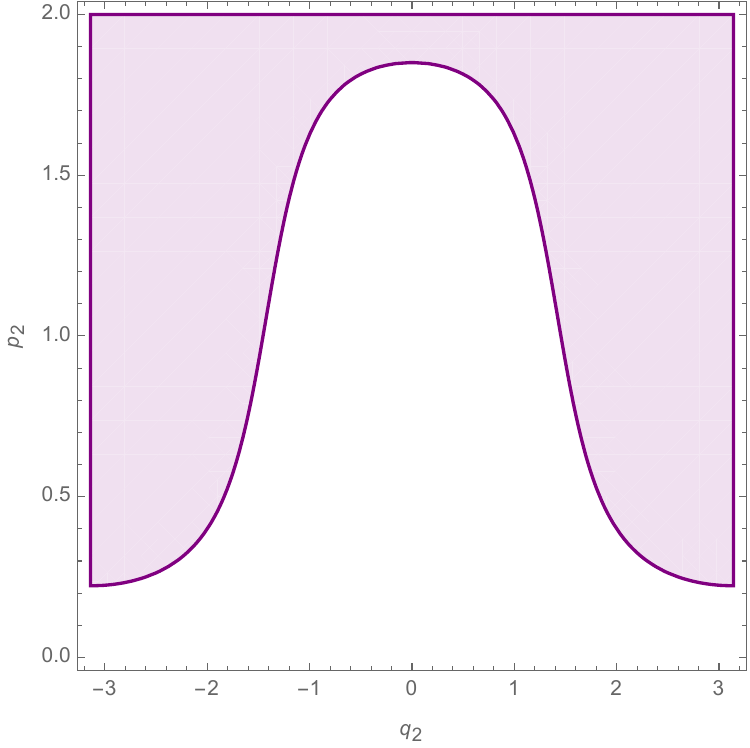}
        \caption*{\small \textit{$s<\tfrac{2}{3}$, $h_{l^+}<h<h_{l^-}$}}    \end{subfigure}
    \begin{subfigure}[b]{3cm}
       \includegraphics[width=3cm]{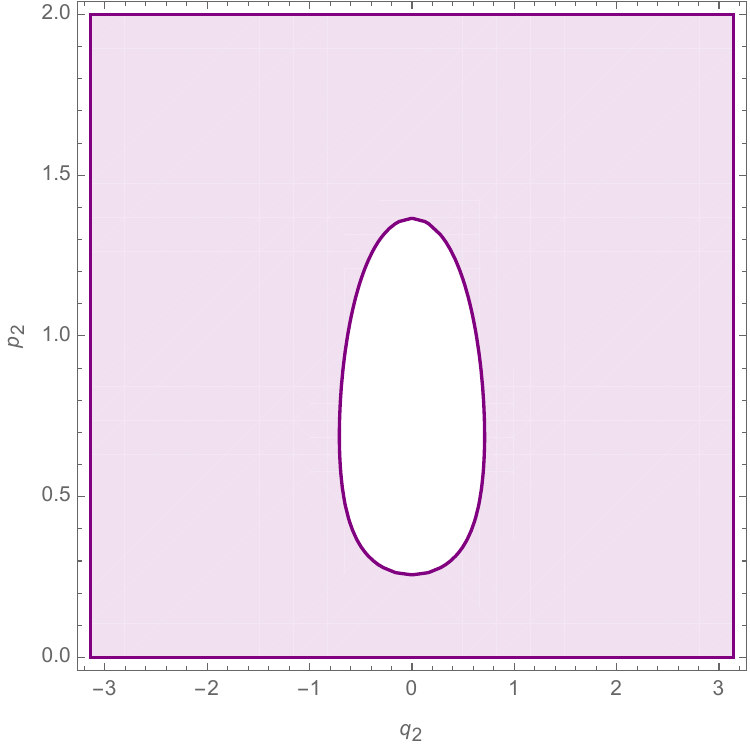}
        \caption*{\small \textit{$s<\tfrac{2}{3}$, $h_{l^-}<h$}}
    \end{subfigure}
    \\[0.5cm]
  	\begin{subfigure}[b]{3cm}
       \includegraphics[width=3cm]{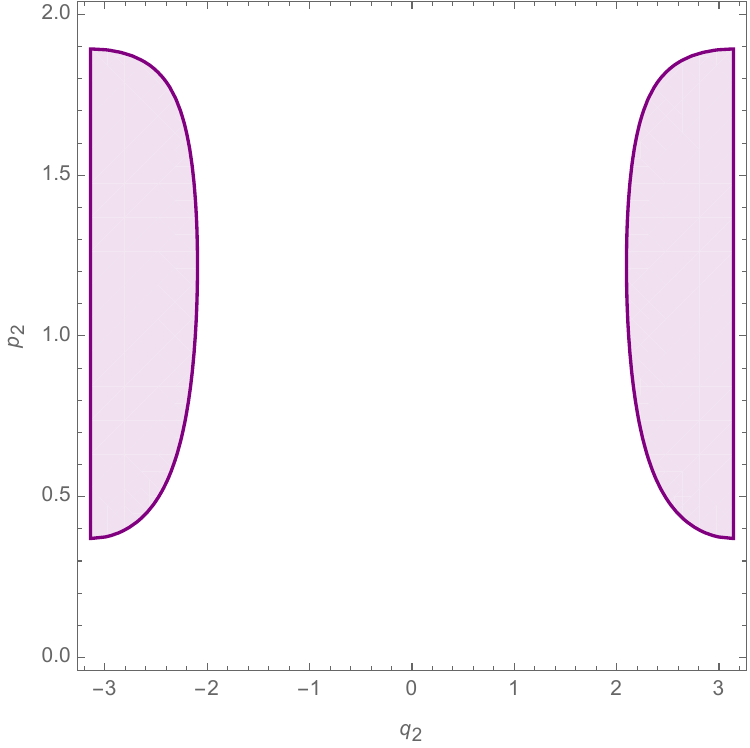}
        \caption*{\small \textit{$s=\tfrac{2}{3}$, $h<h_{l^\pm}$}}
    \end{subfigure}
        \begin{subfigure}[b]{4cm}
        \centering
       \includegraphics[width=3cm]{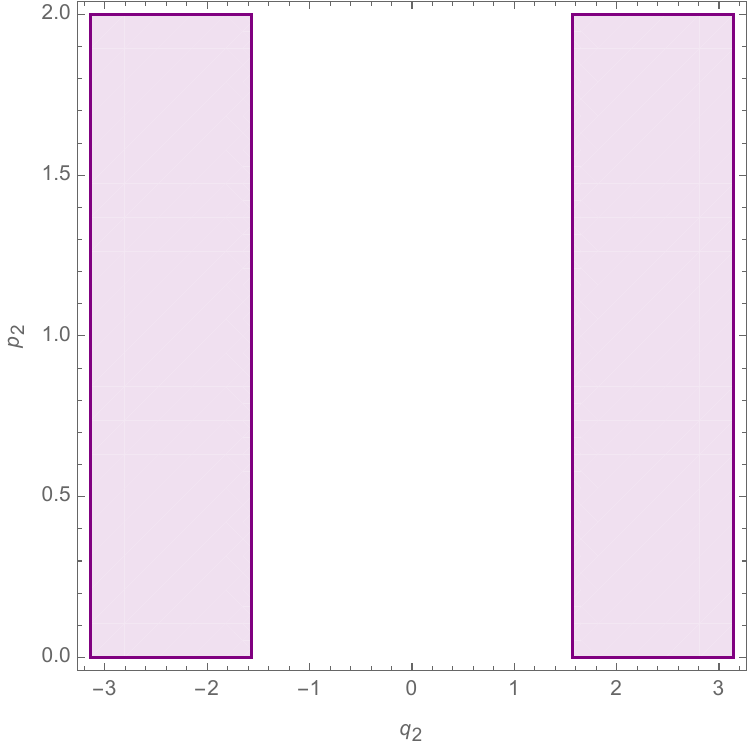}
        \caption*{\small \textit{$s=\tfrac{2}{3}$, $h=h_{l^\pm}$}}    \end{subfigure}
    \begin{subfigure}[b]{3cm}
       \includegraphics[width=3cm]{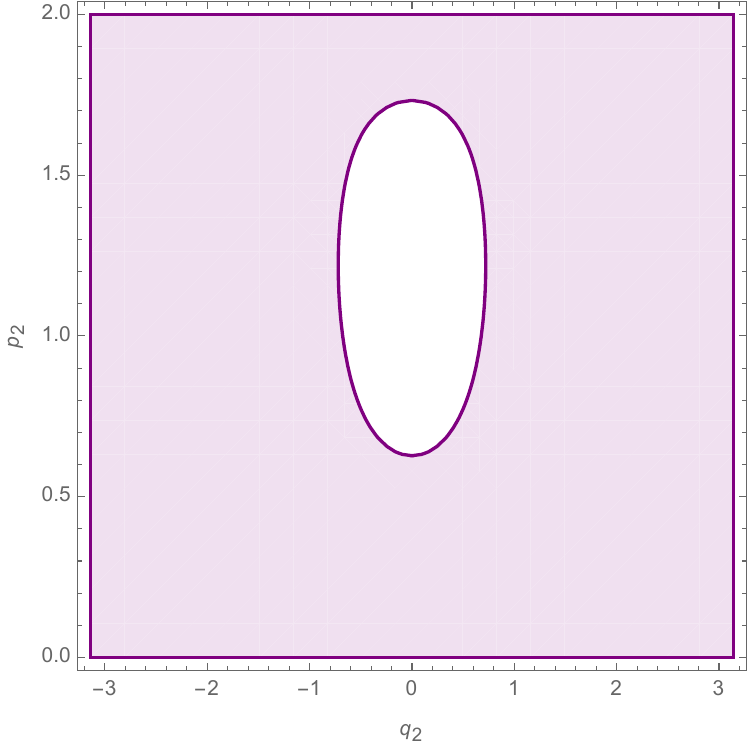}
        \caption*{\small \textit{$s=\tfrac{2}{3}$, $h_{l^\pm}<h$}}
    \end{subfigure}
    \\[0.5cm]
    \begin{subfigure}[b]{3cm}
       \includegraphics[width=3cm]{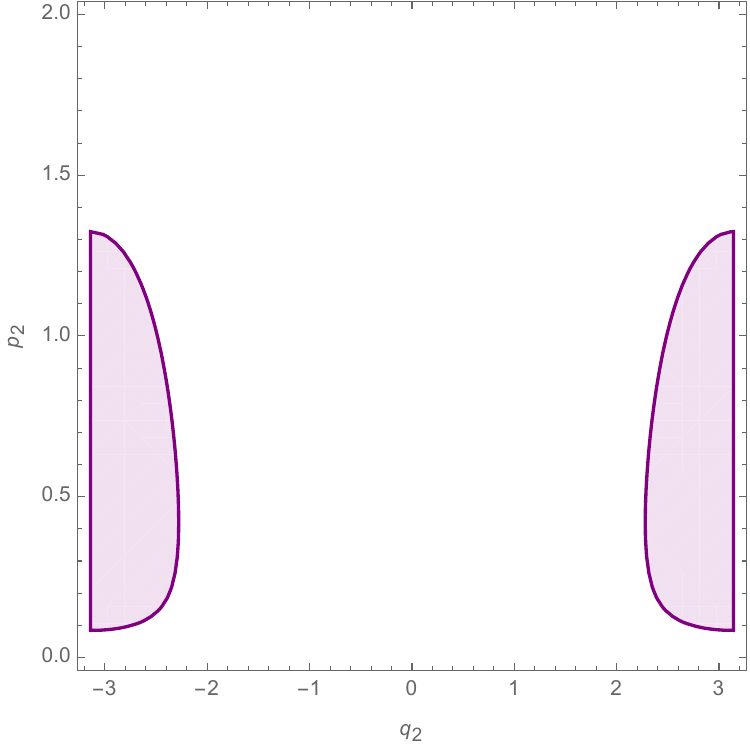}
        \caption*{\small \textit{$s>\tfrac{2}{3}$, $h<h_{l^-}$}}
    \end{subfigure}
        \begin{subfigure}[b]{4cm}
        \centering
       \includegraphics[width=3cm]{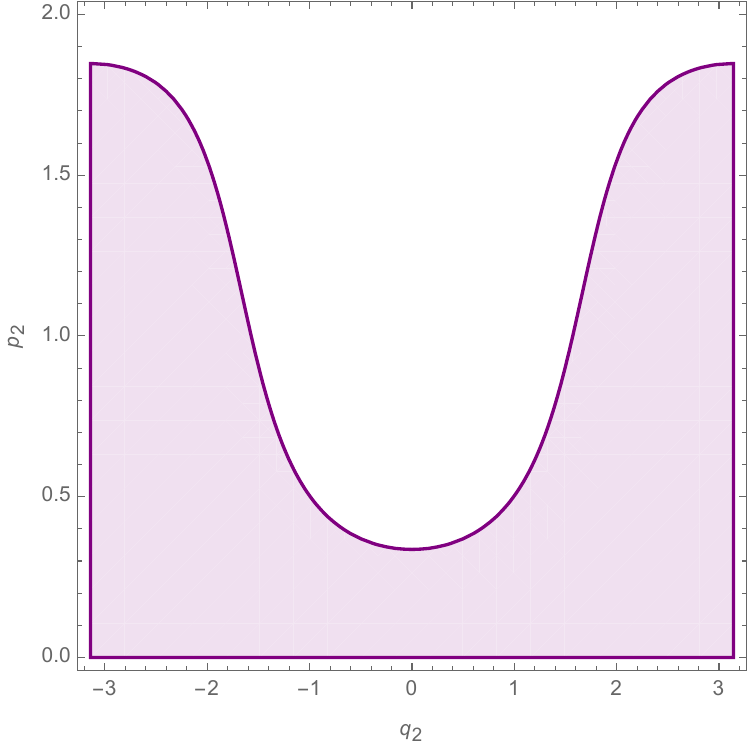}
        \caption*{\small \textit{$s>\tfrac{2}{3}$, $h_{l^-}<h<h_{l^+}$}}    \end{subfigure}
    \begin{subfigure}[b]{3cm}
       \includegraphics[width=3cm]{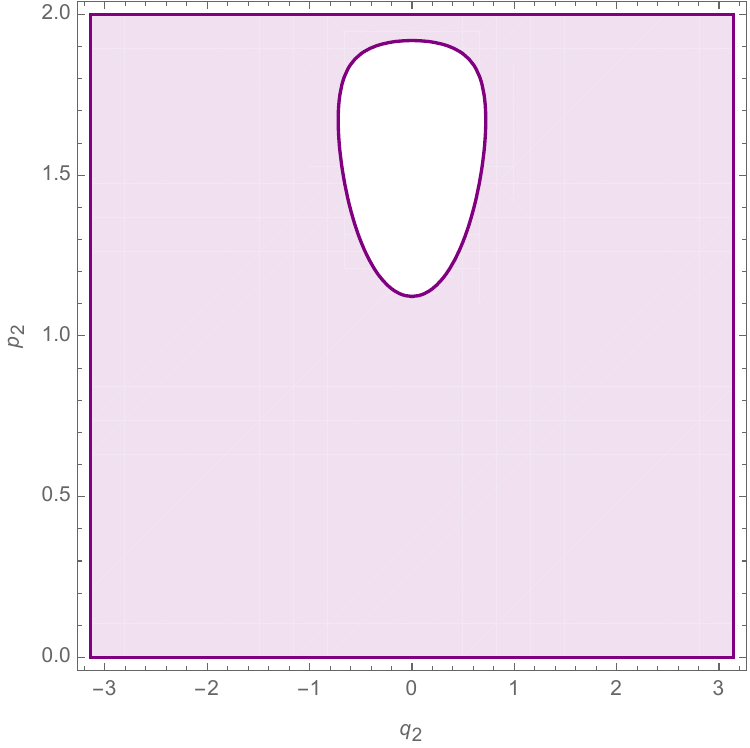}
        \caption*{\small \textit{$s>\tfrac{2}{3}$, $h_{l^+}<h$}}
    \end{subfigure}
\caption{Area representing the action integral for different values of $s$ and $h$. In the left column, we have curves of type III, in the centre, of type II, and on the right, of type I. An exception is the figure in the very center ($s=\tfrac{2}{3}$, $h<h_{l^\pm}$), which is a degenerate case.}
 \label{fig:areas}
\end{figure}

The different types of curves are distinguished by means of special separatrix curves, which correspond to the values $h_{[p_2=0]},h_{[p_2=l]},h_{[p_2=4]},h_{[p_2=2+l]}$ of $H^{\text{red},l}_s(q_2,p_2) -(1-2s)$ along the bounds of the coordinates on $M^{\text{red},l}$, namely $p_2=0$, $p_2=l$, $p_2=4$ and $p_2=2+l$:
\begin{equation}\label{eqn:h0-hl-h4-h2l}
\begin{aligned}
h_0 := h_{[p_2=0]} &= l(1-s), & \qquad h_4 := h_{[p_2=4]} &= 6\left(  s-\frac{2}{3}\right) + l (1-s),  \\[0.3cm]
h_l := h_{[p_2=l]} &= \dfrac{s}{2} l, & \qquad h_{2+l}:=h_{[p_2=2+l]} &= 3\left(s-\frac{2}{3} \right) + \frac{s}{2}l.
\end{aligned}
\end{equation}

That is, setting 
\begin{equation}
l^- := \max\{0,l\}, \quad l^+ := \min\{2+l,4\},
\label{lminmax}
\end{equation}
we have that for each value of $l$, there are two values $h_{l^-}:=h_{[p_2=l^-]}$ and $h_ {l^+}:=h_{[p_2=l^+]}$ that separate the three types of curves, see Figures \ref{fig:h_vs_l} and \ref{4X2orbitsNS}. 
Note that $l^-$ and $l^+$ are the extreme values of $p_2$, i.e. $l^- \leq p_2 \leq l^+$.

The value of the parameter $s$ then determines the type of the curve $\beta^s_{l,h}$. Note that if $s< \tfrac{2}{3}$, then $h_{l^-} > h_{l^+}$, if $s=\tfrac{2}{3}$, then $h_{l^-}=h_{l^+}$ and if $s> \tfrac{2}{3}$, then $h_{l^-} < h_{l^+}$, as illustrated in Figure \ref{fig:h_vs_l}.
 
\begin{figure}[ht]
 \centering
  	\begin{subfigure}[b]{4cm}
       \includegraphics[width=4cm]{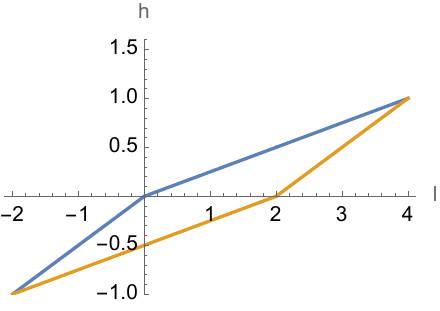}
        \caption*{$s< \tfrac{2}{3}$}
    \end{subfigure}
        \begin{subfigure}[b]{4cm}
       \includegraphics[width=4cm]{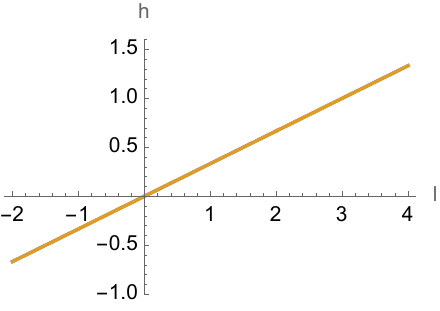}
        \caption*{$s= \tfrac{2}{3}$}    \end{subfigure}
    \begin{subfigure}[b]{4cm}
       \includegraphics[width=4cm]{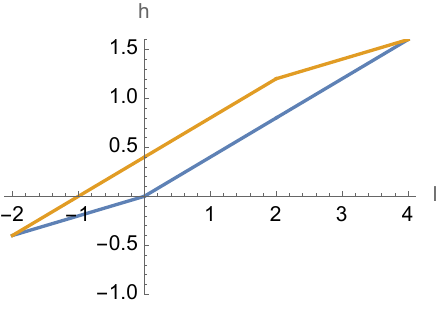}
        \caption*{$s> \tfrac{2}{3}$}
    \end{subfigure}
    \caption{Representation of $h_{l^-}$ (blue) and $h_{l^+}$ (yellow) as a function of $l$.}
    \label{fig:h_vs_l}
\end{figure}

\begin{figure}[ht]
 \centering
  	\begin{subfigure}[b]{3cm}
       \includegraphics[width=3cm]{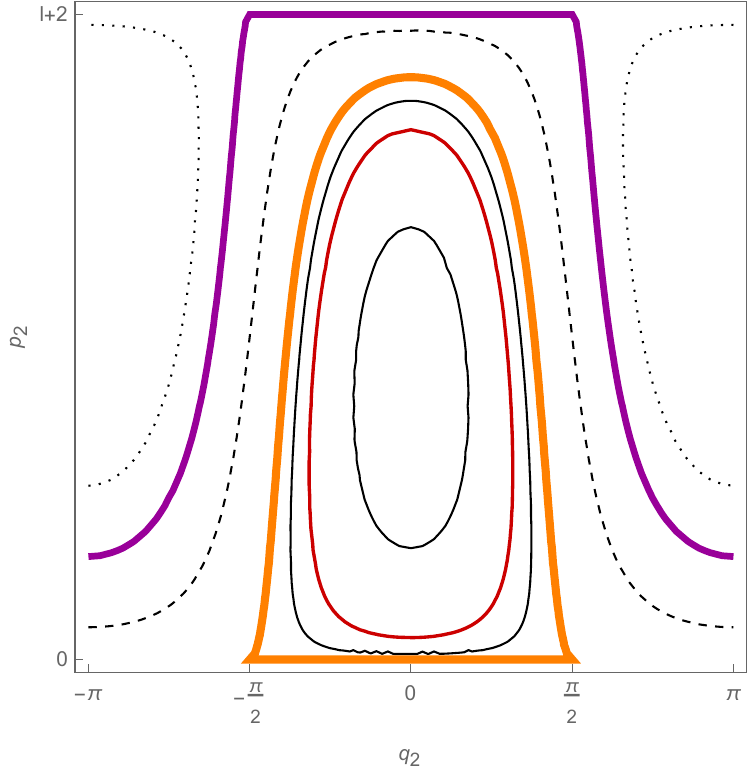}
        \caption*{\small \textit{$s<\tfrac{2}{3}$, $l<0$}}
    \end{subfigure}
        \begin{subfigure}[b]{3cm}
       \includegraphics[width=3cm]{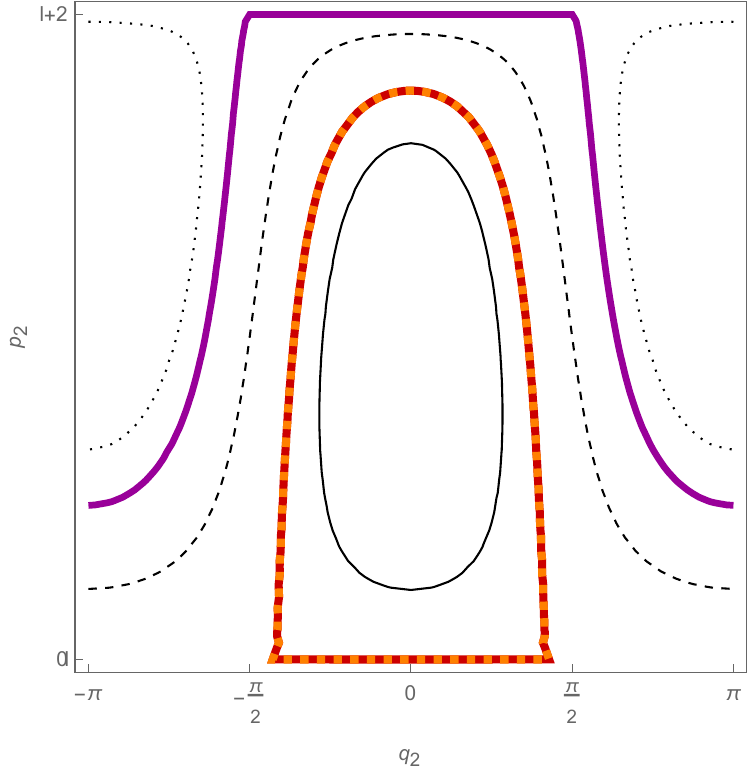}
        \caption*{\small \textit{$s<\tfrac{2}{3}$, $l=0$}}    \end{subfigure}
    \begin{subfigure}[b]{3cm}
       \includegraphics[width=3cm]{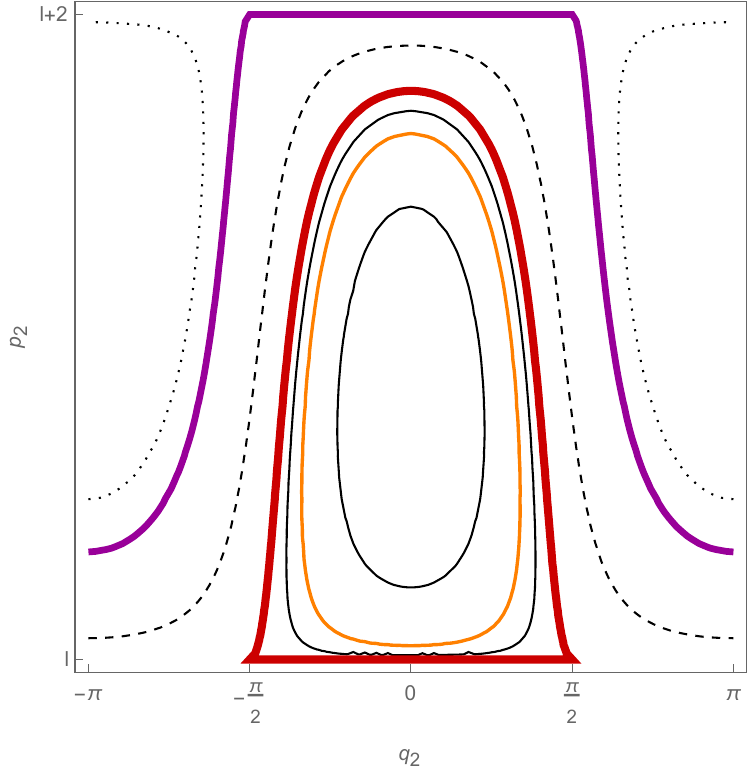}
        \caption*{\small \textit{$s<\tfrac{2}{3}$, $l>0$}}
    \end{subfigure}
    \\[0.5cm]
  	\begin{subfigure}[b]{3cm}
       \includegraphics[width=3cm]{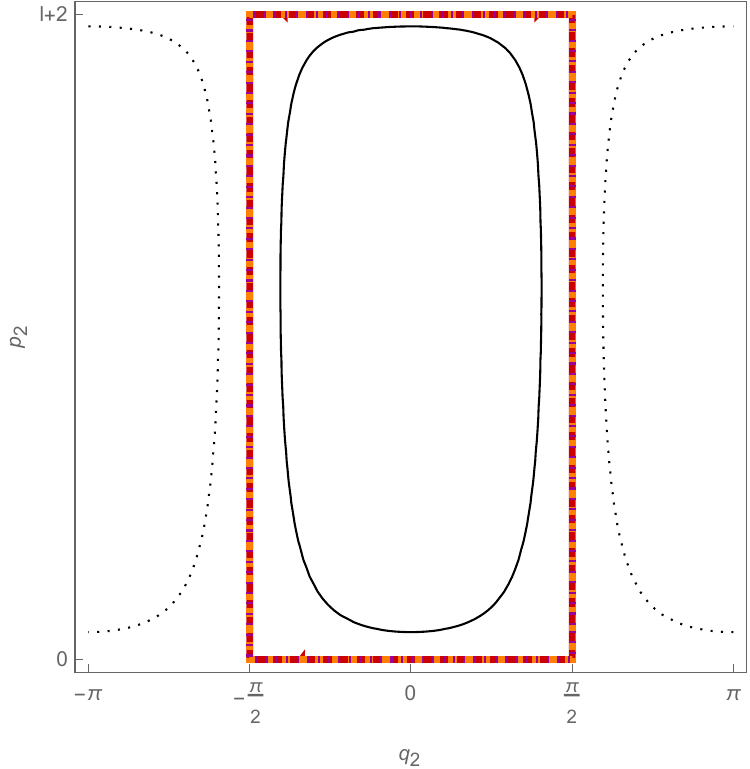}
        \caption*{\small \textit{$s=\tfrac{2}{3}$, $l<0$}}
    \end{subfigure}
        \begin{subfigure}[b]{3cm}
       \includegraphics[width=3cm]{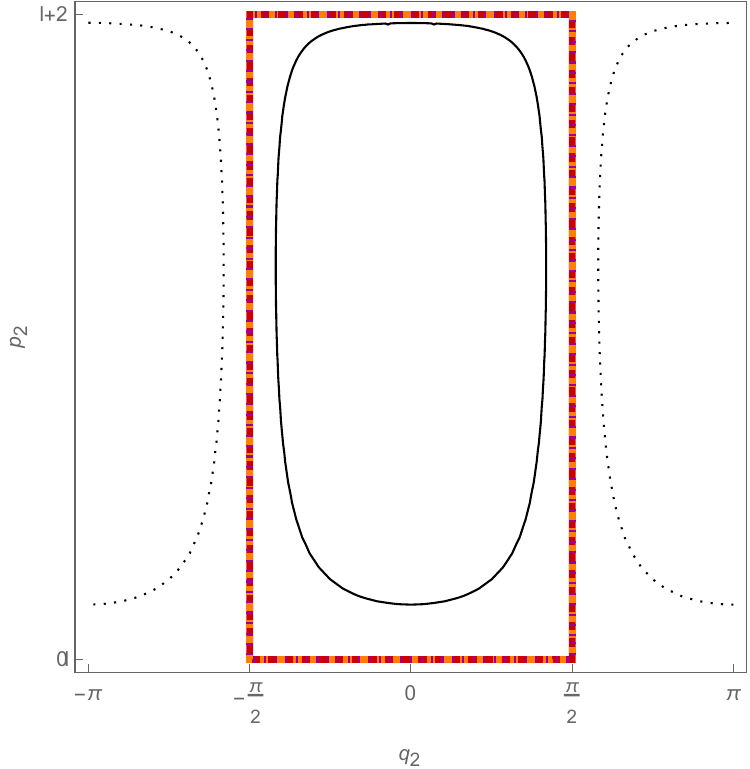}
        \caption*{\small \textit{$s=\tfrac{2}{3}$, $l=0$}}    \end{subfigure}
    \begin{subfigure}[b]{3cm}
       \includegraphics[width=3cm]{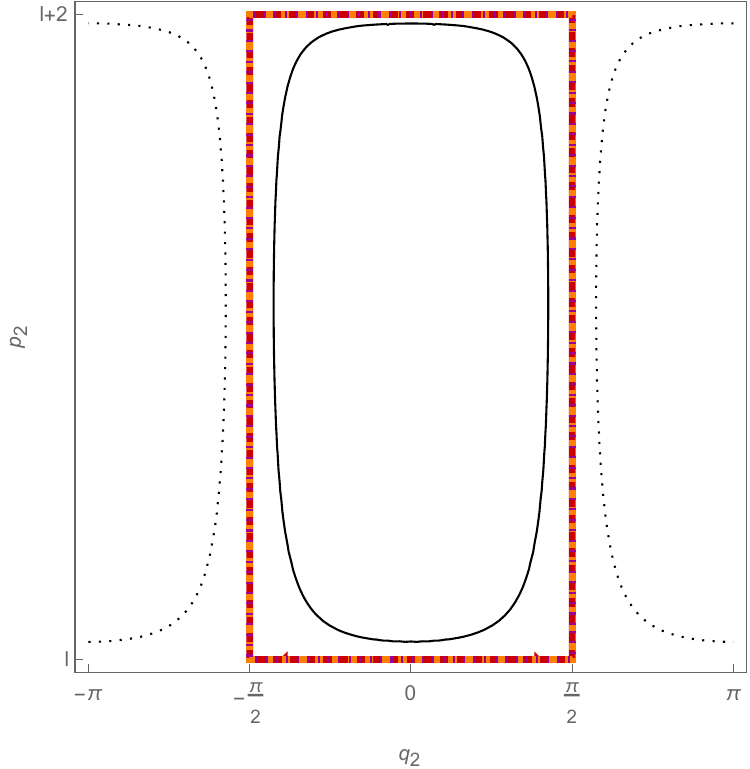}
        \caption*{\small \textit{$s=\tfrac{2}{3}$, $l>0$}}
    \end{subfigure}
    \\[0.5cm]
    \begin{subfigure}[b]{3cm}
       \includegraphics[width=3cm]{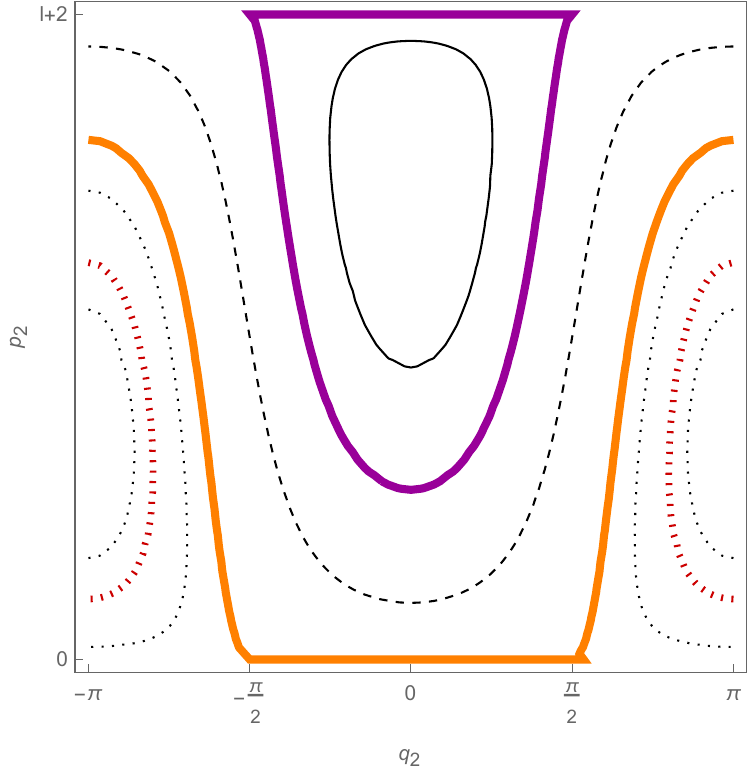}
        \caption*{\small \textit{$s>\tfrac{2}{3}$, $l<0$}}
    \end{subfigure}
        \begin{subfigure}[b]{3cm}
       \includegraphics[width=3cm]{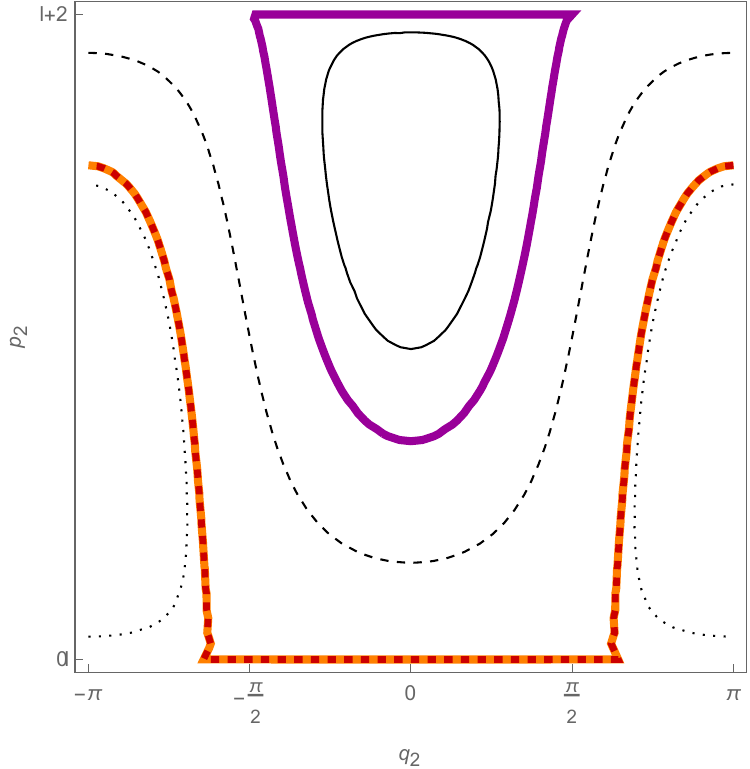}
        \caption*{\small \textit{$s>\tfrac{2}{3}$, $l=0$}}    \end{subfigure}
    \begin{subfigure}[b]{3cm}
       \includegraphics[width=3cm]{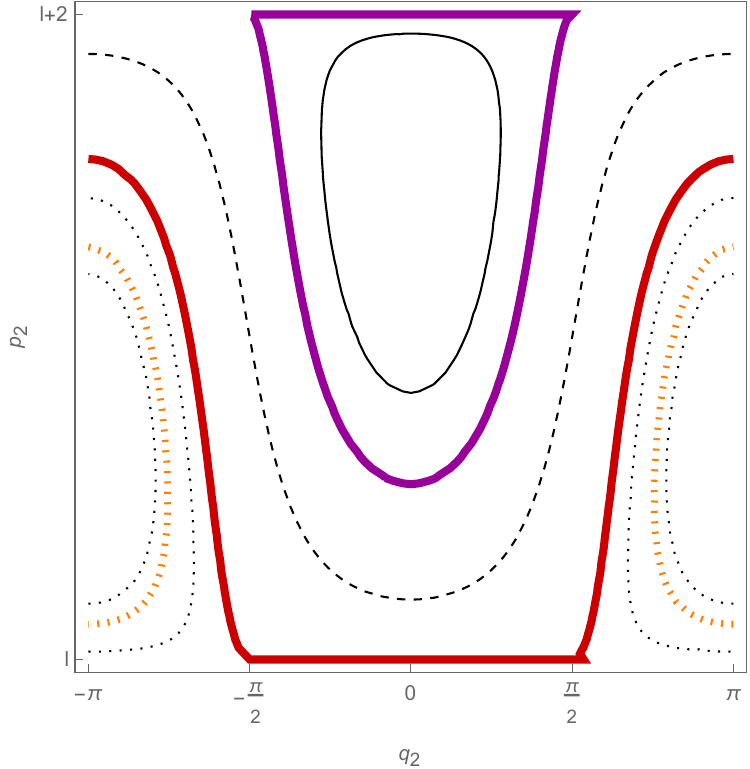}
        \caption*{\small \textit{$s>\tfrac{2}{3}$, $l>0$}}
    \end{subfigure}
\caption{\small \textit{Level sets of the function $H^{\text{red},l}_s(q_2,p_2)$ from Equation~\eqref{eqn:H-red_NS} for different values of $l$ and $s$. Type I orbits are full lines, type II are dashed and type III are dotted. The limit orbits are $h=h_{[p_2=0]}$ in orange, $h=h_{[p_2=l]}$ in red and $h=h_{[p_2=2+l]}$ in purple, where $h=H^{\text{red},l}_s(q_2,p_2)-(1-2s)$ and $h_{[p_2=0]}$, $h_{[p_2=l]}$, and $h_{[p_2=2+l]}$ are as given in Equation~\eqref{eqn:h0-hl-h4-h2l}.}}
 \label{4X2orbitsNS}
\end{figure}

 The curve types are shown in Table~\ref{table:types}.

\begin{table}[ht]
\centering
\setstretch{1.5}
\begin{tabular}{c||c|c|c}
&$s< \tfrac{2}{3}$ & $s= \tfrac{2}{3}$ & $s> \tfrac{2}{3}$  \\  \hline\hline 
$h$ large &$h > h_{l^-}$: type I & $h > h_{l^\pm}$: type I & $h > h_{l^+}$: type I  \\
\hline  
$h$ intermediate &$h_{l^-}> h > h_{l^+}$: type II & \emph{not possible} & $h_{l^+} > h > h_{l^-}$: type II \\
\hline 
$h$ small &$h_{l^+} > h$: type III & $h_{l^\pm} > h$: type III & $h_{l^-} > h$: type III\\[.2cm]
\end{tabular}
\setstretch{1.0}
\caption{The various types of curves depending on the values of $h$ and $s$.}
\label{table:types}
\end{table}

In Figure \ref{4X2orbitsNS} the relation between the different types of curves and the separatrices around the values $l=0$ and $s = \tfrac{2}{3}$ is shown.
Considering the subfigures from left to right, note that when $l<0$, the curve $H_0(q_2,p_2)-(1-2s)=h_{[p_2=0]}$ (orange colour) acts as separatrix, while the curve $H_l(q_2,p_2)-(1-2s)=h_{[p_2=l]}$ (red colour) is a normal curve. For $l=0$ both curves coincide and for $l>0$ the roles of the curves are exchanged. Considering the subfigures from top to bottom, note that for $s<\tfrac{2}{3}$, the curve $H_l(q_2,p_2)-(1-2s)=h_{l^-}$ (orange or red colour) separates orbits of types I and II and the curve $H_l(q_2,p_2)-(1-2s)=h_{l^+}$ (purple colour) separates the orbits of types II and III. For $s=\tfrac{2}{3}$, these curves coincide and for $s>\tfrac{2}{3}$, the roles are exchanged. 

We define now the polynomial 
\begin{equation}
\label{eq:polP}
\begin{aligned}
P^s_{l,h}(p_2) := B_s^l(p_2) - (h+(1+2s)-A_s^l(p_2))^2,
\end{aligned}
\end{equation} which is of degree 4 in $p_2$.  
Let $\ze_1, \ze_2, \ze_3, \ze_4$ denote the roots of this polynomial.
As in Alonso \& Dullin \& Hohloch~\cite[Equation (3.13)]{ADH2}, it turns out that within the bounds of the coordinate $p_2$ given in~Equation~\eqref{eqn:coor-bounds} all roots are real. Moreover, if we assume that they are ordered as  
\[
 \ze_1\leq \ze_2\leq \ze_3\leq \ze_4,
\]
then
\begin{equation}
 \ze_1 \leq  \min\{0,l\} \quad \max\{0,l\} \leq \ze_2 \leq \ze_3 \leq \min\{4, l+2\}, \quad \max\{4,l+2\} \leq \ze_4. 
\label{eq:rootsP}
\end{equation} 
That is, only $\ze_2$ and $\ze_3$ are within the bounds of the coordinate $p_2$ and therefore only these two roots are meaningful on the reduced space $M^{\text{red},l}$. In particular, $\beta_{l,h}^s \subset \{ (q_2,p_2) \mid \ze_2 \leq p_2 \leq \ze_3\}$.

We now express the action integral as follows.

\begin{proposition}
\label{prop:4X2act}
 The action integral $\mcI(l,h)$ from Equation~\eqref{eq:defAct} can be written as
\begin{equation}
\mcI(l,h) =  C^B_{s,h}(l) + \mfI(l,h)
\label{4X2actintNS}
\end{equation}
where the values of $C_{s,h}^B(l)$ are given in Table~\ref{table:areaB} and  $\mfI(l,h)$ is the elliptic integral
 $$\mfI(l,h):= \dfrac{1}{2\pi} \oint_{\beta^s_{l,h}} R_\mcI(p_2) \frac{ \dee p_2}{\sqrt{P^s_{l,h}(p_2)}}.$$
The elliptic curve $\beta^s_{l,h}$ is given by $\vartheta^2 = P^s_{l,h}(p_2)$, where $P^s_{l,h}$ is as in Equation~\eqref{eq:polP}, and 
 \begin{align}\label{eqn:R}
R_\mcI(p_2) =& \dfrac{h_0 + h_l - h_4 - h_{2+l}}{6}p_2 + \dfrac{4h-h_0 - h_l - h_4 -h_{2+l}}{2} \\& +  \dfrac{l}{2}\dfrac{(h-h_l)}{p_2-l} + \dfrac{4}{2} \dfrac{(h-h_4)}{p_2-4} + \dfrac{(2+l)}{2} \dfrac{(h-h_{2+l})}{p_2-2-l}.\nonumber
\end{align}
\end{proposition}

Note that the values of $p_2$ for which the denominators of each of the last three terms of $R_\mcI(p_2)$
are zero corresponds to the separatrix curves discussed above.

\begin{proof}
 We want to compute the integral in Equation~\eqref{eq:defAct}, where the curve $\beta^s_{l,h}$ is defined by the equation 
 \begin{equation}\label{eqn:beta-proof} 
 h =  A_s^l(p_2) + \sqrt{B_s^l(p_2)}\cos(q_2) -(1-2s).
 \end{equation} and measures the areas represented in Figure~\ref{fig:areas}. Because of the symmetry of $q\mapsto -q$ in Equation~\eqref{eqn:beta-proof}, to compute the integral over $\beta^s_{l,h}$ we can restrict to the portion of $\beta^s_{l,h}$ with $q_2\geq 0$ and multiply by $2$. Moreover, we can solve for $q_2$ in Equation~\eqref{eqn:beta-proof},
$$ q_2(p_2) =\arccos \left(\dfrac{h+(1-2s)-A_s^l(p_2)}{\sqrt{B_s^l(p_2)}} \right).$$  We can thus integrate $q=q_2(p_2)$ between $\ze_2$ and $\ze_3$. However, we will also need to deal with the regions to the left and/or right of the curve by including certain correction terms. Let us illustrate this with the case $s < \tfrac{2}{3}$. We want to compute the ``half-areas" represented in Figure~\ref{fig:halfareas}.
 
\begin{figure}[ht]
 \centering
  	\begin{subfigure}[b]{4.5cm}
       \includegraphics[width=4.5cm]{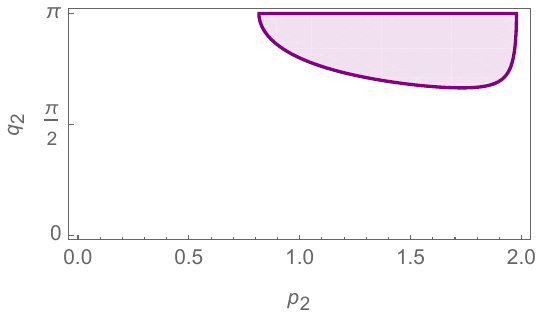}
        \caption*{$h<h_{l^+}$ (Type III)}
    \end{subfigure}
        \begin{subfigure}[b]{4.5cm}
       \includegraphics[width=4.5cm]{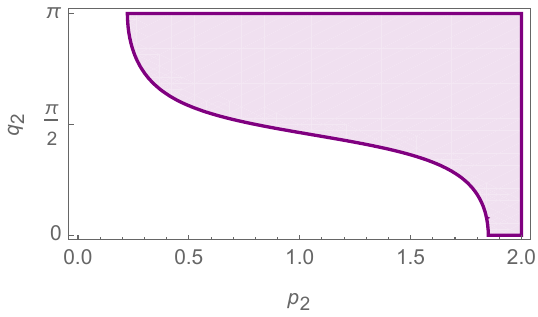}
        \caption*{$h_{l^+}<h<h_{l^-}$ (Type II)}    \end{subfigure}
    \begin{subfigure}[b]{4.5cm}
       \includegraphics[width=4.5cm]{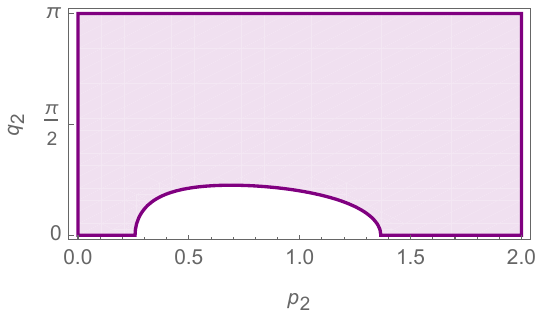}
        \caption*{$h_{l^-}<h$ (Type I)}
    \end{subfigure}
    \caption{Half of the area to be integrated in the case $s< \tfrac{2}{3}$.}
    \label{fig:halfareas}
\end{figure}

The size of the half-rectangles are $(2+l)\pi$ if $l<0$ and $2\pi$ if $l \geq 0$. From this total size, we need to subtract the area below the curve $q_2=q_2(p_2)$ and, depending on the type of the curve, some additional area. 
More specifically, for the right situation (type I) we need not subtract anything extra, for the middle situation (type II) we need to subtract the area to the left of $\ze_2$ and for the left situation (type III) we need to subtract the area to the left of $\ze_2$ and to the right of $\ze_3$. In general, we can write the action integral as

\begin{align}\label{eqn:action-compute}
\mcI(l,h) =  C_{s,h}^A(l) - \frac{1}{\pi}\int_{\ze_2}^{\ze_3} \arccos \left(\dfrac{h+(1-2s)-A_s^l(p_2)}{\sqrt{B_s^l(p_2)}} \right) \dee p_2,
\end{align}

where $C_{s,h}^A(l)$ is given in Table~\ref{table:areaA}.

\begin{table}[ht]
\centering
\setstretch{1.5}
\begin{tabular}{c||c|c|c}
$C_{s,h}^A(l)$&$s< \tfrac{2}{3}$ & $s= \tfrac{2}{3}$ & $s> \tfrac{2}{3}$  \\ \hline\hline
Type I & $l^+-l^-$ & $l^+-l^-$ & $l^+-l^-$ \\
\hline
Type II & $l^+-\ze_2$ & \emph{not possible} & $\ze_3-l^-$ \\
\hline
Type III & $\ze_3 -\ze_2$ & $\ze_3 -\ze_2$ & $\ze_3 -\ze_2$
\end{tabular}
\setstretch{1.0}
\caption{The value of $C_{s,h}^A(l)$ depending on the type of the curve being integrated over.}
\label{table:action-proof}
\label{table:areaA}
\end{table}

The next step is to do integration by parts
\begin{align*}
\mcI(l,h) &= C_{s,h}^A(l) - \frac{p_2}{\pi}\left[ \arccos \left(\dfrac{h+(1-2s)-A^l_s(p_2)}{\sqrt{B^l_s(p_2)}} \right) \right]_{p_2=\ze_2}^{p_2=\ze_3}  \\
&\quad + \dfrac{1}{\pi}\int_{\ze_2}^{\ze_3} p_2 \dfrac{d}{dp_2} \left(\arccos \left(\dfrac{h+(1-2s)-A^l_s(p_2)}{\sqrt{B^l_s(p_2)}} \right)\right) \dee p_2 \\
&= \,
C_{s,h}^B(l) + \dfrac{1}{\pi}\int_{\ze_2}^{\ze_3} R_\mcI(p_2) \frac{ \dee p_2}{\sqrt{P^s_{l,h}(p_2)}}\\
&= \,C_{s,h}^B(l) +\dfrac{1}{2\pi} \oint_{\beta^s_{l,h}} R_\mcI(p_2) \frac{ \dee p_2}{\sqrt{P^s_{l,h}(p_2)}},
\end{align*}
 where $R_\mcI(p_2)$ is as in Equation~\eqref{eqn:R} and
\begin{align*}
C_{s,h}^B(l) &:= C_{s,h}^A(l)- \dfrac{p_2}{\pi} \left[ \arccos \left(\dfrac{h+(1-2s)-A^l_s(p_2)}{\sqrt{B^l_s(p_2)}} \right) \right]_{p_2=\ze_2}^{p_2=\ze_3}.
\end{align*} 
 We list all possible values of $C_{s,h}^B(l)$ in Table~\ref{table:areaB}.
\begin{table}[ht]
\centering
\setstretch{1.5}
\begin{tabular}{c||c|c|c}
$C_{s,h}^B(l)$&$s< \tfrac{2}{3}$ & $s= \tfrac{2}{3}$ & $s> \tfrac{2}{3}$  \\ \hline\hline
Type I & $l^+-l^-$ & $l^+-l^-$ & $l^+-l^-$ \\
\hline
Type II & $l^+$ & \emph{not possible} & $-l^-$ \\
\hline
Type III & $0$ & $0$ & $0$
\end{tabular}
\setstretch{1.0}
\caption{The value of $C_{s,h}^B(l)$ depending on the type of the curve being integrated over.}
\label{table:areaB}
\end{table}
\end{proof}

We now compute the \emph{reduced period} $\mcT$ and the \emph{rotation number} $\mcW$, defined by
\begin{equation}
\mcT(l,h) := 2\pi \dfrac{\partial \mcI}{\partial h},\qquad \mcW(l,h) := -\dfrac{\partial \mcI}{\partial l}.
\label{def:TW}
\end{equation} 

\begin{remark}
Each of the quantities in Equation~\eqref{def:TW} has a geometric interpretation:
since $M$ is compact, any regular fiber of the integrable system $(L,H)$ is a torus, and the quotient of this torus by the circle action generated by $L$ is a circle.
The flow of the Hamiltonian vector field of $H$ descends to this circle, and thus this flow is necessarily periodic,
and the reduced period $\mcT(l,h)$ is the minimal period of this flow. 
This is because $\mcT(l,h)$ computes the relative speed of the flow of $\mathcal{I}$ (which has period $2\pi$) against the speed of the flow of $H$ to determine the time taken for the flow of $H$ to go once around and return to the orbit of the $\mbS^1$-action that it started in.
The rotation number $\mcW(l,h)$ computes the relative speed between the given periodic flow (generated by $L$) and the additional periodic flow generated by $\mathcal{I}$ on a regular fiber.
\end{remark}

\begin{corollary}\label{cor:reduced-period-rot}
The reduced period $\mcT(l,j)$ and the rotation number $\mcW(l,h)$ are given by the complete elliptic integrals
\begin{equation}
\mcT(l,h) = \oint_{\beta^s_{l,h}} \frac{ \dee p_2}{\sqrt{P^s_{l,h}(p_2)}},\qquad \mcW(l,h) = C_{s,h}^C(l)+\dfrac{1}{2\pi} \oint_{\beta^s_{l,h}} R_{\mcW}(p_2) \frac{ \dee p_2}{\sqrt{P^s_{l,h}(p_2)}}
\label{eq:defTW}
\end{equation}
over the elliptic curve $\vartheta^2 = P^s_{l,h}(p_2)$, where
\begin{align*}
R_{\mcW}&(p_2) = \dfrac{s}{2} - \dfrac{1}{2} \dfrac{(h-h_l)}{p_2-l} - \dfrac{1}{2} \dfrac{(h-h_{2+l})}{p_2-2-l}.
\end{align*}
\end{corollary}
\begin{proof}
We compute the derivatives given in Equation~\eqref{def:TW} using the expression for $\mcI(l,h)$ from Proposition~\ref{prop:4X2act}
and the explicit formula in Equation~\eqref{eqn:action-compute}. For the reduced period, we have
\begin{align*}
\mcT(l,h) &= 2\pi \dfrac{\partial \mcI}{\partial h}(l,h) = -2 \frac{\partial}{\partial h}\int_{\ze_2}^{\ze_3} \arccos \left(\dfrac{h+(1-2s)-A_s^l(p_2)}{\sqrt{B_s^l(p_2)}} \right) \dee p_2 \\&= 2 \int_{\ze_2}^{\ze_3} \frac{ \dee p_2}{\sqrt{P^s_{l,h}(p_2)}} = \oint_{\beta^s_{l,h}} \frac{ \dee p_2}{\sqrt{P^s_{l,h}(p_2)}}
\end{align*} and we proceed similarly for $\mcW$,
\begin{align*}
\mcW(l,h) &= - \dfrac{\partial \mcI}{\partial l}(l,h) = - \dfrac{\partial C_{s,h}^B(l)}{\partial l} + \dfrac{1}{\pi} \int_{\ze_2}^{\ze_3} R_{\mcW}(p_2) \frac{ \dee p_2}{\sqrt{P^s_{l,h}(p_2)}} \\&= 
C_{s,h}^C(l) + \dfrac{1}{2\pi} \oint_{\beta^s_{l,h}} R_{\mcW}(p_2) \frac{ \dee p_2}{\sqrt{P^s_{l,h}(p_2)}},
\end{align*} where 
\begin{align*}
R_{\mcW}&(p_2) = \dfrac{s}{2} - \dfrac{1}{2} \dfrac{h-h_l}{p_2-l} - \dfrac{1}{2} \dfrac{h-h_{2+l}}{p_2-2-l}
\end{align*} and
the values of 
$$ C_{s,h}^C(l) := - \dfrac{\partial C_{s,h}^B(l)}{\partial l}$$ are given in Table~\ref{table:areaC}.
\end{proof}

\begin{table}[ht]
\centering
\setstretch{1.5}
\begin{tabular}{l||l|l|l|l}
&Range of $h$ & $l<0$ & $0<l<2$ & $2<l$  \\ \hline\hline 
$s< \tfrac{2}{3} \qquad$&$h> h_{l^-}$ & $-1$ & $0$ & $1$ \\ \hline
&$ h_{l^-} > h > h_{l^+}$ & $-1$ & $-1$ & $0$ \\ \hline
&$h_{l^+} > h$ & $0$ & $0$ & $0$ \\ \hline\hline
$s= \tfrac{2}{3} \qquad$&$h> h_{l^\pm}$ & $-1$ & $0$ & $1$ \\ \hline
&$h_{l^\pm} > h$ & $0$ & $0$ & $0$ \\ \hline\hline
$s> \tfrac{2}{3} \qquad$&$h> h_{l^+}$ & $-1$ & $0$ & $1$ \\ \hline
&$ h_{l^+} > h > h_{l^-}$ & $0$ & $1$ & $1$ \\ \hline
&$h_{l^-} > h$ & $0$ & $0$ & $0$ \\
\end{tabular}
\setstretch{1.0}
\caption{Values of $C_{s,h}^C(l)$.}
\label{table:areaC}
\end{table}

The integrals in~\eqref{eq:defTW} can be expressed in terms of two basic integrals, namely 
\begin{equation}
\mcN_A := \int^{\ze_3}_{\ze_2} \dfrac{\dee p_2}{\vartheta},\qquad \mcN_{B,\eta} := \int^{\ze_3}_{\ze_2} \dfrac{1}{p_2-\eta} \dfrac{\dee p_2}{\vartheta},
\label{NANB}
\end{equation} where $\eta$ is a constant. More precisely, we have
\begin{equation} \mcT = 2 \mcN_A,\qquad \mcW = C^C_{s,h} + \dfrac{1}{2\pi} \left( \dfrac{s}{2} \mcN_A - \dfrac{h-h_l}{2} \mcN_{B,l} - \dfrac{h-h_{2+l}}{2} \mcN_{B,2+l} \right). 
\label{eq:TWN}
\end{equation}
The two integrals in \eqref{NANB} can be rewritten in Legendre's standard form by changing the integration variable $x$ and defining the parameters $k, n_\eta$, where
$$ x := \sqrt{\dfrac{\ze_3-p_2}{\ze_3-\ze_2}},\qquad  k^2 := \dfrac{\ze_3-\ze_2}{\ze_3-\ze_1},\qquad n_\eta := \dfrac{\ze_3-\ze_2}{\ze_3-\eta}.$$ We write now $\mcN_A$ in terms of the complete elliptic integral of first kind
\begin{equation}
\mcN_A = \dfrac{\sqrt{2}}{\sqrt{\ze_3-\ze_1}} \int_0^1 \dfrac{\dee x}{\sqrt{(1-x^2)(1-k^2x^2)}} = \dfrac{\sqrt{2}}{\sqrt{\ze_3-\ze_1}} K(k)
\label{eq:NAEl}
\end{equation} and $\mcN_{B,\eta}$ as a complete elliptic integral of third kind
\begin{align}
\nonumber \mcN_{B,\eta} &= \dfrac{\sqrt{2}}{(\ze_3 - \eta)\sqrt{\ze_3-\ze_1}} \int_0^1 \dfrac{\dee x}{(1-n_\eta x^2)\sqrt{(1-x^2)(1-k^2x^2)}} \\&= \dfrac{\sqrt{2}}{(\ze_3 - \eta)\sqrt{\ze_3-\ze_1}} \Pi(n_\eta,k).
\label{eq:NBEl}
\end{align}

\subsection{Preparations for the Taylor series invariant}\label{sec:prep-for-TS}
We now set up the general theory to compute the Taylor series invariant of the system given in Equation~\eqref{eqn_ssys}. We will use the method from Alonso \& Dullin \& Hohloch \cite{ADH2}, based on the properties of complex elliptic curves and the expansions of the reduced period and the rotation number. It is similar to the method introduced by Dullin \cite{Du} and used in Alonso \& Dullin \& Hohloch \cite{ADH}.

The elliptic integrals in Equations~\eqref{eq:defAct} and~\eqref{eq:defTW} go along the real elliptic curve $\vartheta^2 = P^s_{l,h}(p_2)$, where $P^s_{l,h}$ was defined in \eqref{eq:polP}. However, we can also consider the variables as complex numbers,
$$ \Gamma_l = \left\{ \left. (p_2,\vartheta) \in \overline{\C}^2 \;\right|\; \vartheta^2 = P^s_{l,h}(p_2) \right\}, $$
where $\overline{\mathbb{C}} = \mathbb{C}\cup\{\infty\}$ is the Riemann sphere.  
Since  $P^s_{l,h}(p_2)$ is a polynomial of degree 4, the curve $\Gamma_l$ is homeomorphic to a two-torus, and thus its first homotopy group is generated by two cycles, as represented in Figure \ref{fig:elliptic}. 
We will work with the distinguished cycles $\alpha=\al^s_{l,h}$ and $\beta=\beta^s_{l,h}$ which have the properties we describe now.
The \emph{real cycle} $\beta$, defined in Equation~\eqref{eqn:beta}, corresponds to the real elliptic curve and connects the roots $\ze_2$ and $\ze_3$. The other cycle, $\alpha$, is known as the \emph{imaginary cycle} and connects the roots $\ze_1$ and $\ze_2$.

\begin{figure}[ht]

\centering
    \begin{subfigure}[b]{0.5\textwidth}
    	\centering
        \includegraphics[width=0.9\textwidth]{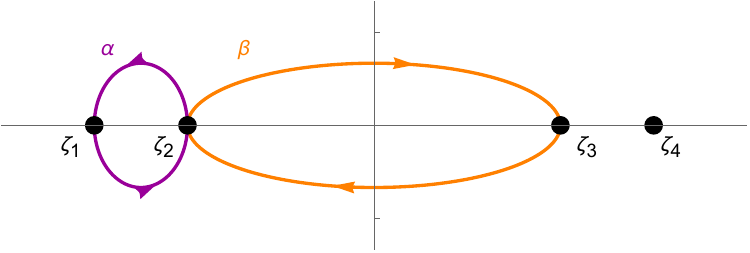}
        \caption{\small{Cycles $\al$ and $\be$ on $\C$.}}
    \end{subfigure}
    \begin{subfigure}[b]{0.4\textwidth}
    	\centering
        \includegraphics[trim=0 2.4cm 0 0.5cm, width=0.8\textwidth]{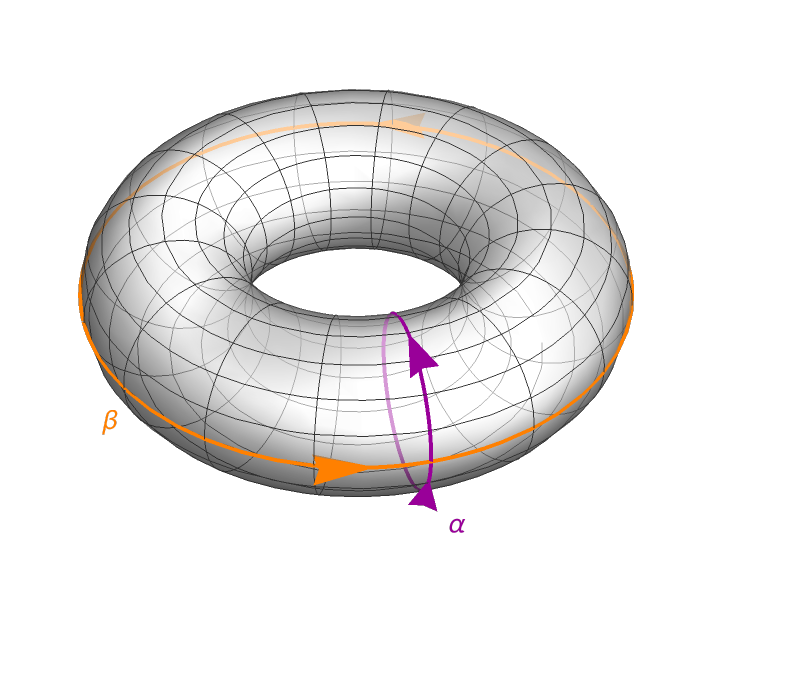}
        \caption{\small{Cycles $\al$ and $\be$ on $\T^2$}}
    \end{subfigure} 
\caption{Complex elliptic curves are homeomorphic to two-tori. The cycle $\beta$ is called the \emph{real cycle}, since it coincides with the elliptic curve when $p_2$ is considered as a real variable, and $\alpha$ is called the \emph{imaginary cycle}. Furthermore, $\alpha$ is also known as the \emph{vanishing cycle}.}
\label{fig:elliptic}
\end{figure}

As we approach the focus-focus value $(l,h)=(0,0)$, the roots $\ze_1$ and $\ze_2$ move closer to each other and coincide in the limit, so one representative of the cycle $\alpha$ collapses.
 For this reason, $\alpha$ is known as the \emph{vanishing cycle}. 
 The elliptic integrals of the action, the reduced period and the rotation number along this cycle are purely imaginary complex numbers, so we can divide by the imaginary unit and obtain real numbers. We define thus the \emph{imaginary action}
\begin{equation}\label{eqn:im-action} 
J(l,h):=\dfrac{1}{2\pi i} \oint_{\al^s_{l,h}} R_\mcI(p_2) \frac{ \dee p_2}{\sqrt{P^s_{l,h}(p_2)}},
\end{equation}
the \emph{imaginary reduced period} and the  \emph{imaginary rotation number}, 
$$ \mathcal{T}^{\al}(l,h) := \dfrac{1}{i}\oint_{\al^s_{l,h}} \frac{ \dee p_2}{\sqrt{P^s_{l,h}(p_2)}},\qquad \mathcal{W}^{\al}(l,h) := \dfrac{1}{2\pi i} \oint_{\al^s_{l,h}} R_W(p_2) \frac{ \dee p_2}{\sqrt{P^s_{l,h}(p_2)}}$$
all defined along the vanishing cycle $\alpha^s_{l,h}$ of the elliptic curve $\vartheta^2 = P^s_{l,h}(p_2)$, where $P^s_{l,h}$ is as given in Equation~\eqref{eq:polP}.
 Since the cycle $\alpha_{l,h}^s$ vanishes as we approach the focus-focus singular value, the Taylor expansion of these quantities around the singular value can be obtained using the residue theorem, which we perform in the following section.

\begin{remark} \label{rmk:varrho_is_J}
It was first observed by Dullin \cite{Du} that the imaginary action coincides with the second component of the local diffeomorphism $\varrho$ from Theorem \ref{EliassonMZ}, that is, $\varrho(l,h) = (l,J(l,h))$. See also the diagrams in Figures \ref{fig:diagram-local} and \ref{fig:diagram-semilocal}. By inverting $j = J(l,h)$ we will obtain the Birkhoff normal form of the focus-focus singularity as $h = Z(l,j)$. This will allow us to express the action integral \eqref{eq:defAct} as a function of $(l,j)$ instead of $(l,h)$. In other words, $\varrho^{-1}(l,j) = (l,Z(l,j))$. For more details on this construction, see Dullin \cite{Du} and Alonso \& Dullin \& Hohloch \cite{ADH,ADH2}.
\end{remark}

\begin{remark}
 Note that when approaching the other focus-focus point of the system, the roots $\ze_3$ and $\ze_4$ come together and
 coincide, causing a different representative of the same cycle $\alpha$ to collapse.
\end{remark}

 Since the Taylor series invariant has no constant term, it is fully determined by its partial derivatives.
To determine the partial derivatives of the Taylor series invariant associated to the singularity $\mathcal{N} \times \mathcal{S}$, we will combine the Birkhoff normal form with the following result, which relates the real and imaginary versions of the reduced period and the rotation number. 

\begin{theorem}[Theorem 3.9 of \cite{ADH2}]\label{thm:birk-to-partials}
Let $p\in M$ be a focus-focus singular point.
Let $Z(l,j)$ be the associated Birkhoff normal form and $w = l+ij$ where $j$ is the value of the imaginary action $J$. 
Let $S(l,j)$ denote the desingularized action at $p$.
Then
$$
\begin{aligned}
\frac{\partial S}{\partial l} (l,j) &= 2\pi \left.\left( \mathcal{W}^\alpha (l,h) \frac{\mathcal{T}(l,h)}{\mathcal{T}^\alpha(l,h)} - \mathcal{W}(l,h) \right)\right|_{h=Z(l,j)} \hspace{-0.2cm}+ \arg(w), \\[0.2cm] \frac{\partial S}{\partial j}(l,j)  &= 2\pi \left.\frac{\mathcal{T}(l,h)}{\mathcal{T}^\alpha(l,h)}\right|_{h=Z(l,j)}\hspace{-0.2cm}+\ln |w|.
\end{aligned}
$$ 
\end{theorem}

Note that there are various choices of desingularized action $S$, cf.\ Remark \ref{rmk:nonunique}. In this theorem we will obtain the one corresponding to our action \eqref{eq:defAct}, which is related to $S$ by \eqref{eqn:def-of-S}. The function $S$ is thus defined up to addition of flat functions, but since we are only using the series expansion $Z(l,j)$, the theorem holds for any such choice.

\subsection{Computing the Taylor series invariant}
\label{sec:proof-of-Taylor}
 
With Theorem~\ref{thm:birk-to-partials} in hand, we now proceed to complete the computation of the Taylor series.
In order to do this, we must compute the Birkhoff normal form, the reduced period, the rotation number, the imaginary reduced period, and the imaginary rotation number. Proceeding analogously to~\cite{ADH2}, we now calculate expansions of each of these quantities.
 
To obtain the expansions in this article, we scaled the parameters by $\varepsilon$ and expanded with respect to $\varepsilon$. That is, we replaced $(l,h)$ by $(\varepsilon l, \varepsilon h)$ and expanded around $\varepsilon=0$, which corresponds to the focus-focus point $(l,h)=(0,0)$.
Finally, we set $\varepsilon=1$ to obtain the desired Taylor series expansions.
In the following, we use $\mcO(N)$ to denote terms of order $N$ or higher in the variables $l$ and $j$.

\begin{lemma}\label{lem:birkhoff}
The Birkhoff normal form around the focus-focus point is given by
\begin{small}
\begin{equation*}
	\begin{aligned}
	Z(l,j) &= \dfrac{1}{4} (j\, \ra + l\, (2 - s))  +  \dfrac{(1 - s)^2 {s}^2 (2 j\ l\ \ra - 9 j^2 (2 - 3 s) - 
   3 l^2 (2 - 3 s)) }{4\ \ra^2} + \mcO (3),
	\end{aligned}
\end{equation*} where $\rho_2$ is as in Equation~\eqref{eqn:rho}.
\end{small}
\end{lemma}

\begin{proof}We apply the same technique as in the proof of~\cite[Lemma 3.8]{ADH2}.
The imaginary action \eqref{eqn:im-action} is an elliptic integral along the cycle $\alpha_{l,h}^s$, which vanishes as $(l,h)$ approaches the singular value $(0,0)$. This means that we can use the residue theorem of complex analysis and obtain a series expansion of $J(l,h)$ around $(0,0)$:
\begin{align*}
	J(l,h) &= \dfrac{1}{\ra} ( - (2-s)l+4h) \\&+ \dfrac{8s^2(1-s^2)}{\ra^5} \left((44s^5-128 s^4+115 s^3-28 s^2+2s-4)l^2  \right. \\& \hspace{2cm}\left. + 2(16s^4-32s^3+25s^2-30 s+16)hl-18(2-3s)h^2  \right) + \mcO (3).
\end{align*}
 We then solve $j=J(l,h)$ for $h$, obtaining the Birkhoff normal form $h = Z(l,j)$. 
\end{proof}

We are now prepared to prove Theorem~\ref{thm:Taylor-intro}.

\begin{proof}[Proof of Theorem~\ref{thm:Taylor-intro}]
We start by computing the Taylor series invariant at the $\mathcal{N}\times\mathcal{S}$ singularity, which we denote by $S_{1,s}^\infty(l,h)$.
Since we wrote $\mathcal{T}$ and $\mathcal{W}$ as elliptic integrals in Legendre's canonical form in \eqref{eq:TWN}, \eqref{eq:NAEl}, and \eqref{eq:NBEl}, we may apply expansions for elliptic integrals to obtain expansions for $\mathcal{T}$ and $\mathcal{W}$ similar as it was done in~\cite{ADH}.
Furthermore, $\mathcal{T}^\alpha$ and $\mathcal{W}^\alpha$ are obtained as derivatives of the Birkhoff normal form given in Lemma~\ref{lem:birkhoff}.
Thus, we can apply the equation given in Theorem~\ref{thm:birk-to-partials} to obtain expansions for the derivatives of $S_{1,s}^\infty(l,h)$.

The partial derivative of the Taylor series invariant with respect to $j$ is:
\begin{equation}\label{eqn:dSdj}
\begin{aligned}
 \frac{\partial S_{1,s}^\infty}{\partial j}(l,j) &= \log \left( \dfrac{\ra^3}{ (1-s)^2 {s}^2\sqrt{2}\rc}  \right) \\
 &+ \dfrac{1}{16 \ra^3 \rc^2} \left( l\ra  (16 - 96 {s} + 360 {s}^2 - 936 {s}^3 + 2693 {s}^4 \right. \\&\qquad  \left.- 6200 {s}^5 + 
 8004 {s}^6 - 5120 {s}^7 + 1280 {s}^8) \right. \\&\qquad  \left.+ j(-96 + 720 {s}   - 7248 {s}^2 + 36312 {s}^3 - 99558 {s}^4 \right. \\&\qquad  \left. + 174957 {s}^5 - 
 211536 {s}^6 + 171924 {s}^7 - 83328 {s}^8 \right. \\&\qquad  \left.+ 17856 {s}^9) \right)   + \dfrac{d_1}{d_2}  + \mcO(3),
\end{aligned}
\end{equation}
where $d_1$ and $d_2$ are as given in Appendix~\ref{appendix:d1d2} and $\rho_1,\rho_2$ are as in Equation~\eqref{eqn:rho}.
From the explicit formula in Appendix~\ref{appendix:d1d2}, for each fixed $s$, we note that $\frac{d_1}{d_2} = d_3^sj^2 + d_4^s jl + d_5^sl^2$ for some constants $d_3^s, d_4^s, d_5^s\in\R$, each depending only on the parameter $s$. 
The partial derivative of the Taylor series invariant with respect to $l$ is given by
\begin{equation}\label{eqn:dSdl}
\begin{aligned}
\frac{\partial S_{1,s}^\infty}{\partial l}(l,j) &= \arctan \left( \dfrac{6-9s}{\ra} \right) + \dfrac{1}{16 \ra^3 \rc^2} \left(  3l(2-3s)(16 - 96 {s} \right. \\& \qquad \left.- 216 {s}^2 + 1944 {s}^3  - 3211 {s}^4 + 424 {s}^5 + 
 3252 {s}^6 \right. \\& \qquad \left.- 2816 {s}^7 + 704 {s}^8) + j \ra (16 - 96 {s} + 360 {s}^2 \right. \\& \qquad \left.- 936 {s}^3 + 2693 {s}^4 - 6200 {s}^5 + 
 8004 {s}^6 - 5120 {s}^7 + 1280 {s}^8)\right) \\&+ \mcO(2)
\end{aligned}
\end{equation}

We can now combine Equations~\eqref{eqn:dSdj} and~\eqref{eqn:dSdl} to obtain the Taylor series invariant up to $\mcO(3)$, as in the statement of the theorem. The formula for $S_{2,s}^\infty(l,j)$ in terms of $S_{1,s}^\infty(l,j)$ is from Lemma~\ref{lem:symmetry}.
\end{proof}

\begin{remark} \label{rmk:right_cycle}
In Equation \eqref{eqn:dSdl} notice that the constant term of $\frac{\partial S}{\partial l}$ lies in $[-\frac{\pi}{2},\frac{3\pi}{2}[$.
Therefore, the desingularised action $S$ that we have chosen is actually the preferred choice $\hat{S}$ described in Lemma \ref{lem:Spref}, and therefore the Taylor series expansion in Theorem~\ref{thm:Taylor-intro} is the preferred one $(\Sp)^\infty$ discussed in Theorem~\ref{thm:twist-intro}. In other words, it turns out that our choice of cycle $\gamma_2^z$ that we made in \eqref{eq:defAct} is actually the preferred cycle $\gamma_\text{pref}^z$ defined in Section \ref{sec:geometricinterpret}, and therefore $\mathcal{I}$ from Equation~\eqref{eq:defAct} satisfies
\[
 \mathcal{I} = \xi \circ \varrho
\]
where $\varrho$ is as in the diagram in Figure~\ref{fig:diagram-semilocal}.
If this would not have been the case, we would simply have had to add or substract multiples of $2\pi l$ to obtain $\hat{S}$ from $S$.
\end{remark}

\subsection{The twisting index invariant}
\label{sec:twist-proof}

The goal of this section is to prove Theorem~\ref{thm:twist-intro}, that is, we compute the twisting index invariant of the system given in Equation~\eqref{eqn_ssys}. We extend the procedure used in Alonso \& Dullin \& Hohloch \cite{ADH,ADH2} to a situation with two focus-focus singularities. That is, we use the Taylor series invariants at each of the focus-focus points to obtain local privileged momentum maps, extend these local maps to the entire manifold, and then compare the images of these maps to the polygon invariant to compute the twisting index.

To be more precise, recall that the twisting index invariant is an assignment of a tuple of integers (one for each focus-focus point) to each representative of the polygon invariant, and furthermore recall that the value of such an assignment on a single given representative uniquely determines its value on all representatives by the group action given in Equation~\eqref{eq:twisact}.
To calculate the twisting index we start by extending the privileged momentum maps $\nu_1$ and $\nu_2$, which are defined in a small neighborhood of the focus-focus point $m_1$ and $m_2$, respectively, to the entire manifold. Recall that, as explained in Subsections \ref{sss:polygon} and \ref{sss:twisting}, each choice of representative of the polygon invariant $\De$ corresponds to a momentum map $\mu_{\De}\colon M\to \R^2$ which is toric away from the preimages of the cuts and satisfies $\De = \mu_{\De}(M)$. Near each focus-focus point $m_r$ we have that 
\[\mu_{\De} = T^{\kappa_r^{\Delta}} \circ \nu_r\]
for some $\kappa_r^{\Delta}\in\Z$. 
For each $r\in\{1,2\}$, our goal is to use the image of $\nu_r$ to determine the choice of polygon $\Delta$ for which $\kappa_r^\Delta=0$. 

After fixing a choice of cut directions, the associated representatives of the polygon invariant are related by applying iterates of the linear map $T$ from Equation~\eqref{eqn:T}, which changes the representatives drastically.
See for instance Figure \ref{fig:polygons-with-kappa-1}.
Thus, a sufficiently accurate approximation of the images of the privileged momentum maps is enough to determine which choice of polygon has a momentum map locally equal to the given privileged momentum map, i.e.~the choice of $\De$ such that $\mu_\De = \nu_r$. 
This representative has index $\kappa_r^\Delta=0$, and therefore repeating this for $r=1$ and $r=2$ completely determines the twisting index invariant of the system. The approximate image shown in Figure~\ref{fig:twisComp} is produced from our expansions of the Taylor series invariant from Theorem~\ref{thm:Taylor-intro}. It is accurate enough to determine the twisting index by comparing it to the polygons in Figure~\ref{fig:polygons-with-kappa-1}, and we will see that the polygon with $\kappa_r^\Delta=0$ for both $r=1$ and $r=2$ is the representative shown in Figure~\ref{fig:twis-cuts-down}.
The twisting index of other polygons is then determined by the
action given in Equation \eqref{eq:twisact}, as shown in Figures~\ref{fig:polygons-with-kappa-1} and~\ref{fig:polygons-with-kappa-2}.

\begin{figure}[ht]
\includegraphics[width=250pt]{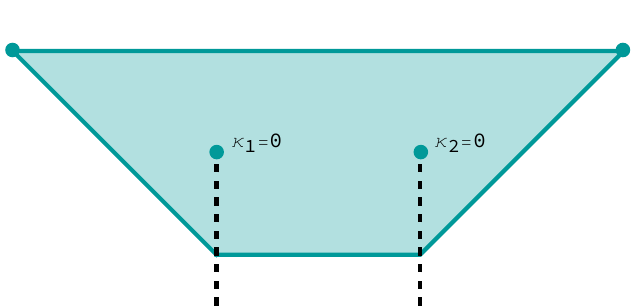}
\caption{A representative of the semitoric polygon. In the proof of Theorem~\ref{thm:twist-intro} we prove that the associated twisting index labels are $(\kappa_1,\kappa_2)=(0,0)$. The indices of the other polygons are then found by applying the action defined in Equation \eqref{eq:twisact}, see Figures~\ref{fig:polygons-with-kappa-1} and~\ref{fig:polygons-with-kappa-2}.
}
\label{fig:twis-cuts-down}
\end{figure}

\begin{figure}[ht]
\centering
    \begin{subfigure}[b]{0.38\textwidth}
    	\centering
        \includegraphics[width=\textwidth]{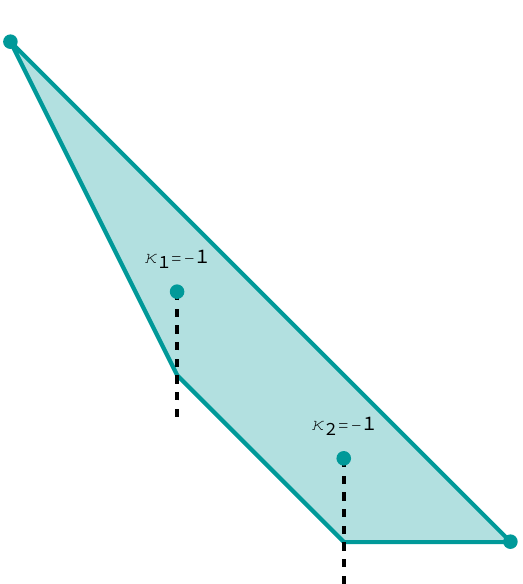}
        \caption{\small{
        Application of $T^{-1}$
        }}
    \end{subfigure}\hspace{-1.6cm}
    \begin{subfigure}[b]{0.42\textwidth}
    	\centering
        \includegraphics[width=\textwidth]{fig13_2_n}
        	\vspace{0.8cm}
        \caption{\small{
        Application of $T^0$
        }}
    \end{subfigure}\hspace{-1.6cm}
    \begin{subfigure}[b]{0.38\textwidth}
    	\centering
        \includegraphics[width=\textwidth]{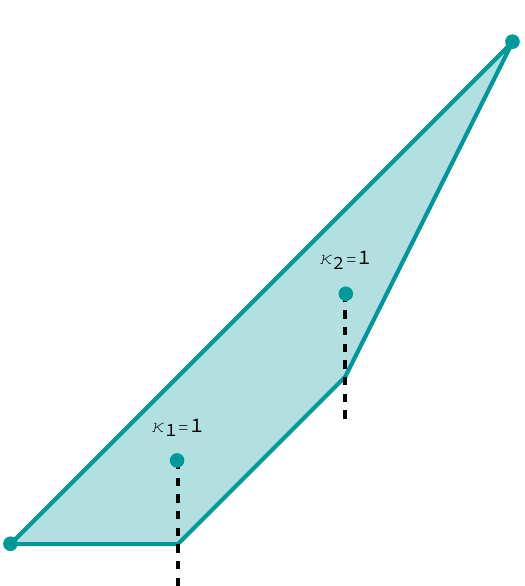}
        \caption{\small{
        Application of $T^1$
        }}
    \end{subfigure}\\
    \caption{Polygons with $\varepsilon_1 = +1$, $\varepsilon_2 = +1$ obtained by applying powers of $T$ to the polygon from Figure~\ref{fig:twis-cuts-down} with the twisting indices labeled.}
    \label{fig:polygons-with-kappa-1}
\end{figure}

\begin{remark}\label{rmk:zeros}
Note that in our case there is a single representative of the polygon invariant which turns out to have both $\kappa_1=0$ and $\kappa_2=0$ (the one shown in Figure~\ref{fig:twis-cuts-down}). This is just a coincidence for this system. In general the representatives of the polygon invariant with $\kappa_1=0$ and $\kappa_2=0$ can be distinct.
\end{remark}

\begin{remark}
If there is only one focus-focus point the situation is much easier, see Alonso \& Dullin \& Hohloch \cite{ADH2}. There the twisting index is deduced by comparing \cite[Figure 13]{ADH2} with \cite[Figure 14]{ADH2}, which correspond to our Figures~\ref{fig:twis-cuts-down} and~\ref{fig:twisComp}, respectively.
\end{remark}

\begin{proof}[Proof of Theorem~\ref{thm:twist-intro}]
Let $s_-,s_+$ be the lower and upper bounds for the range of parameters $s$ for which the system has two focus-focus points, as given in Proposition~\ref{prop:nff}.
Let $\rho_2 = \rho_2(s)$ be as given in Equation~\eqref{eqn:rho} and let $s\in \,\,]s_-,s_+[$.
Consider the first order factor of $l$ in the Taylor series $S_r^\infty(l,j)$ from Theorem~\ref{thm:Taylor-intro}, which is given by
\begin{equation}\label{eqn:arctan}
\arctan \left( \dfrac{6-9s}{\ra} \right)\quad  \text{for }r=1,\qquad\quad \arctan \left( \dfrac{6-9s}{\ra} \right) + \pi \quad \text{for }r=2.
\end{equation}
No matter the value of $s$, the values of the terms in Equation~\eqref{eqn:arctan} stay in the interval $]-\tfrac{\pi}{2},\tfrac{3\pi}{2}[$. Since the twisting index only changes when this term surpasses either $-\tfrac{\pi}{2}$ or $\tfrac{3\pi}{2}$, we conclude that the twisting index invariant at both focus-focus points $m_1$ and $m_2$ is independent of the value of $s$.
For the remainder of the proof we thus suppress all dependencies on $s$ in the notation.

To obtain the twisting index invariant we need to compare the privileged momentum maps $\nu_1$ and $\nu_2$ associated to each of the focus-focus singularities with the momentum map $\mu_\Delta$ associated to a polygon $\De$ of the polygon invariant from Theorem \ref{thm:poly}.

Following Section \ref{sss:polygon}, we make the choice of signs $\varepsilon = (\varepsilon_1, \varepsilon_2) = (-1,-1)$ and we consider the half-lines
$$
b_{\lam_1}^{\epsilon_1} = \{ (0,y) \;|\; y<0\}\quad \mbox{and} \quad b_{\lam_2}^{\epsilon_2} = \{ (2,y) \;|\; y<0 \} 
$$ where $\lambda_1=0$, $\lambda_2=2$. For each $r\in\{1,2\}$, we have seen in Section \ref{sss:taylor} that we can find a local neighborhood $U_r \subset M$ of the focus-focus singularity $m_r$, a neighborhood $W_r := F^{-1} F(U)$ of the corresponding focus-focus fiber and $\tilde{W}_r = F^{-1}(F(W_r) \backslash b_{\lambda_r}^{\varepsilon_r})$. We also have a map $\Phi_r = \varrho_r \circ F$ and we use a local coordinate $z = l + ij \in \Phi_r(W_r) \subset \C$. In Section \ref{sss:twisting} we have defined the privileged momentum map $\nu_r = (L,\Xi_r)$, where $\Xi_r = \xi_r \circ \Phi : W_r \to \R$ and the function $\xi_r(z)$ is understood as a preferred local action (cf.\ Remark \ref{re:xi_as_action}). In Lemma \ref{lem:Spref} we have defined its corresponding desingularised action $\hat{S}$, which is related to $\xi_r$ by
\begin{equation}\label{eqn:proof-xi}
2\pi \xi_r(z)  = \hat{S}_r(z) - \textup{Im}(z \log z -z) + 2\pi \xi_r(0)
\end{equation}
for $z\in\Phi_r(W_r)$. In Theorem \ref{thm:Taylor-intro} we have computed the Taylor series invariant up to second order, which consists of the Taylor series of the functions $\hat{S}_r(z)$, $r\in\{1,2\}$. We will use this result to approximate the functions $\xi_r(z)$.

\begin{figure}[ht]
\centering
    \begin{subfigure}[b]{0.48\textwidth}
    	\centering
        \includegraphics[width=0.9\textwidth]{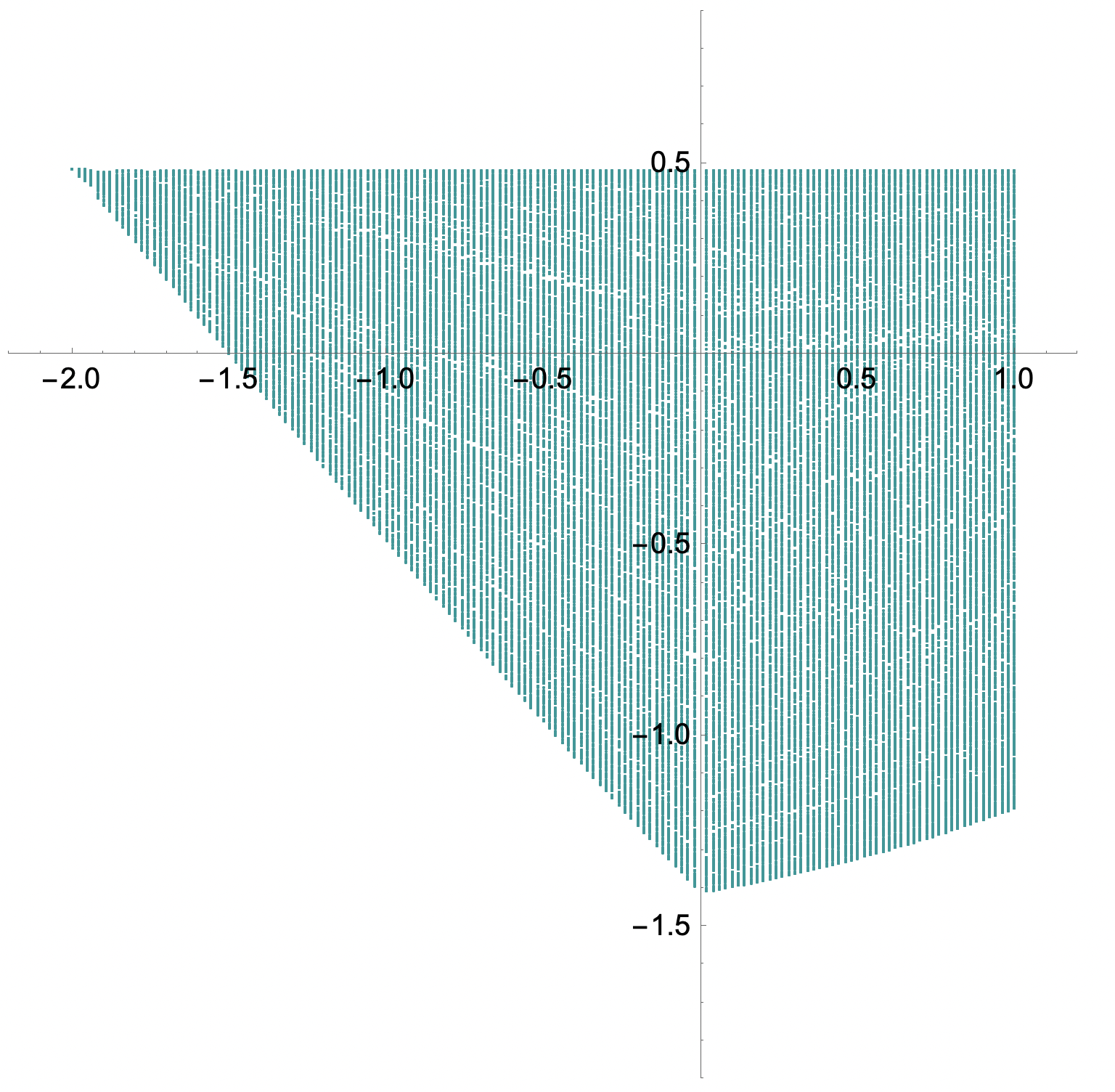}
        \caption{\small{Singularity $\mathcal{N} \times \mathcal{S}$}}
    \end{subfigure}
    \begin{subfigure}[b]{0.48\textwidth}
    	\centering
        \includegraphics[width=0.9\textwidth]{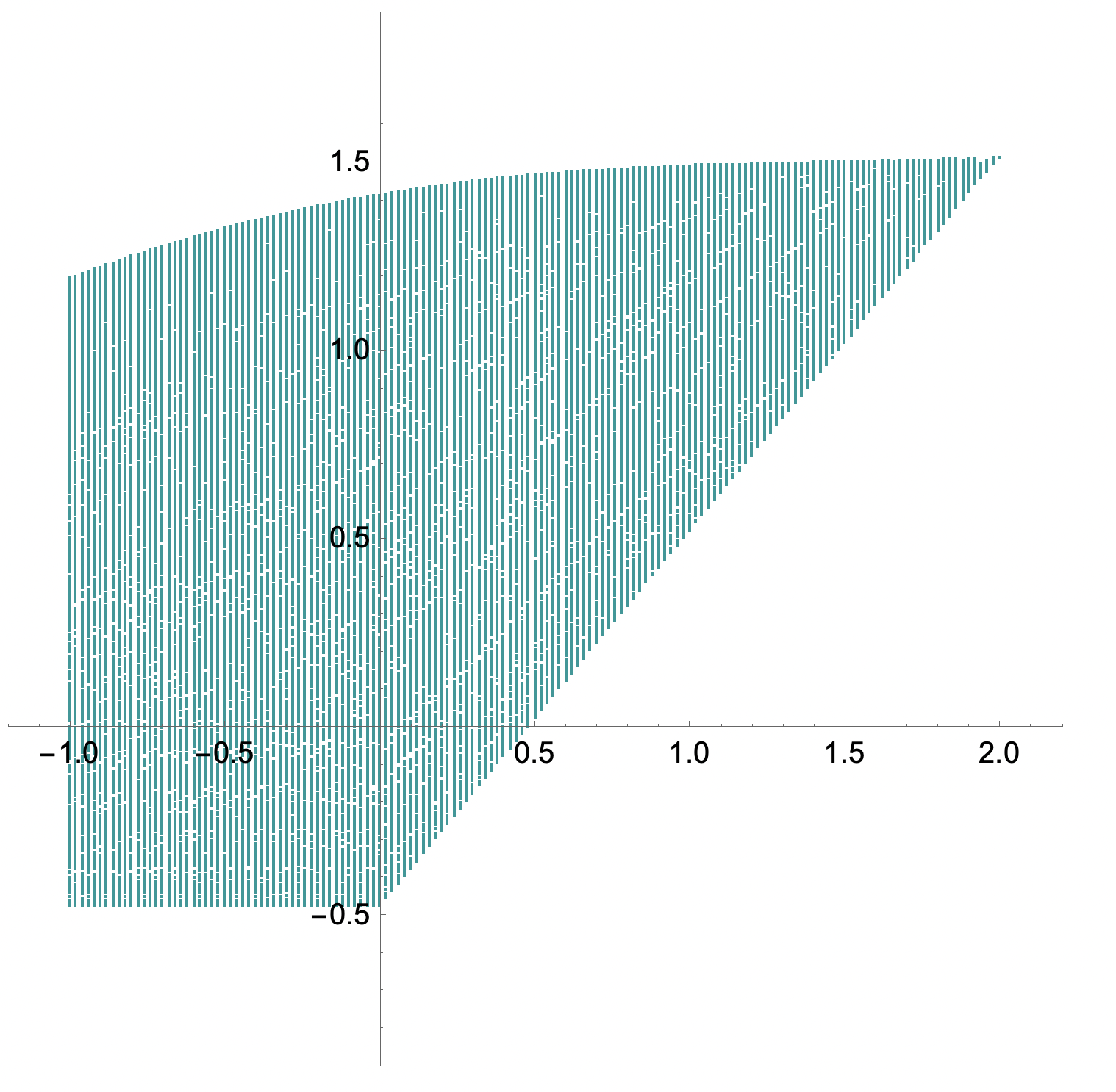}
        \caption{\small{Singularity $\mathcal{S} \times \mathcal{N}$}}
    \end{subfigure} 
\caption{The image of the extensions of the local privileged momentum maps $\nu_1 = (L,\Xi_1)$ and $\nu_2=(L,\Xi_2)$ are shown for the focus-focus singularity $\mathcal{N} \times \mathcal{S}$ in subfigure (a) and $\mathcal{S} \times \mathcal{N}$ in subfigure (b).
Here we have taken the parameter value $s=\tfrac{1}{2}$. 
We used Mathematica to find the images of $\nu_1$ and $\nu_2$ by applying Equations~\eqref{eq:LH} and~\eqref{eq:approxtwis} to 386,560 points
$(q_1,p_1,q_2,p_2)$ in the set given by $-\pi < q_1 <\pi$ and $\,-3 < p_1 < 0$ and $\,0 < q_2 < \pi$ and $\,0 < p_2 <4$.}
\label{fig:twisComp}
\end{figure}

In Remark \ref{rmk:varrho_is_J}, we have seen that $\varrho_r(l,h) = (l,J_r(l,h)),$ where $J_r(l,h)$ is the imaginary action defined in \eqref{eqn:im-action}.
For $x\in W_r$, we thus have that
\begin{equation}\label{eq:approxtwis}
\Xi_r (x) = \xi_r \circ \varrho_r \circ F(x) = \xi_r(L(x), J_r(L(x),H(x))).
\end{equation} 
Note that $J_r$ is smooth around the focus-focus value and can thus be expanded around it, as in the proof of Lemma \ref{lem:birkhoff}.
Furthermore, $\xi_r$ is determined by $\hat{S}$ via Equation~\eqref{eqn:proof-xi}, and $\hat{S}$ can be approximated by the finite expansion from Theorem \ref{thm:Taylor-intro}.
Putting this together, from Equation~\eqref{eq:approxtwis} and Theorem \ref{thm:Taylor-intro} we have obtained an approximation for $\Xi_r$.

We take $W_1$ to be the focus-focus point $m_1$ and all regular points in the range $-3<L<1$, and for $W_2$ we take $m_2$ and all regular points in the range $-1<L<3$.
Using the approximation of $\Xi_r$ from Equation~\eqref{eq:approxtwis}, we now plot the 
image of the map $\nu_r = (L,\Xi_r)$
using Mathematica, which yields the image shown in Figure~\ref{fig:twisComp}. For simplicity in the computations we have taken $s = \frac{1}{2}$, since as discussed above the resulting value of the twisting index will be independent of $s$. 

Necessarily, for $r\in\{1,2\}$, the map $\nu_r$ is equal in a neighborhood of $c_r$ to one of the possible generalized toric momentum maps $\mu_\Delta$, for some choice of $\De$ with downwards cuts. All such choices are related by a global application of the linear map $T$; several such polygons are shown in Figure~\ref{fig:polygons-with-kappa-1}.

We now compare the polygons from Figure~\ref{fig:polygons-with-kappa-1} with the images of the approximations of
$
 \nu_r = (L,\Xi_r)
$
shown in Figure~\ref{fig:twisComp}.
We see that the order of our approximation is sufficiently high to clearly identify that the images from Figure~\ref{fig:twisComp} are most similar to the polygon singled out in Figure~\ref{fig:twis-cuts-down}, since the other options (such as those shown in Figure~\ref{fig:polygons-with-kappa-1}) have a sufficiently different shape from the one shown in Figure~\ref{fig:twis-cuts-down}. 
Thus, the indices associated to the polygon in Figure~\ref{fig:twis-cuts-down} are $(\kappa_1,\kappa_2)=(0,0)$. This determines the twisting index labels for all representatives of the polygon by applying the group action \eqref{eq:twisact}.
\end{proof}

\newpage

\appendix 
\section{Values of constants}
\label{appendix:d1d2}


The values of the terms $d_1$ and $d_2$ which appear in Equation~\eqref{eqn:dSdj} are given by:
\begin{equation*}
\begin{aligned}
d_1 &= -\left(16 - 96 (1 + \rc) {s} + 8 (139 + 48 \rc) {s}^2 - 
    8 (587 + 183 \rc) {s}^3 \right. \\& \quad \left. + 3 (3515 + 1032 \rc) {s}^4 - 
    64 (241 + 45 \rc) {s}^5 + 96 (157 + 10 \rc) {s}^6 \right. \\& \quad \left.- 8704 {s}^7 + 
    2176 {s}^8\right) \left(-6 j l \ra (-512 + 6912 {s} - 32256 {s}^2 + 29952 {s}^3 \right. \\& \quad \left.+ 
      651072 {s}^4 - 5470176 {s}^5 + 25000480 {s}^6 - 78708528 {s}^7 + 
      181951998 {s}^8 \right. \\& \quad \left.- 308836805 {s}^9 + 366523680 {s}^{10} - 
      264177144 {s}^{11} + 45690784 {s}^{12} + 122026416 {s}^{13} \right. \\& \quad \left.- 
      141477888 {s}^{14} + 75192576 s2^{15} - 20766720 {s}^{16} + 
      2396160 {s}^{17}) \right. \\& \quad \left.+ 
   l^2 (5120 - 76800 {s} + 536832 {s}^2 - 2331648 {s}^3 + 9224832 {s}^4 - 
      43900032 {s}^5 \right. \\& \quad \left.+ 196060832 {s}^6 - 664211520 {s}^7 + 
      1701591876 {s}^8 - 3591176044 {s}^9 \right. \\& \quad \left.+ 6833670885 {s}^{10} - 
      11790448080 {s}^{11} + 17057911016 {s}^{12} - 19014926976 {s}^{13}\right. \\& \quad \left. + 
      15302126928 {s}^{14} - 8262067200 {s}^{15} + 2541006336 {s}^{16} - 
      132489216 {s}^{17} \right. \\& \quad \left.- 200669184 {s}^{18} + 66846720 {s}^{19} - 
      6684672 {s}^{20}) \right. \\& \quad \left.+ 
   j^2 (-5120 + 76800 {s} + 118528 {s}^2 - 6843392 {s}^3 + 
      69320064 {s}^4 - 441754496 {s}^5 \right. \\& \quad \left.+ 1994226016 {s}^6 - 
      6694907840 {s}^7 + 17301290172 {s}^8 - 35191498900 {s}^9 \right. \\& \quad \left.+ 
      57016666395 {s}^{10} - 74143143472 {s}^{11} + 78236327192 {s}^{12} - 
      68381724032 {s}^{13} \right. \\& \quad \left.+ 50940505264 {s}^{14} - 32953613312 {s}^{15} + 
      18193326592 {s}^{16} - 8071077888 {s}^{17} \right. \\& \quad \left.+ 2625804288 {s}^{18} - 
      547880960 {s}^{19} + 54788096 {s}^{20})\right)\\[0.2cm]      
d_2 &= 256 \ra^6 \rc^4 (16 - 96 (1 + \rc) {s} + 8 (139 + 48 \rc) {s}^2 - 
   8 (587 + 183 \rc) {s}^3 \\ & \quad + 3 (3515 + 1032 \rc) {s}^4 - 
   64 (241 + 45 \rc) {s}^5 + 96 (157 + 10 \rc) {s}^6 \\& \quad- 8704 {s}^7 + 
   2176 {s}^8).
\end{aligned}
\label{d1d2}
\end{equation*}
Note that $\frac{d_1}{d_2} = d^s_3j^2 + d^s_4 jl + d^s_5l^2$ for some $d^s_3, d^s_4, d^s_5\in\R$ depending on the parameter $s$, since $\rho_1$ and $\rho_2$ are also determined by $s$, see Equation~\eqref{eqn:rho}.

\bibliography{twistingindexref}  
\bibliographystyle{alpha}

  \noindent
  \textbf{Jaume Alonso}\\
  Technische Universität Berlin\\
  Institute of Mathematics\\
  Str.\ des 17. Juni 136\\
  D-10623 Berlin, Germany\\
  {\em E\--mail}: \texttt{alonso@math.tu-berlin.de}\\

  \noindent
  \textbf{Sonja Hohloch}\\
  University of Antwerp\\
  Department of Mathematics\\
  Middelheimlaan 1\\
  B-2020 Antwerpen, Belgium\\
  {\em E\--mail}: \texttt{sonja.hohloch@uantwerpen.be}\\

  \noindent
  \textbf{Joseph Palmer}\\
  University of Illinois
  at Urbana-Champaign\\
  Department of Mathematics\\
  1409 W. Green St\\
  Urbana, IL 61801 USA\\
  {\em E\--mail}: \texttt{jpalmer5@illinois.edu}\\

\end{document}